\documentclass[aop]{imsart}
\usepackage[utf8x]{inputenc}
\usepackage{ucs}
\usepackage{amsmath}
\usepackage{amsfonts}
\usepackage{amssymb}
\usepackage{enumitem}
\usepackage{cancel}
\usepackage{amsbsy}
\usepackage{genyoungtabtikz}
\usepackage{stmaryrd}

\usepackage{abraces}
\usepackage{amscd} 
\usepackage{amsgen} 
\usepackage{extdash}
\usepackage{amsopn} 
\usepackage{amstext}
\usepackage{amsthm} 
\usepackage{amsxtra}
\usepackage{hyperref}
\usepackage[capitalise,noabbrev]{cleveref}
\usepackage{rotating}
\usepackage{tikz}
\newtheorem{thm}{Theorem}[section]
\newtheorem{lem}[thm]{Lemma}
\newtheorem{prop}[thm]{Proposition}
\newtheorem{defn}[thm]{Definition}
\newtheorem{cor}[thm]{Corollary}

\newtheorem*{lem*}{Lemma}
\newtheorem*{prop*}{Proposition}
\newtheorem*{thm*}{Theorem}
\newtheorem*{defn*}{Definition}
\newtheorem*{cor*}{Corollary}

\DeclareMathOperator{\Id}{Id}

\DeclareMathOperator{\Leb}{Leb}

\DeclareMathOperator{\law}{law}

\DeclareMathOperator{\Res}{Res}
\begin{document}
\begin{frontmatter}

\title{Zigzag diagrams and Martin boundary}
\runtitle{Zigzag diagrams and Martin Boundary}

\begin{aug}
  \author{\fnms{Pierre}  \snm{ Tarrago}\corref{}\thanksref{t2}\ead[label=e1]{pierre.tarrago@cimat.mx}} 

  \thankstext{t2}{Supported by the Deutsch-Franz\"{o}siche Hochschule}

  \runauthor{P. Tarrago}

  \affiliation{Centro de Investigación en Matemáticas (CIMAT)}

  \address{De Jalisco 8A, Valenciana, 36240 Guanajuato, Gto.,\\ 
          \printead{e1}}

\end{aug}

\begin{abstract}
\indent We investigate the asymptotic behavior of random paths on a graded graph which describes the subword order for words in two letters. This graph, denoted by $\mathcal{Z}$, has been introduced by Viennot, who also discovered a remarkable bijection between paths on $\mathcal{Z}$ and sequences of permutations. Later on, Gnedin and Olshanski used this bijection to describe the set of Gibbs measures on this graph. Both authors also conjectured that the Martin boundary of $\mathcal{Z}$ should coincide with its minimal boundary. We give here a proof of this conjecture by describing the distribution of a large random path conditioned on having a prescribed endpoint. We also relate paths on the graph $\mathcal{Z}$ with paths on the Young lattice, and we finally give a central limit theorem for the Plancherel measure on the set of paths in $\mathcal{Z}$.
\end{abstract}

\begin{keyword}[class=MSC]
\kwd[Primary ]{60C05}
\kwd{60J45}
\kwd[; secondary ]{05E99}
\end{keyword}

\begin{keyword}
\kwd{Descent set of a permutation}
\kwd{Martin boundary}
\kwd{Compositions}
\end{keyword}

\end{frontmatter}
\begin{abstract}
\indent We investigate the asymptotic behavior of random paths on a graded graph which describes the subword order for words in two letters. This graph, denoted by $\mathcal{Z}$, has been introduced by Viennot, who also discovered a remarkable bijection between paths on $\mathcal{Z}$ and sequences of permutations. Later on, Gnedin and Olshanski used this bijection to describe the set of Gibbs measures on the graph. Both authors also conjectured that the Martin boundary of $\mathcal{Z}$ should coincide with its minimal boundary. We give here a proof of this conjecture by describing the distribution of a large random path conditioned on having a prescribed endpoint. We also relate paths on the graph $\mathcal{Z}$ with paths on the Young lattice, and we finally give a central limit theorem for the Plancherel measure on the set of paths in $\mathcal{Z}$.
\end{abstract}
\maketitle

\section*{Introduction} 
\indent Let $A$ be a finite alphabet. A word $w$ in $A$ is an element of the free monoid $A^{*}$ generated by $A$: $w$ is uniquely written as $w_{1}w_{2}\cdots w_{n}$ with $n\geq 0$ and $w_{i}\in A$ for each $1\leq i\leq n$. A word $w'$ is called a subword of $w$ and is written as $w'\prec w$ when $w'$ is equal to $w_{i_{1}}\cdots w_{i_{r}}$ for some $1\leq i_{1}<\ldots<i_{r}\leq n$; the relation $\prec$ is an order relation on the set of words in $A$ which is called the subword relation. Following the subword relation, $A^{*}$ can be turned into a graded graph: the grading is given by the length of the words, and an edge is given between two words $w,w'$ of consecutive grading if and only if $w\prec w'$. Such a graph is rooted by abstractly adding a vertex $*$ and an edge between $*$ and the empty word $\emptyset$. The simplest non-trivial case is given by the alphabet $A_{2}$ of cardinality $2$ whose elements are written as $+$ and $-$: the corresponding graph, which is denoted by $\mathcal{Z}$ in this paper, has been introduced and first studied by Viennot in \cite{viennot1983maximal}. The first part of $\mathcal{Z}$ is displayed in \cref{graphZ}.

\indent From a purely algebraic point of view, the graph $\mathcal{Z}$ encodes the inductions and restrictions for the simple modules of the tower of Hecke algebras $\lbrace H_{n}(0)\rbrace_{n\geq 1}$ exactly like the Young graph for the tower of algebras $\lbrace \mathbb{C}S_{n}\rbrace_{n\geq 1}$, where $S_{n}$ denotes the symmetric group of order $n$ (see for example \cite{duchamp2002noncommutative}).
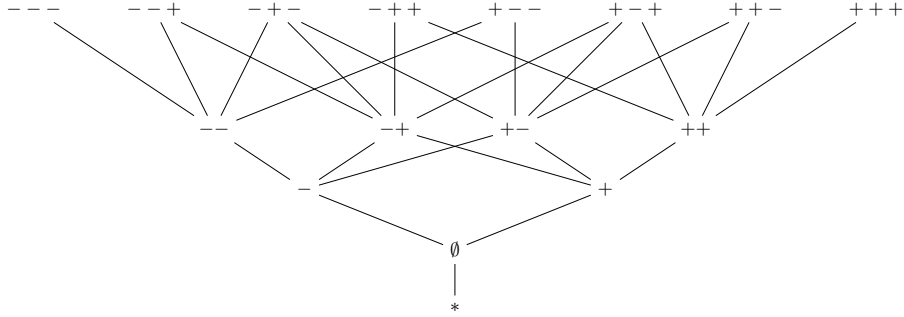
\begin{figure}\scalebox{0.8}{
\begin{tikzpicture}
\node (0) at (5,0){*};
\node (11) at (5,1){$\emptyset$};
\node (21) at (7.5,2){$+$};
\node (22) at (2.5,2){$-$};
\node (31) at (9,3){$++$};
\node (32) at (6,3){$+-$};
\node (33) at (4,3){$-+$};
\node (34) at (1,3){$--$};
\node (41) at (12,5){$+++$};
\node (42) at (10,5){$++-$};
\node (43) at (8,5){$+-+$};
\node (44) at (6,5){$+--$};
\node (45) at (4,5){$-++$};
\node (46) at (2,5){$-+-$};
\node (47) at (0,5){$--+$};
\node (48) at (-2,5){$---$};

\draw (0) to (11); 
\draw (11) to (21);
\draw (11) to (22);
\draw (21) to (31);
\draw (21) to (32);
\draw (21) to (33);
\draw (22) to (32);
\draw (22) to (33);
\draw (22) to (34); 
\draw (31) to (41);
\draw (31) to (42);
\draw (31) to (43);
\draw (31) to (45);
\draw (32) to (42);
\draw (32) to (43);
\draw (32) to (44);
\draw (32) to (46);
\draw (33) to (43);
\draw (33) to (45);
\draw (33) to (46);
\draw (33) to (47);
\draw (34) to (44);
\draw (34) to (46);
\draw (34) to (47);
\draw (34) to (48);

\end{tikzpicture}}
\caption{\label{graphZ}First levels of $\mathcal{Z}$.}

\end{figure}
Words in the alphabet $A_{2}$ play a particular role in the combinatorics of permutations. For $1\leq i\leq n-1$, we say that a permutation $\sigma$ of $n$, written as a word $\sigma(1)\sigma(2)\cdots\sigma(n)$, has a descent at $i$ when $\sigma(i)>\sigma(i+1)$. Using this definition, we can map each permutation $\sigma$ to a word $w(\sigma)=w_{1}\cdots w_{n-1}$ in $A_{2}^{*}$ by saying that $w_{i}=-$ if and only if $i$ is a descent of $\sigma$: in this way, $w(\sigma)$ describes exactly the set of descents of $\sigma$ and is called the descent word of $\sigma$. The description of the set of permutations having a given descent word is an important problem in combinatorics (see in particular the works of Viennot \cite{viennot1979permutations,viennot1981equidistribution}, Niven \cite{niven1968combinatorial} and de Bruijn \cite{de1970permutations}); for example, permutations whose descent words are the alternating words $+-+\cdots+-$ or $+-+\cdots -$ are called alternating permutations, and have been studied by Niven and de Bruijn in \cite{niven1968combinatorial} and \cite{de1970permutations}. Désiré André gave already in 1881 (see \cite{andre1881permutations}) an expression for the number of alternating permutations of $n$ and showed that this number is asymptotically $2(2/\pi)^{n}n!$ as $n$ grows to infinity. Several refinements of this result for other similar words have been found since (\cite{bender2004asymptotics,ehrenborg2002probabilistic}).

\indent In this paper, we use the link between words in $A_{2}^{*}$ and descents of permutations to study random paths on the graph $\mathcal{Z}$: in the sequel, a path on $\mathcal{Z}$ is always a (possibly infinite) path starting at the root and such that, at each step, the rank of the visited vertex is increased by one. As it has been suggested before, the graph $\mathcal{Z}$ is at the interface between different subjects in combinatorics and probability. Namely, paths on $\mathcal{Z}$ of length $n$ ending at a vertex $w$ are in bijection with permutations $\sigma$ of $n$ such that $w(\sigma)=w$ (see \cite{gnedin2006coherent,viennot1983maximal}). Moreover, infinite paths on $\mathcal{Z}$ correspond to infinite sequences of permutations  $(\sigma_{1},\sigma_{2},\ldots)$ which satisfy a particular coherence property: each $\sigma_{n}$ is obtained from $\sigma_{n+1}$ by deleting the $n+1$ symbol (in the word description of permutations). An example of such correspondence is given in \cref{equivPathPerm}.

\begin{figure}[h!]\scalebox{0.75}{
\begin{tikzpicture}
\node (0) at (0,3){$*$};
\node (1) at (3,3){$\emptyset$};
\node (2) at (6,3){$-$};
\node (3) at (9,3){$+-$};
\node (4) at (12,3){$-+-$};
\node (5) at (15,3){$-+--$};

\node (0') at (0,0){$*$};
\node (1') at (3,0){$(\mathbf{1})$};
\node (2') at (6,0){$(\mathbf{2}1)$};
\node (3') at (9,0){$(2\mathbf{3}1)$};
\node (4') at (12,0){$(\mathbf{4}231)$};
\node (5') at (15,0){$(42\mathbf{5}31)$};

\draw [->](1,3)--(2,3);
\draw [->](4,3)--(5,3);
\draw [->](7,3)--(8,3);
\draw [->](10,3)--(11,3);
\draw [->](13,3)--(14,3);

\draw [->](1,0)--(2,0);
\draw [->](4,0)--(5,0);
\draw [->](7,0)--(8,0);
\draw [->](10,0)--(11,0);
\draw [->](13,0)--(14,0);

\draw [->, very thick] (7.5,2)--(7.5,1);
\end{tikzpicture}}
\caption{\label{equivPathPerm}An example of the equivalence between paths on $\mathcal{Z}$ and sequences of coherent permutations.}
\end{figure}
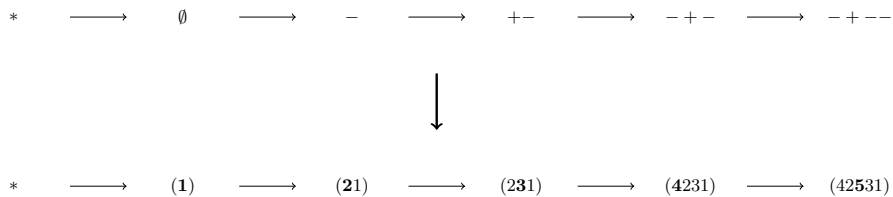
\indent A measure on the set of paths on a graded graph is called harmonic if the following property holds for the corresponding random path: conditioned on the event that a path goes through a vertex $\mu$ of rank $n$, the distribution of the first $n$ steps of the random path is the uniform distribution on the set of all paths between the root and $\mu$. A harmonic measure which is supported on the set of finite paths is simply called finite: a finite harmonic measure admits a straightforward description as a convex combination of uniform measures. If the harmonic measure is supported on the set of infinite paths, then this measure is called a Gibbs measure (see for example \cite{YoungBouquet}). Gibbs measures are in general much more difficult to describe than finite harmonic measures. An important result from the Martin boundary theory yields that any Gibbs measure can be approximated by finite harmonic measures (see \cref{MartEntrBoundary} for a precise statement).

\indent In the case of the graph $\mathcal{Z}$, the description of the set of Gibbs measures has been achieved by Gnedin and Olshanski in \cite{gnedin2006coherent} by using a link between the graph structure of $\mathcal{Z}$ and the ring $QSym$ of quasi-symmetric functions. The latter is a refinement of the ring $Sym$ of symmetric functions, and homogeneous bases of $QSym$ are indexed by words in $\lbrace +,-\rbrace$; then, the graph $\mathcal{Z}$ encodes the multiplication in a particular basis of $QSym$ called the fundamental basis. Thanks to this relation with $QSym$, Gnedin and Olshanski also established a relationship between Gibbs measures on $\mathcal{Z}$ and Gibbs measures on other graphs: the Young graph $\mathcal{Y}$ introduced by Vershik and Kerov, Kingman's partition graph $\mathcal{P}$ \cite{kingman1978representation} and the composition graph $\mathcal{C}$ introduced by Gnedin \cite{gnedin1997representation}. The Young graph is described in \cref{RelYouGra,presentationYoungGraph} and we refer to \cite[Section 1]{gnedin2006coherent} for a description of the graphs $\mathcal{P}$ and $\mathcal{C}$. In particular, Gnedin and Olshanski introduced a natural harmonic measure on $\mathcal{Z}$ which yields the Plancherel measure on the Young graph: this measure, also called the Plancherel measure, is roughly speaking the uniform measure on the set of infinite paths on $\mathcal{Z}$.

\indent  Whereas the relation between $\mathcal{Z}$ and $\mathcal{Y}$ has only be proven at the level of the harmonic measures, the one between $\mathcal{Y},\mathcal{P}$ and $\mathcal{C}$ has been shown to hold also at the level of paths. On the one hand, the set of paths on $\mathcal{P}$ is the same as the set of paths on $\mathcal{Y}$, but with different weights (see \cite{kerov1998boundary} for the description of the weights for $\mathcal{P}$ and for a natural interpolation between $\mathcal{Y}$ and $\mathcal{P}$), and the same holds for $\mathcal{C}$ and $\mathcal{Z}$. On the other hand, the set of paths on $\mathcal{C}$ is in bijection with a particular set of sequences of permutations, and considering the cycle structure of these permutations yields a surjective map from the paths on $\mathcal{C}$ to the ones on $\mathcal{P}$ (see \cite[Section 1]{gnedin2006coherent} and \cite{GnedinRegeneration} for a detailed description of this map). 
 
\indent In this paper, we present three results on the graph $\mathcal{Z}$ which continue the study of Gnedin and Olshanski:
\begin{itemize}
\item we describe how finite harmonic measures approximate Gibbs measures: we show that the convergence in law of a sequence of finite harmonic measures towards a Gibbs measure can be parametrized by a simple topological space. As a consequence, we prove that the classification of harmonic measures by Gnedin and Olshanski has a geometrical meaning in terms of a particular topology on $\mathcal{Z}$: from a harmonic analysis point of view, this yields a description of the Martin boundary of the graph $\mathcal{Z}$ (which had been already conjectured in \cite{gnedin2006coherent}) and an identification of the Martin boundary of $\mathcal{Z}$ with its minimal boundary;
\item we provide a correspondence between paths on $\mathcal{Z}$ and paths on $\mathcal{Y}$: this correspondence uses the Robinson-Schensted-Knuth (RSK) algorithm, which establishes a bijection between permutations of $n$ and pairs of standard Young tableaux of size $n$ and same shape. This completes the description of the relations between paths on $\mathcal{Y},\mathcal{P},\mathcal{C}$ and $\mathcal{Z}$;
\item we provide a law of large numbers for the asymptotic behavior of a random path following a Gibbs measure. We describe also the fluctuations for the Plancherel measure. 
\end{itemize}
\indent An introduction to the concepts of graded graph, harmonic measure and Martin compactification is given in \cref{SectionProbaGradedGraph}. \cref{sectionSummary} is devoted to the presentation of the graph $\mathcal{Z}$ and a precise statement of the results. Along the proofs, it is often more convenient to use compositions, which are just an alternative description of the words in $\lbrace +,-\rbrace$: compositions are also introduced in \cref{sectionSummary}. \cref{sectionOrientedPaintbox} details the main result of \cite{gnedin2006coherent}, which provides a description of the set of Gibbs measure on $\mathcal{Z}$ thanks to the so-called oriented paintbox construction. \crefrange{SectionintroFamilyXi}{SectionmartinBoundary} are the successive steps leading to the description of the Martin boundary: a sketch of the method is given at the end of \cref{sectionOrientedPaintbox}. \cref{SectionRelationYong} deals with the path correspondence between $\mathcal{Y}$ and $\mathcal{Z}$, and \cref{SectionLawLargeNumber} shows the law of large number and central limit theorem for the Plancherel measure.

\section{Probability on graded graphs}\label{SectionProbaGradedGraph}
In this section, we introduce various notations and concepts related to graded graphs; in particular, we explain the notion of harmonic measures and Martin entrance boundary for such graphs. The content of this section mostly follows \cite[Section 2]{kerov1998boundary}.
\subsection{Graded graph}
The notations used here are from \cite{stanley2011enumerative}. A rooted, graded graph $\mathcal{G}$ is a triple $(V,\rho,E)$, where
\begin{itemize}
\item $V$ is a denumerable set of vertices with a distinguished element $\mu_{0}$,
\item $\rho:V\rightarrow \mathbb{N}$ is a map satisfying $\rho^{-1}(\lbrace 0\rbrace)=\lbrace \mu_{0}\rbrace$ and $\rho^{-1}(\lbrace n\rbrace)$ finite for all $n\geq 0$, and
\item the adjacency matrix E is a $V\times V$-matrix with entries in $\mathbb{R}^{+}$, such that $E(\mu,\nu)$ is zero if $\rho(\nu)\not=\rho(\mu)+1$.
\end{itemize}
The map $\rho$ is called the rank map of the graded graph; the set $\rho^{-1}(\lbrace n\rbrace)$ is called the $n-$th level set of $\mathcal{G}$ and denoted by $V_{n}$. We write $\mu\nearrow \nu$ if $E(\mu,\nu)>0$. A path $\gamma$ on $\mathcal{G}$ is a sequence of vertices $(\gamma_{0},\gamma_{1},\ldots)$ of increasing rank such that for all $i\geq 0$, $\gamma_{i}\nearrow\gamma_{i+1}$. A path is said finite (resp.~infinite) if the corresponding sequence of vertices is finite (resp.~infinite): in the finite case, the length of the path, denoted by $l(\gamma)$, is the number of edges which are crossed along the path, which is also the number of vertices minus one; in the infinite case, we simply set $l(\gamma)=\infty$. The set of all paths on $\mathcal{G}$ starting at the root is denoted by $\Gamma(\mathcal{G})$: this set is partitioned into two sets, the set $\Gamma_{f}(\mathcal{G})$ of finite paths and the set $\Gamma_{\infty}(\mathcal{G})$ of infinite paths.

If $\gamma$ is a path on $\mathcal{G}$ and $k\leq l(\gamma)$ is an integer, $\gamma_{k}$ always denotes the position of $\gamma$ after $k$ steps, whereas $\gamma_{\downarrow k}$ denotes the path $\gamma$ restricted to its $k$ first steps. The weight $w(\gamma)$ of a finite path $\gamma=(\gamma_{0},\ldots,\gamma_{n})$ is defined as the product $\prod_{i=0}^{n-1}E(\gamma_{i},\gamma_{i+1})$.
Let $\mu,\nu$ be two vertices such that $\rho(\mu)\leq \rho(\nu)$. The set of paths from $\mu$ to $\nu$ is denoted by $\Gamma(\mu,\nu)$; we define a path counting function $d:V\times V\rightarrow \mathbb{R}$ by the sum of of the path weights:
$$d(\mu,\nu)=\sum_{\gamma\in \Gamma(\mu,\nu)} w(\gamma).$$
When all the entries of the adjacency matrix are in $\lbrace 0,1\rbrace$, $d(\mu,\nu)$ is equal to the cardinality of $\Gamma(\mu,\nu)$. When $\mu=\mu_{0}$, we simply write $\Gamma(\nu)$ and $d(\nu)$ respectively instead of $\Gamma(\mu_{0},\nu)$ and $d(\mu_{0},\nu)$.

\subsection{Harmonic measures on graded graphs}\label{harmonicmeasuresGradedgraphs}
This paragraph is an introduction to the concept of harmonic measures on a graded graph $\mathcal{G}$. The terminology may vary in the literature, and infinite harmonic measures are also called Gibbs measures (see \cite{YoungBouquet}). We should first precise the $\sigma-$algebra on the set $\Gamma(\mathcal{G})=\Gamma_{f}(\mathcal{G})\sqcup\Gamma_{\infty}(\mathcal{G})$: we consider the $\sigma-$algebra $\mathcal{A}(\mathcal{G})$, which denotes the coarsest $\sigma-$algebra containing all the sets $\lbrace \gamma\in\Gamma(\mathcal{G}), l(\gamma)\geq k, \gamma_{k}=\mu\rbrace$ for $k\geq 0$ and vertices $\mu$ of rank $k$. By the axioms of a $\sigma-$algebra, $\mathcal{A}_{\mathcal{G}}$ contains each set $\lbrace \gamma \in\Gamma(\mathcal{G}),l(\gamma)=k\rbrace$ and induces by restriction the discrete $\sigma-$algebra on $\Gamma_{f}(\mathcal{G})$. 

For $\tau\in\Gamma_{f}(\mathcal{G})$ with $l(\tau)=k$, the set $\Gamma_{\tau}=\lbrace \gamma\in\Gamma, l(\gamma)\geq k, \gamma_{\downarrow k}=\tau\rbrace$ is thus a measurable set on $\Gamma$. Moreover, by the definition the $\sigma-$algebra $\mathcal{A}(\mathcal{G})$, any probability measure on $\mathcal{A}(\mathcal{G})$ is uniquely defined by its values on the sets $\Gamma_{\tau},\tau \in \Gamma_{f}(\mathcal{G})$.
\begin{defn}\label{harmMeasure}
A probability measure $\mathbb{P}:\mathcal{A}(\mathcal{G})\longrightarrow \mathbb{R}$ is called harmonic when
$$\frac{\mathbb{P}(\Gamma_{\tau})}{w(\tau)}=\frac{\mathbb{P}(\Gamma_{\tau'})}{w(\tau')},$$
for all $\tau,\tau'\in\Gamma_{f}(\mathcal{G})$ such that $l(\tau)=l(\tau')$ and $\tau_{l(\tau)}=\tau'_{l(\tau')}$.
\end{defn}
The set of harmonic measures (resp.~harmonic measures with support in $\Gamma_{f}(\mathcal{G})$, resp.~harmonic measures with support in $\Gamma_{\infty}(\mathcal{G})$) is denoted by $\mathcal{H}(\mathcal{G})$ (resp.~$\mathcal{H}_{f}(\mathcal{G})$, resp.~$\mathcal{H}_{\infty}(\mathcal{G})$). The elements of $\mathcal{H}_{\infty}(\mathcal{G})$ are also called Gibbs measures on $\mathcal{G}$.

It is readily seen that $\mathcal{H}(\mathcal{G})$ is a convex set, and by the application of Choquet theory (see \cite[ Proposition 10.21]{Goodearl}), each element of $\mathcal{H}(\mathcal{G})$ can be uniquely represented as a convex mixture of the extreme points of $\mathcal{H}(\mathcal{G})$. Let us denote by $\partial \mathcal{H}(\mathcal{G})$ the set of extreme points of $\mathcal{H}(\mathcal{G})$. Conditioning harmonic measures on the length of the random path shows that an extreme harmonic measure has to belong either to $\mathcal{H}_{f}(\mathcal{G})$ or to 
$\mathcal{H}_{\infty}(\mathcal{G})$. Therefore,
the description of $\partial\mathcal{H}(\mathcal{G})$ can be split into two cases, by first considering the extreme points in $\mathcal{H}_{f}(\mathcal{G})$, and then the ones in $\mathcal{H}_{\infty}(\mathcal{G})$.

The finite case is much simpler than the infinite one. Let $\mathbb{P}$ be a harmonic measure with support in $\Gamma_{f}(\mathcal{G})$. For any $\nu\in  V$ such that $\mathbb{P}\big(\Gamma(\nu)\big)>0$, conditioning $\mathbb{P}$ on $\Gamma(\nu)$ yields a new harmonic measure $\mathbb{P}_{\nu}$ defined by $\mathbb{P}_{\nu}(\cdot)=\frac{\mathbb{P}(\cdot\cap \Gamma(\nu))}{\mathbb{P}(\Gamma(\nu))}$.
By harmonicity, $\mathbb{P}_{\nu}$ is the unique distribution on $\Gamma(\nu)$ with 
$$\mathbb{P}_{\nu}(\tau)=\frac{w(\tau)}{d(\nu)},\quad \tau\in \Gamma(\nu).$$
Moreover, since the family $\big\lbrace \Gamma(\nu)\big\rbrace_{\nu\in V}$ is a partition of $\Gamma_{f}(\mathcal{G})$, $\mathbb{P}$ is written as the convex combination $\mathbb{P}=\sum_{\nu\in V}\mathbb{P}\big(\Gamma(\nu)\big)\mathbb{P}_{\nu}$.
Therefore, the set of extreme points in $\mathcal{H}_{f}(\mathcal{G})$ is exactly the set $\lbrace \mathbb{P}_{\nu}\rbrace_{\nu\in V}$. In particular, there is a natural bijection between $V$ and $\partial\mathcal{H}(\mathcal{G})\cap \mathcal{H}_{f}(\mathcal{G})$ which sends the vertex $\nu$ to $\mathbb{P}_{\nu}$.
For each $\nu\in V$, denote by $\gamma_{\nu}$ the random path following the measure $\mathbb{P}_{\nu}$. The law of any restriction of $\gamma_{\nu}$ can be described thanks to the Martin kernel of the graph $\mathcal{G}$.
\begin{defn}\label{martinKernelDef}
The Martin kernel of the graph $\mathcal{G}$ is the function 
$$K:\left\lbrace \begin{matrix}
V\times V&\longrightarrow &\mathbb{R}\\
(\mu, \nu)&\mapsto&\frac{d(\mu,\nu)}{d(\nu)}
\end{matrix}\right..$$
\end{defn}
To emphasize the different roles played by $\mu$ and $\nu$, we will write $K_{\mu}(\nu)$ instead of $K(\mu,\nu)$. For any fixed $\mu\in V$, $K_{\mu}(\cdot)$ is a map from $V$ to $\mathbb{R}$ that vanishes on $\lbrace \nu\in V, \Gamma(\mu,\nu)=\emptyset\rbrace$. With this definition, for any vertices $\mu,\nu$ such that  $\rho(\mu)\leq \rho(\nu)$ and for any path $\tau$ ending at $\mu$,
$$\mathbb{P}_{\nu}(\Gamma_{\tau})=w(\tau)K_{\mu}(\nu).$$
In the infinite case, the description of extreme harmonic measures is much more difficult and greatly depends on the graph structure. Let $\mathbb{P}$ be a harmonic measure such that $\mathbb{P}(\Gamma_{f}(\mathcal{G}))=0$. By \cref{harmMeasure}, there exists a function $p:V\longrightarrow \mathbb{R}$ such that for any vertex $\nu\in V$ and any finite path $\tau$ ending at $\nu$, 
\begin{equation}\label{probaWeight}
\mathbb{P}(\Gamma_{\tau})=w(\tau)p(\nu).
\end{equation} 
Since a probability measure on $\Gamma(\mathcal{G})$ is characterized by its value on the sets $\Gamma_{\tau}, \tau\in \Gamma_{f}(\mathcal{G})$, $\mathbb{P}$ is completely determined by the map $p$; let us examine the properties of $p$.

Suppose that $\tau$ is a finite path of length $k$ ending at $\mu$, and denote by $G(\tau)$ the set of finite paths $\gamma$ of length $k+1$ such that $\gamma_{\downarrow k}=\tau$. Since $$\Gamma_{\tau}=\lbrace \tau\rbrace\sqcup\left(\bigsqcup\limits_{\gamma \in G(\tau)}\Gamma_{\gamma}\right)$$ 
and $\mathbb{P}$ is supported on the set of infinite paths (implying that $\mathbb{P}(\lbrace \tau\rbrace)=0)$,
\begin{equation}\label{partiProba}
\mathbb{P}(\Gamma_{\tau})=\sum_{\gamma\in G(\tau)} \mathbb{P}(\Gamma_{\gamma}).
\end{equation}
A path $\gamma$ of length $k+1$ is in $G(\gamma)$ if and only if $\gamma_{\downarrow k}=\tau$ and  $\mu=\gamma_{k}\nearrow\gamma_{k+1}$, and, if it is the case, its weight is $w(\gamma)=w(\tau)E(\mu,\gamma_{k+1})$. Therefore, with \eqref{probaWeight} and \eqref{partiProba},
\begin{equation}\label{recursion}
p(\mu)=\sum_{\mu\nearrow \nu}E(\mu,\nu)p(\nu).
\end{equation}
Reciprocally, any non-negative function $p:V\longrightarrow \mathbb{R}^{+}$ satisfying $p(\mu_{0})=1$ and \eqref{recursion} yields a harmonic measure in $\mathcal{H}_{\infty}(\mathcal{G})$ with the formula \eqref{probaWeight}. Therefore, the problem of describing the set of extreme points in $\mathcal{H}_{\infty}(\mathcal{G})$ is equivalent to the problem of finding the set of extreme points in the convex set of non-negative solutions of \eqref{recursion} with value $1$ on the root of $\mathcal{G}$.

\subsection{Martin entrance boundary}\label{MartEntrBoundary}

The theory of Martin boundary has been introduced by Doob and Martin (see \cite{doob1958discrete}) in order to approximate Gibbs measures by finite harmonic measures. This requires first the choice of an adequate topology on the set of probability measures on $\Gamma(\mathcal{G})$. We use the coarsest topology such that each evaluation map $\mathbb{P}\mapsto \mathbb{P}(\Gamma_{\tau})$ with $\tau\in \Gamma_{f}(\mathcal{G})$ is continuous: namely, a sequence of measures $(\mathbb{P}_{n})_{n\geq 1}$ converges to $\mathbb{P}$ if and only if $\mathbb{P}_{n}(\Gamma_{\tau})$ converges to $\mathbb{P}(\Gamma_{\tau})$ for all $\tau\in \Gamma_{f}(\mathcal{G})$. 

By Tychonoff's theorem and standard arguments, $\mathcal{H}(\mathcal{G})$ is a compact space in this topology. The discussion on the finite case in the previous paragraph yields that $\partial \mathcal{H}(\mathcal{G})\cap \mathcal{H}_{f}(\mathcal{G})$ is a discrete space with respect to this topology: since the set $\partial \mathcal{H}(\mathcal{G})\cap \mathcal{H}_{f}(\mathcal{G})$ is in bijection with $V$ with the map $\mu\mapsto \mathbb{P}_{\mu}$, this yields a topological embedding of $V$ (with the discrete topology) inside $\mathcal{H}(\mathcal{G})$. From now on, $V$ is identified with its image in $\mathcal{H}(\mathcal{G})$.

The choice of the topology and the definition of the Martin kernel in \cref{martinKernelDef} yield that a sequence $(\nu_{n})_{n\geq 0}$ in $V^{\mathbb{N}}$ converges in $\mathcal{H}(\mathcal{G})$ if and only if $K_{\mu}(\nu_{n})$ converges for each $\mu\in V$. We extend continuously $K_{\mu}$ on the closure of $V$ in $\mathcal{H}(\mathcal{G})$: for each element $\omega=\lim \nu_{n}$ in the closure of $V$, we set $K_{\mu}(\omega)=\lim K_{\mu}(\nu_{n})$. By compactness of $\mathcal{H}(\mathcal{G})$, the closure of $V$ in $\mathcal{H}(\mathcal{G})$ is a compact space, denoted by $\hat{V}$. 

The boundary of this space $\partial_{M}\mathcal{G}=\hat{V}\setminus V$ is called the Martin entrance boundary of $\mathcal{G}$. The fundamental result of Doob describes how Gibbs measures are approximated by finite harmonic measures.
\begin{thm}\label{martinEntranceBoundary}
With the notations above, the two following results hold:
\begin{itemize}
\item there exists a Borel subset $\partial_{\min}\mathcal{G}\subset \partial_{M}\mathcal{G}$, called the minimal boundary of $\mathcal{G}$, such that for any measure $\mathbb{P} \in \mathcal{H}_{\infty}(\mathcal{G})$, there exists a unique measure $\lambda_{\mathbb{P}}$ on $\partial_{\min}\mathcal{G}$ giving the kernel representation
$$\mathbb{P}(\Gamma_{\tau})=w(\tau) \int_{\partial_{\min}\mathcal{G}}K_{\mu}(x)d\lambda_{\mathbb{P}}(x),$$
for all $\mu\in V$ and $\tau\in \Gamma(\mu)$; 
\item if $\gamma$ is random path with distribution $\mathbb{P}\in\mathcal{H}_{\infty}(\mathcal{G})$, $(\gamma_{k})_{k\geq 0}$ converges almost surely to a $\partial_{\min}\mathcal{G}$-valued random variable $\gamma_{\infty}$; the law of $\gamma_{\infty}$ is given by the distribution $\lambda_{\mathbb{P}}$ from the first statement. Moreover, the probability that $(\gamma_{k})_{k\leq 0}$ goes through $\mu$ is exactly $d(\mu)\mathbb{E}\big(K_{\mu}(\gamma_{\infty})\big)$. 
\end{itemize}
 \end{thm}
The proof of this theorem can be found in \cite{doob1958discrete}. By the the first point of the theorem, the minimal boundary of $\mathcal{G}$ is exactly the set of extreme points of $\mathcal{H}_{\infty}(\mathcal{G})$: namely, $\partial_{\min}\mathcal{G}=\partial\mathcal{H}\cap\mathcal{H}_{\infty}(\mathcal{G})$. Being a subset of the Martin boundary, an element of the minimal boundary can be obtained as the limit of a sequence in $V$. The main problem is the identification of the minimal boundary inside the Martin boundary. In many graded graphs, both boundaries actually coincide; however, this is not always the case. This problem is usually solved by obtaining an adequate geometric description of the Martin boundary, which has been until now only defined as a limit object.

Remark that the space $\hat{V}$ contains all the information on the asymptotic behavior of $\mathbb{P}_{\mu}$ for vertices $\mu$ of large rank. The description of the Martin boundary is thus equivalent to the study of the approximation of Gibbs measures by finite harmonic measures.
\subsection{Geometric realization}
In many favorable cases, the two previous problems can be solved by introducing a geometric realization of the graph. The content of this paragraph is largely inspired by \cite[Section 7]{kerov1996boundary} and \cite[Section 2 and 3]{kerov1998boundary}: in particular, \cref{homeoTwoTopo} below is a straightforward generalization of the example given in \cite[Section 3]{kerov1998boundary}.
\begin{defn}\label{geomreal}
Let $\mathcal{G}=(V,E,\rho)$ be a graded graph. $\mathcal{G}$ admits a geometric realization $(Y,f,g)$ if there exists a compact space $Y$ and two maps $f:Y\longrightarrow \partial_{\min}\mathcal{G}$, $g:V\longrightarrow Y$, such that:
\begin{itemize}
\item $f$ is a homeomorphism;
\item $g$ is injective on $V_{n}$ for each $n\geq 0$;
\item for any sequence $(\mu_{n})_{n\geq 0}$ in $V^{\mathbb{N}}$ such that $\big(\rho(\mu_{n})\big)_{g\geq 1}$ is increasing, the convergence of $g(\mu_{n})$ towards $y\in Y$ implies the convergence of $(\mu_{n})_{n\geq 1}$ towards $f(y)$.
\end{itemize}
\end{defn}
The geometric realization $(Y,f,g)$ of a graded graph $\mathcal{G}$ gives a geometric way to encode approximations of Gibbs measures. In particular, $\hat{V}$ is described as a compact subset of $[0,1]\times Y$.
\begin{prop}\label{homeoTwoTopo}
Suppose that $\mathcal{G}$ is a graded graph which admits a geometric realization $(Y,f,g)$. Then,
\begin{itemize}
\item $\partial_{M}\mathcal{G}=\partial_{\min}\mathcal{G}$, and
\item $\hat{V}$ is homeomorphic to the space $\left(\bigcup_{n\geq 0} \lbrace\frac{1}{n+1}\rbrace\times g(V_{n})\right)\cup \left(\lbrace 0\rbrace\times Y\right)$, with the topology induced from $[0,1]\times Y$ by restriction.
\end{itemize}
\end{prop}
\begin{proof}
Let $\omega\in \partial_{M}\mathcal{G}$. There exists a sequence $(\mu_{n})_{n\geq 0}$ in $V$ which converges to $\omega$. Since $Y$ is compact, there exists an increasing sequence $(n_{k})_{k\geq 0}$ and $y\in Y$ such that $g(\mu_{n_{k}})$ converges to $y$ when $k$ goes to infinity.

By the definition of the geometric realization, $\mu_{n_{k}}$ converges to $f(y)\in \partial_{\min}\mathcal{G}$. Thus, $\omega=f(y)$ and $\omega\in\partial_{\min}\mathcal{G}$. Therefore, $\partial_{\min}\mathcal{G}=\partial_{M}\mathcal{G}$.

Let $\Phi$ be the map from $\hat{V}$ to $[0,1]\times Y$ defined by
$$\left\lbrace\begin{matrix}
\Phi(\mu)=\left(\frac{1}{\rho(\mu)+1},g(y)\right)&\text{ if }\mu\in V,\\
\Phi(\omega)=\big(0,f^{-1}(\omega)\big)&\text{ if }\omega\in\partial_{M}\mathcal{G}.
\end{matrix}\right.$$
The map $\Phi$ is injective and continuous by the definition of the geometric realization; since $\hat{V}$ is compact, $\Phi$ is a homeomorphism from $\hat{V}$ to $\Phi(\hat{V})$. Since $\Phi(\hat{V})= \bigcup_{n\geq 0} \lbrace\frac{1}{n+1}\rbrace\times g(V_{n})\cup \lbrace 0\rbrace\times Y$, the result is deduced.
\end{proof}
In the case of the graph $\mathcal{Z}$ that we are studying, the minimal boundary has already been described by Gnedin and Olshanski in \cite{gnedin2006coherent}. They also suggested a geometric realization for this graph in Conjecture 45 of the aforementioned paper. We prove this conjecture in \cref{SectionmartinBoundary}.

\section{Summary of the results}\label{sectionSummary}
This section is devoted to the definition of the graph $\mathcal{Z}$ and the statement of the results. This section uses the notations of the previous section.
\subsection{The graph $\mathcal{Z}$}\label{construcGraphZ}
As in the introduction, $A_{2}$ denotes the set with two elements $+,-$, $A_{2}^{*}$ denotes  the set of words in $A_{2}$, and $l(w)$ denotes the length of a word $w\in A_{2}^{*}$.
\begin{defn}
The graph $\mathcal{Z}$ is the graded graph $(V_{\mathcal{Z}},\rho_{\mathcal{Z}},E_{\mathcal{Z}})$ such that:
\begin{itemize}
\item the set of vertices $V_\mathcal{Z}$ is $A_{2}^{*}\cup\lbrace *\rbrace$;
\item the rank function $\rho_{\mathcal{Z}}$ is defined by $\rho_{\mathcal{Z}}(w)=l(w)+1$ for $w\in X^{*}_{2}$ and $\rho_{\mathcal{Z}}(*)=0$;
\item for all $w,w'\in V_{\mathcal{Z}}$, $E_{\mathcal{Z}}(w,w')\in \lbrace 0,1\rbrace$, and $E_{\mathcal{Z}}(w,w')=1$ if and only if $w=*$ and $w'=\emptyset$, or $w,w'\in X^{*}_{2}$ are such that $l(w')=l(w)+1$ and $w$ is a subword of $w'$.
\end{itemize}
\end{defn}
In the sequel, the set of vertices of $\mathcal{Z}$ is simply denoted by $\mathcal{Z}$ instead of $V_{\mathcal{Z}}$: in particular, the level sets of $\mathcal{Z}$ are denoted by $\mathcal{Z}_{n}$ for $n\geq 0$ and the Martin compactification, minimal boundary and Martin boundary of $\mathcal{Z}$ are respectively denoted by $\hat{\mathcal{Z}},\partial_{\min}\mathcal{Z}$ and $\partial_{M}\mathcal{Z}$. The first levels of $\mathcal{Z}$ are displayed in \cref{graphZ}. This graph has been introduced by Viennot, and all the results of this paragraph can be found in \cite{viennot1983maximal}. 

There is a simple description of finite paths on $\mathcal{Z}$. Following \cite{viennot1983maximal}, define a block (resp.~positive block, resp.~negative block) as a word in $A_{2}^{*}$ whose elements are all identical (resp.~equal to $+$, resp.~equal to $-$). A word $w\in A_{2}^{*}$ can be uniquely written as a sequence of blocks $b_{1}\cdots b_{r}$, where consecutive blocks have opposite signs. For example, the word $+++---+--+++$ admits the decomposition $b_{1}b_{2}b_{3}b_{4}b_{5}$ where $b_{1}=+++$, $b_{2}=---$, $b_{3}=+$, $b_{4}=--$ and $b_{5}=+++$.

Suppose that $w$ is a word in $A_{2}^{*}$ with the block decomposition $b_{1}\cdots b_{r}$; then, $w$ has exactly $r$ subwords of length $l(w)-1$, each of them obtained by decreasing the length of one block by one (where a block is simply deleted when its length is equal to zero). We choose the convention that the length of a positive block is reduced by deleting the last letter of the block and the length of a negative block is reduced by deleting the first letter of the block.

From the rule above, we have that for each word $w=w_{1}\cdots w_{n-1}$ of length $n-1$, there are exactly $n+1$ distinct words of length $n$ having $w$ as a subword: the words $w^{+0}:=-w$ and $w^{+n}:=w+$, and each word $w^{+i}:=w_{1}\cdots w_{i-1}+-w_{i+1}\cdots w_{n-1}$ for $1\leq i\leq n-1$ (see \cite[Lemma 2.1]{viennot1983maximal}). In particular, each finite path $\gamma$ of length $n$ on $\mathcal{Z}$ is uniquely determined by a sequence of integers $(t(1),\ldots,t(n-1))$ with $t(i)\in \lbrace 0,\ldots,i+1\rbrace$: the vertex $\gamma_{i+1}$ is defined as $\gamma_{i+1}=\gamma_{i}^{+t(i)}$. Reciprocally, any sequence of this form yields a path of length $n$ on $\mathcal{Z}$.

\subsection{Alternative description of $\mathcal{Z}$}\label{alternativeDescription}
We will often use an alternative description of the graph $\mathcal{Z}$, which is based on a bijective correspondence between words in $\lbrace +,-\rbrace$, ribbon Young diagrams and compositions.
\begin{defn} 
Let $n\geq 1$. A ribbon Young diagram of size $n$ is a skew Young diagram of size $n$ which is connected and does not contain any $2\times2$ square.\\
A composition $\lambda$ of $n$, written as $\lambda\vdash n$, is a sequence of positive integers $(\lambda_{1},\ldots,\lambda_{r})$ such that $\sum_{j=1}^{r} \lambda_{j}=n$.
\end{defn}

A ribbon Young diagram is completely described by the sequence of lengths of its rows and we simply write $(\lambda_{1},\ldots,\lambda_{r})$ to denote a ribbon Young diagram with first row having length $\lambda_{1}$, second row having length $\lambda_{2}$, and so on. This gives a canonical bijection between compositions of $n$ and ribbon Young diagrams of size $n$, which is similar to the one between partitions of $n$ and Young diagrams of size $n$. In the sequel, compositions and ribbon Young diagrams are considered as the same object: for example, the cell $i$ of a composition $\lambda$ refers to the cell $i$ of the ribbon Young diagram corresponding to $\lambda$ by the above bijection. 

The set of ribbon Young diagrams (or, equivalently, compositions) is denoted by $\Lambda$ and the set of ribbon Young diagrams of size $n$ is denoted by $\Lambda_{n}$.
 We number from $1$ to $n$ the cells of a ribbon Young diagram of size $n$, starting from the top left cell and finishing at the bottom right cell: for $2\leq i\leq n$ the cell $i$ is either below or right to the cell $i-1$. 

A word $w$ of length $n-1$ yields a ribbon Young diagram $\lambda(w)$ with $n$ cells as follows: the cell $i+1$ is right to the cell $i$ in $\lambda(w)$ when $w_{i}=+$, and the cell $i+1$ is below the cell $i$ when $w_{i}=-$. For example, the ribbon Young diagram $\lambda(w)$ of length $10$ corresponding to $w=++-+-+++-$ is represented in \cref{fig1} (the numbering of the cells has been omitted).
\begin{figure}[!h]
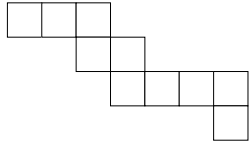


\gyoung(;;;,::;;,:::;;;;,::::::;)
\caption{\label{fig1}Ribbon Young diagram $(3,2,4,1)$ corresponding to the word $++-+-+++-$.}
\end{figure}

This gives clearly a bijection between words in $A_{2}^{*}$ of length $n-1$ and ribbon Young diagrams of size $n$. If $\lambda$ is a ribbon Young diagram, we denote by $w(\lambda)$ the corresponding word in $A_{2}^{*}$: if we write $w(\lambda)$ as $w_{1}\cdots w_{n-1}$, then $w_{i}=+$ if the cell $i+1$ is right to the cell $i$ and $w_{i}=-$ if the cell $i+1$ is below the cell $i$. The bijection between ribbon Young diagrams and words in $A_{2}^{*}$ can be translated into the language of compositions. Denote by $D_{\lambda}$ the descent set of a composition $\lambda=(\lambda_{1},\ldots,\lambda_{r})$ of $n$, which is the subset $\lbrace \lambda_{1},\lambda_{1}+\lambda_{2},\ldots,\sum_{1}^{r-1}\lambda_{i}\rbrace$ of $\lbrace 1,\ldots,n-1\rbrace$. Then, the equivalent bijection maps each composition $\lambda$ of $n$ to the word $w(\lambda)=w_{1}\cdots w_{n-1}$ such that $w_{i}=+$ if and only if $i\not\in D_{\lambda}$. 

Thanks to the above bijections, the graph $\mathcal{Z}$ can be alternatively described in terms of ribbon Young diagrams. 
 \begin{enumerate}
\item The set of vertices of rank $n$ of $\mathcal{Z}$ is the set of ribbon Young diagrams of size $n$. The unique vertex of rank $0$ is simply denoted by $\emptyset$.
\item Let $\mu=(\mu_{1},\ldots,\mu_{s})$ and $\lambda=(\lambda_{1},\ldots,\lambda_{r})$ be two ribbon Young diagrams of size respectively $n$ and $n+1$. There is an edge between $\mu$ and $\lambda$ if and only if
\begin{itemize}
\item either $r=s$ and for each $i$ except one $\mu_{i}=\lambda_{i}$ (thus, we have $\lambda_{j}=\mu_{j}+1$ for exactly one $1\leq j\leq r$), or
\item $r=s+1$, and there exists $j$ such that: for $k<j$, $\lambda_{k}=\mu_{k}$, $\lambda_{j}+\lambda_{j+1}-1=\mu_{j}$, and for $k>j$, $\lambda_{k+1}=\mu_{k}$ (namely, the row $\mu_{j}$ is split, and one cell is added at the end of the first piece).
\end{itemize}
\end{enumerate}
The first four levels of the graph $\mathcal{Z}$ in this alternative description are displayed in \cref{fig4}.
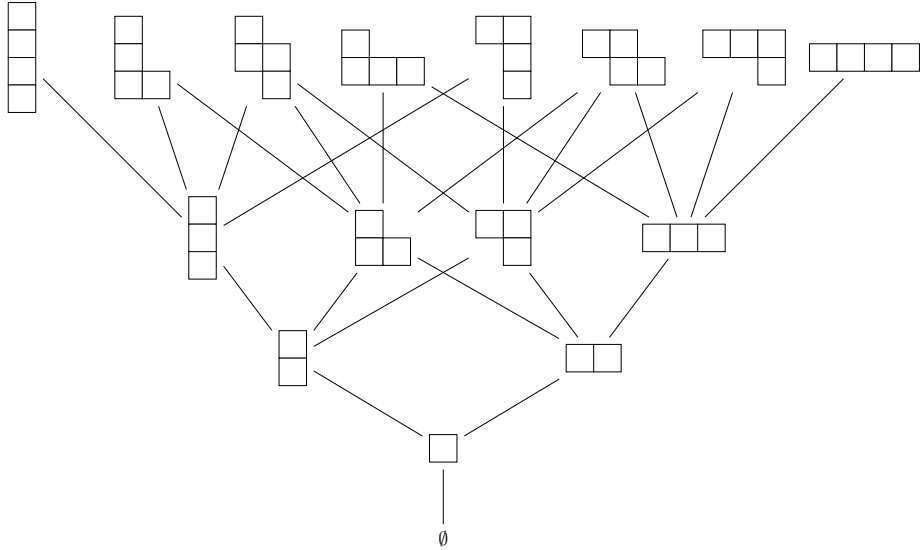
\begin{figure}[h!]\scalebox{.8}{
\begin{tikzpicture}
\node (0) at (5,0){$\emptyset$};
\node (11) at (5,1.5){$\gyoung(;)$};
\node (21) at (7.5,3){$\gyoung(;;)$};
\node (22) at (2.5,3){$\gyoung(;,;)$};
\node (31) at (9,5){$\gyoung(;;;)$};
\node (32) at (6,5){$\gyoung(;;,:;)$};
\node (33) at (4,5){$\gyoung(;,;;)$};
\node (34) at (1,5){$\gyoung(;,;,;)$};
\node (41) at (12,8){$\gyoung(;;;;)$};
\node (42) at (10,8){$\gyoung(;;;,::;)$};
\node (43) at (8,8){$\gyoung(;;,:;;)$};
\node (44) at (6,8){$\gyoung(;;,:;,:;)$};
\node (45) at (4,8){$\gyoung(;,;;;)$};
\node (46) at (2,8){$\gyoung(;,;;,:;)$};
\node (47) at (0,8){$\gyoung(;,;,;;)$};
\node (48) at (-2,8){$\gyoung(;,;,;,;)$};

\draw (0) to (11); 
\draw (11) to (21);
\draw (11) to (22);
\draw (21) to (31);
\draw (21) to (32);
\draw (21) to (33);
\draw (22) to (32);
\draw (22) to (33);
\draw (22) to (34); 
\draw (31) to (41);
\draw (31) to (42);
\draw (31) to (43);
\draw (31) to (45);
\draw (32) to (42);
\draw (32) to (43);
\draw (32) to (44);
\draw (32) to (46);
\draw (33) to (43);
\draw (33) to (45);
\draw (33) to (46);
\draw (33) to (47);
\draw (34) to (44);
\draw (34) to (46);
\draw (34) to (47);
\draw (34) to (48);

\end{tikzpicture}}
\caption{\label{fig4}Vertices of $\mathcal{Z}$ of rank $0$ to $4$.}

\end{figure}

\subsection{Geometric realization}\label{GeomRealizResult}
In \cite{gnedin2006coherent}, Gnedin and Olshanski suggested a possible geometric realization for the graph $\mathcal{Z}$.  This description is based on a topological space consisting of pairs of disjoint open sets in $[0,1]$, called oriented paintboxes.
\begin{defn}{\cite[Section 6.1]{gnedin2006coherent}}
An oriented paintbox is an ordered pair of disjoint open subsets of $]0,1[$.

The metric space $\mathcal{U}^{(2)}$ is the space 
$$\big\lbrace (U_{\uparrow},U_{\downarrow})\vert U_{\uparrow}\text{ and } U_{\downarrow}\text{ disjoint open subsets of } ]0,1[\big\rbrace$$ 
equipped with the distance $d$ between two oriented paintboxes $(U_{\uparrow},U_{\downarrow})$ and $(V_{\uparrow},V_{\downarrow})$ given by
$$d\big((U_{\uparrow},U_{\downarrow}),(V_{\uparrow},V_{\downarrow})\big)=\max\big(d_{\mathrm{Haus}}(U_{\uparrow}^{c},V_{\uparrow}^{c}),d_{\mathrm{Haus}}(U_{\downarrow}^{c},V_{\downarrow}^{c})\big),$$
where $X^{c}$ denotes $[0,1]\setminus X$ for $X\subset [0,1]$ and $d_\mathrm{Haus}$ denotes the Hausdorff distance.
\end{defn} 
From the definition of the metric, $\big(U_{\uparrow}(j),U_{\downarrow}(j)\big)_{j\geq 1}$ converges to $(V_{\uparrow},V_{\downarrow})$ if and only if, for each $\epsilon>0$, the two following conditions occur:
\begin{itemize}
\item for $j$ large enough, the numbers of connected components of size larger than $\epsilon$ in $U_{\uparrow}(j)$ (resp. $U_{\downarrow}(j)$) and in $V_{\uparrow}$ (resp. $V_{\downarrow}$) are equal;
\item the boundary points of the connected components of size larger than $\epsilon$ in $U_{\uparrow}(j)$ (resp. $U_{\downarrow}(j)$) converge to the ones of the connected components of size larger than $\epsilon$ in $V_{\uparrow}$ (resp. $V_{\downarrow}$).
\end{itemize}
In particular, $(U_{\uparrow}(j),U_{\downarrow} (j))_{j\geq 1}$ converges to $(\emptyset,\emptyset)$ if and only if the size of the largest connected component in $U_{\uparrow}(j)\cup U_{\downarrow}(j)$ tends to $0$.

There exists a map from $\mathcal{Z}$ to $\mathcal{U}^{(2)}$ which is injective on each level $\mathcal{Z}_{n}$ for $n\geq 0$. Namely, for each word $w=w_{1}\cdots w_{n-1}$ in $\mathcal{Z}_{n}$, set
\begin{equation}\label{definitionU(w)} 
U_{\uparrow}(w):=\left(\bigcup_{\substack{1\leq i\leq n-1\\ w_{i}=+}}\left[\frac{i-1}{n-1},\frac{i}{n-1}\right]\right)^{\mathrm{o}}
\end{equation}
and 
\begin{equation}\label{definitionU(w)bis}
U_{\downarrow}(w):=\left(\bigcup_{\substack{1\leq i\leq n-1\\ w_{i}=-}}\left[\frac{i-1}{n-1},\frac{i}{n-1}\right]\right)^{\mathrm{o}},
\end{equation}
where $X^{\mathrm{o}}$ denotes the interior of a set $X$. Define the map $g_{\mathcal{Z}}$ from $\mathcal{Z}$ to $\mathcal{U}^{(2)}$ by the formula $g_{\mathcal{Z}}(w)=\big(U_{\uparrow}(w),U_{\downarrow}(w)\big)$. 

Gnedin and Olshanski \cite{gnedin2006coherent} proved that  $\partial_{\min}\mathcal{Z}$ is homeomorphic to $\mathcal{U}^{(2)}$ by exhibiting a natural homeomorphism $\Phi_{\mathcal{Z}}$ from $\mathcal{U}^{(2)}$ to $\partial_{\min}\mathcal{Z}$, which is described in \cref{paintboxDef}. This suggests that $(\mathcal{U}^{(2)},\Phi_{\mathcal{Z}},g_{\mathcal{Z}})$ should be a geometric realization of $\mathcal{Z}$ in the sense of \cref{geomreal}. This conjecture has been first given in \cite{gnedin2006coherent} and is the first result of this paper.
\begin{thm}\label{mainresult1}
$(\mathcal{U}^{(2)},\Phi_{\mathcal{Z}},g_{\mathcal{Z}})$ is a geometric realization of $\mathcal{Z}$. In particular, $\partial_{\min}\mathcal{Z}=\partial_{M}\mathcal{Z}$.
\end{thm}
Following \cref{geomreal}, \cref{mainresult1} is proven by showing that for any sequence $(\mu_{n})_{n\geq 1}$ of vertices of increasing rank in $\mathcal{Z}$ and any element $(U_{\uparrow},U_{\downarrow})$ in $\mathcal{U}^{(2)}$, 
\begin{equation}\label{eqMainResult1}
\left(g_{\mathcal{Z}}(\mu_{n})\xrightarrow[n\rightarrow \infty]{}  (U_{\uparrow},U_{\downarrow})\right)\implies \left(\mu_{n}\xrightarrow[n\rightarrow \infty]{} \Phi_{\mathcal{Z}}(U_{\uparrow},U_{\downarrow})\right).
\end{equation}
The convergence on the right hand side is understood as the convergence in the topology of $\mathcal{Z}$ described in \cref{SectionProbaGradedGraph}, which is equivalent to the convergence of $\mathbb{P}_{\mu_{n}}(\Gamma_{\tau})$ towards $\mathbb{P}_{\Phi_{\mathcal{Z}}(U_{\uparrow},U_{\downarrow})}(\Gamma_{\tau})$ for all cylinders $\Gamma_{\tau}$, $\tau\in\Gamma_{f}(\mathcal{Z})$. The proof of the implication $\eqref{eqMainResult1}$ is done in \cref{SectionmartinBoundary} and is based on several intermediate results which are given in \cref{SectionintroFamilyXi,SectionCombinatorics,SectionAsymptotics}.
\subsection{Relation with the Young graph}\label{RelYouGra}
As we already explained in the introduction, the graded graph $\mathcal{Z}$ has an algebraic meaning which allows to relate it to several other graded graphs of special interest in probability. We will mainly focus on the relation between the graph $\mathcal{Z}$  and the Young graph $\mathcal{Y}$. In this paragraph, we only need the fact that the vertices of $\mathcal{Y}$ are indexed by partitions; refer to \cref{presentationYoungGraph} for a precise definition of $\mathcal{Y}$. Let us first describe the algebraic background of these two graphs, which relies on the ring $QSym$ of quasisymmetric functions and the ring $Sym$ of symmetric functions.

Let $X=\lbrace X_{1},X_{2},\ldots\rbrace$ be an infinite collection of commuting variables. On the one hand, $QSym$ is the ring of polynomials in $X$ which are shift invariant: namely, if $P\in QSym$ and $(\lambda_{1},\ldots,\lambda_{r})$ is a composition, then the coefficient of the monomial $X_{i_{1}}^{\lambda_{1}}\cdots X_{i_{r}}^{\lambda_{r}}$ in $P$ is the same for all increasing sequences $(i_{1}<i_{2}<\cdots<i_{r})$ of indices. The ring $QSym$ is graded, with the grading given by the degree of the polynomials, and each of its homogeneous bases is indexed by compositions. On the other hand, $Sym$ is the ring of polynomials in $X$ which are invariant under permutation of the variables. This ring is also graded by the degree of the polynomials and each homogeneous basis of $Sym$ is indexed by partitions. It is readily seen from the above definitions that $Sym$ is a subring of $QSym$.

The graphs $\mathcal{Y}$ and $\mathcal{Z}$ are respectively related to the Schur basis $\lbrace s_{\pi}\rbrace_{\pi\in \mathcal{Y}}$ of $Sym$ and to the fundamental basis $\lbrace F_{\lambda}\rbrace_{\lambda\in\Lambda}$ of $QSym$, where $s_{\pi}$ is the Schur function with respect to the partition $\pi\in \mathcal{Y}$ and for a composition $\lambda\in\Lambda_{n}$, $F_{\lambda}$ is defined as
$$F_{\lambda}(X_{1},X_{2},\ldots)=\sum_{\substack{i_{1}\leq i_{2}\leq \ldots\leq i_{n}\\ i_{j}<i_{j+1}\text{ if }j\in D_{\lambda}}}X_{i_{1}}\cdots X_{i_{n}}.$$
Both bases are among the most important ones in the theory of symmetric and quasisymmetric functions, and they have many applications in combinatorics, representation theory and probability (see for example \cite[Part I]{macdonald1998symmetric} for the theory of Schur functions and  \cite[Chapter 7.19]{stanley2011enumerative} for the theory of fundamental quasisymmetric functions).

The relation between the graph $\mathcal{Z}$ and the basis $\lbrace F_{\lambda}\rbrace_{\lambda\in\Lambda}$ is given by the following equality: for $\lambda\in \Lambda$,
$$F_{\lambda}F_{(1)}=\sum_{w(\lambda)\nearrow w(\mu)}F_{\mu},$$
where the relation $w(\lambda)\nearrow w(\mu)$ is the one given by the edge structure of $\mathcal{Z}$. As a consequence of the above equality, we can show that solving the equation \eqref{recursion} for $\mathcal{Z}$ is equivalent to finding the linear maps $\phi:QSym\longrightarrow \mathbb{R}$ which are non-negative on $\lbrace F_{\lambda}\rbrace_{\lambda\in\Lambda}$ and such that 
\begin{equation}\label{recursionbis}
\phi(F_{\lambda}F_{(1)})=\phi(F_{\lambda})
\end{equation}
 for all $\lambda\in\Lambda$. The latter fact is central in the description of the minimal boundary of $\mathcal{Z}$ given by Gnedin and Olshanski in \cite{gnedin2006coherent}.

The Young graph satisfies a similar relation with the Schur basis $\lbrace s_{\pi}\rbrace_{\pi\in \mathcal{Y}}$. Namely, for $\pi\in\mathcal{Y}$,
$$s_{\pi}s_{(1)}=\sum_{\pi\nearrow \pi'}s_{\pi'},$$
where $\pi\nearrow \pi'$ means that there is an edge between $\pi$ and $\pi'$ in $\mathcal{Y}$. Moreover, solving the equation \eqref{recursion} for $\mathcal{Y}$ is also equivalent to finding the linear maps $\phi:Sym\longrightarrow \mathbb{R}$ which are non-negative on $\lbrace s_{\pi}\rbrace_{\pi\in\mathcal{Y}}$ and such that 
\begin{equation}\label{recursionter}
\phi(s_{\pi}s_{(1)})=\phi(s_{\pi})
\end{equation}
 for all $\pi\in\mathcal{Y}$. Thoma proved in \cite{thoma1964unzerlegbaren} that the minimal boundary $\partial_{\min}\mathcal{Y}$ of the Young graph is homeomorphic to the set 
$$\Delta^{(2)}:=\lbrace (\alpha_{1}\geq \alpha_{2}\geq \ldots,\beta_{1}\geq \beta_{2}\geq \ldots)\vert \alpha_{i}\geq 0,\beta_{i}\geq 0, \sum_{i=1}^{\infty} \alpha_{i}+\beta_{i}\leq 1\rbrace,$$ 
with the topology of the pointwise convergence. Vershik and Kerov later showed \cite{vershik1981asymptotic} that the minimal boundary and the Martin boundary of $\mathcal{Y}$ are equal.
 
Since $s_{(1)}= F_{(1)}=\sum X_{i}$ and since any Schur function admits an expansion in the fundamental basis with non-negative coefficients (see \cite[Theorem 7.19.7]{stanley2011enumerative}), any solution of \eqref{recursionbis} for $\lbrace F_{\lambda}\rbrace_{\lambda\in\Lambda}$ gives also a solution of \eqref{recursionter} for $\lbrace s_{\pi}\rbrace_{\pi\in\mathcal{Y}}$, yielding a map from $\mathcal{H}_{\infty}(\mathcal{Z})$ to $\mathcal{H}_{\infty}(\mathcal{Y})$. It has been proved in \cite[Corollary 32]{gnedin2006coherent} that this map restricts to a continuous and surjective map $\Psi$ from $\partial_{\min}\mathcal{Z}$ to $\partial_{\min}\mathcal{Y}$. The map $\Psi$ has a simple meaning in terms of the geometric description of both boundaries: it sends an element $(U_{\uparrow},U_{\downarrow})$ of $\mathcal{U}^{(2)}$ to the double decreasing sequence $(\alpha_{n},\beta_{n})_{n\geq 1}$, where $(\alpha_{n})_{n\geq 1}$ (resp.~$(\beta_{n})_{n\geq 1}$) is the sequence of lengths of the interval components of $U_{\uparrow}$ (resp.~$U_{\downarrow}$) in the decreasing order.

In the following theorem, we give a combinatorial interpretation to the map $\Psi$ by establishing directly a map from paths on $\mathcal{Z}$ to paths on $\mathcal{Y}$, such that the induced action on $\partial_{\min}\mathcal{Z}$ is exactly $\Psi$.
\begin{thm}\label{mainResultYoung}
There is a surjective map $\Pi:\Gamma(\mathcal{Z})\longrightarrow \Gamma(\mathcal{Y})$ such that:
\begin{itemize}
\item if $\gamma\in\Gamma(\mathcal{Z})$ and $k\leq l(\gamma)$, then $\Pi(\gamma_{\downarrow k})=\Pi(\gamma)_{\downarrow k}$;
\item the pushforward by $\Pi$ of a harmonic measure on $\Gamma(\mathcal{Z})$ is a harmonic measure on $\Gamma(\mathcal{Y})$;
\item $\Pi$ induces the surjective map $\Psi:\partial_{\min}\mathcal{Z}\longrightarrow \partial_{\min}\mathcal{Y}$. Namely, if $\gamma_{\omega}$ is a random path with distribution $\omega\in \partial_{\min}\mathcal{Z}$, then $\Pi(\gamma_{\omega})$ is a random path with distribution $\Psi(\omega)$.

\end{itemize}
\end{thm} 
Note that this surjective map does not come from a map from the vertices of $\mathcal{Z}$ to the vertices of $\mathcal{Y}$. The construction of the map $\Pi$ is done in \cref{SectionRelationYong} using the RSK algorithm. In \cite{kerov1986characters}, Kerov and Vershik have given a representation of $\partial_{\min}\mathcal{Y}$ using a slightly modified version of the RSK algorithm. Their method is similar to the oriented paintbox construction of Gnedin and Olshanski, and an alternative proof of \cref{mainResultYoung} may be given using the results of \cite{kerov1986characters}. An interesting question is to push further this similarity. In \cite{sniady2014robinson}, \'{S}niady has been able to invert the map given in \cite{kerov1986characters} by applying the jeu de taquin procedure on infinite standard Young tableaux: this suggests that it would be possible to invert the map $\Phi_{\mathcal{Z}}$ of \cref{mainresult1} using a similar procedure.
\subsection{Law of large numbers and Central Limit Theorem}\label{lawLargeNumberCentralLimit}
We are considering a law of large numbers and a central limit theorem for random vertices of $\mathcal{Z}$ coming from a path following a Gibbs measure. The study of the shape of such random vertices is done by mapping a word of $\mathcal{Z}$ to a continuous function on $[0,1]$ which measures the local densities of $+$ and $-$ in the word. This map has been introduced in \cite{oshanin2004random}.

Consider the space $\mathcal{C}([0,1],\mathbb{R})$ of continuous functions on the interval $[0,1]$. For each element $(U_{\downarrow},U_{\downarrow})\in\mathcal{U} 
^{(2)}$, we define the function $f_{(U_{\uparrow},U_{\downarrow})}\in\mathcal{C}([0,1],\mathbb{R})$ with the formula
$$f_{(U_{\uparrow},U_{\downarrow})}(t)=\int_{0}^{t}\mathbf{1}_{U_{\uparrow}}(u)-\mathbf{1}_{U_{\downarrow}}(u)du.$$
In particular, if $w=w_{1}\cdots w_{n-1}$ is a word in $\mathcal{Z}_{n}$ and $0\leq s<t\leq 1$, then the quantity $f_{g_{\mathcal{Z}}(w)}(t)-f_{g_{\mathcal{Z}}(w)}(s)$ describes the proportion of $+$ and $-$ in $w$ between $w_{\lfloor sn\rfloor}$ and $w_{\lfloor tn\rfloor}$.
\begin{thm}\label{lawLageNumbers}
Let $(U_{\uparrow},U_{\downarrow})\in\mathcal{U}^{(2)}$, and let $\gamma$ be a random path following the law $\Phi_{\mathcal{Z}}(U_{\uparrow},U_{\downarrow})$. Then, almost surely,
$$f_{g_{\mathcal{Z}}(\gamma_{k})}\xrightarrow[k\rightarrow \infty]{} f_{(U_{\uparrow},U_{\downarrow})},$$
with respect to the supremum norm in $\mathcal{C}([0,1],\mathbb{R})$.

If $(U_{\downarrow},U_{\uparrow})=(\emptyset,\emptyset)$, then $f_{(\emptyset,\emptyset)}=0$ and the following convergence in law holds in
$\big(\mathcal{C}([0,1],\mathbb{R}),\Vert.\Vert_{\infty}\big)$:
$$\sqrt{n}f_{g_{\mathcal{Z}}}(\gamma_{k})\xrightarrow[k\rightarrow \infty]{} \frac{1}{\sqrt{3}}\mathcal{B},$$
where $\mathcal{B}$ is a standard Brownian motion on $[0,1]$.
\end{thm}
It would be very interesting to have a central limit theorem for all Gibbs measure $\Phi_{\mathcal{Z}}(U)$ for $U\in\mathcal{U}^{(2)}$. Whereas the central limit result of \cref{lawLageNumbers} can be easily extended in the case where $U$ has a finite number of interval components, the situation becomes more cumbersome when the number of interval components becomes infinite (see \cite{barbour2008small} for an analogous problem).
\section{Oriented paintbox construction and minimal boundary of $\mathcal{Z}$}\label{sectionOrientedPaintbox}
This section is devoted to the description of the minimal boundary of $\mathcal{Z}$ by Gnedin and Olshanski. This description is based on a bijection between finite paths on $\mathcal{Z}$ and permutations: this bijection, which has been found by Viennot in \cite{viennot1983maximal}, is the main tool in the proofs of this paper.
\subsection{Arrangements and paths on $\mathcal{Z}$}\label{EquivalencePathsArrangements}
Let $S_{n}$ denote the set of permutations of $n$ elements. As suggested in the introduction, a permutation $\sigma$ of $n$ is always written as a word $\sigma(1)\cdots\sigma(n)$ in $\lbrace 1,\ldots,n\rbrace$, with each integer appearing exactly once. We will also assume the existence of a permutation of $0$ elements consisting of the empty word, in order to simplify later statements.

Just like for words in $A_{2}^{*}$, we can define a subword order on the set of permutations: namely, a permutation $\sigma'$ of $n-1$ is a subword of $\sigma$ if there exist $1\leq i_{1}<i_{2}<\cdots<i_{n-1}\leq n$ such that $\sigma'=\sigma_{i_{1}}\cdots\sigma_{i_{n-1}}$. Since $\sigma'$ is a word in $\lbrace 1,\ldots,n-1\rbrace$, the only possibility is to discard the letter $n$ in $\sigma$. The only such permutation is denoted by $\sigma_{\downarrow}$. For example, if $\sigma=35728146$ then $\sigma_{\downarrow}=3572146$. Iterating the operation $\downarrow$ yields a sequence of permutations $(\sigma_{\downarrow 0},\sigma_{\downarrow 1},\ldots,\sigma_{\downarrow n-1}=\sigma_{\downarrow},\sigma_{n}=\sigma)$ such that $(\sigma_{\downarrow i})_{\downarrow}=\sigma_{\downarrow i-1}$. The meaning of $\sigma_{\downarrow k}$ is straightforward in the word description of $\sigma$: $\sigma_{\downarrow k}$ is exactly the word $\sigma$ where all integers larger than $k$ have been erased. For example, $(35728146)_{\downarrow 4}=(3214)$.
\begin{defn}\label{definitionCoherenceArragement}
A sequence of permutations $(\sigma_{1},\dots,\sigma_{l})$ with $l\in\mathbb{N}\cup\lbrace\infty\rbrace$ is called coherent when $\sigma_{i}=(\sigma_{i+1})_{\downarrow}$ for all $1\leq i\leq l$. 

An arrangement is an infinite sequence of permutations which is coherent.
\end{defn}
For example, the following sequence is the beginning of an arrangement
$$\big((1),(21),(231),(2341),(52341),\ldots\big).$$
The set of arrangements is denoted by $\mathbb{A}$.   For each $k\geq 0$, there is a map $p_{k}:\mathbb{A}\longrightarrow S_{k}$ which sends $(\sigma_{k})_{k\geq 0}$ to $\sigma_{k}$. $\mathbb{A}$ is considered with the initial topology with respect to the set of maps $\lbrace p_{k}\rbrace_{k\geq 0}$ and with the corresponding borelian $\sigma-$algebra. Thus, by Kolmogorov's extension theorem, any random variable $(\sigma_{k})_{k\geq 1}$ on $\mathbb{A}$ is uniquely determined by the law of its finite-dimensional projections $(\sigma_{1},\ldots,\sigma_{n})$ for all $n\geq 1$.

For $\sigma\in S_{n}$ and $0\leq i\leq n$, denote by $\sigma^{+i}$ the permutation $\sigma(1)\cdots\sigma(i)(n+1)\sigma(i+1)\cdots\sigma(n)$ (in particular, $\sigma^{+0}=(n+1)\sigma(1)\cdots\sigma(n)$ and $\sigma^{+n}=\sigma(1)\cdots\sigma(n)(n+1)$). Then, for each permutation $\sigma\in S_{n}$, the set of permutations in $S_{n+1}$ whose subword in $S_{n}$ is $\sigma$ is exactly the set $\lbrace \sigma^{+i}\rbrace_{0\leq i\leq n}$. In particular, for a permutation $\sigma\in S_{n}$, there is a unique $0\leq i_{n-1}\leq n-1$ such that $\sigma=(\sigma_{\downarrow})^{+i}$. Iterating this property yields that $\sigma$ is uniquely determined by a sequence $(i_{1},\ldots,i_{n-1})$, with $0\leq i_{j}\leq j$; the expression for $\sigma$ is then given by $\sigma=\left(\cdots ((1)^{+i_{1}})^{+i_{2}}\cdots\right)^{+i_{n-1}}$.

The link between permutations and paths on $\mathcal{Z}$ comes from the notion of ascent and descent of a permutation.
\begin{defn}
Let $\sigma$ be a permutation of $n$ elements and let $1\leq i\leq n-1$. We say that $\sigma$ has an ascent at $i$ when $\sigma(i)<\sigma(i+1)$ and that $\sigma$ has a descent at $i$ when $\sigma(i)>\sigma(i+1)$.

The descent word of the permutation $\sigma$ is the word $w(\sigma)=w_{1}\cdots w_{n-1}$ in $A_{2}^{*}$ such that $w_{i}=+$ if $i$ is an ascent of $\sigma$ and $w_{i}=-$ otherwise.
\end{defn}
For $w\in \mathcal{Z}$, the set of permutations with descent word $w$ is denoted by $D_{w}$. Let $\sigma$ be a permutation of $n$ elements and $0\leq i\leq n$; write $w_{1}\cdots w_{n-1}$ the descent word of $\sigma$ and $w'_{1}\cdots w'_{n}$ the one of $\sigma^{+i}$. Write $\sigma_{\leq i}$ (resp. $\sigma_{>i}$) to denote the word $\sigma$ restricted to its $i$ first symbols (resp. $n-i$ last symbols). Since $\sigma^{+i}$ is the permutation obtained from $\sigma$ by inserting $n+1$ between $\sigma(i)$ and $\sigma(i+1)$, the descent word of $\sigma^{+i}$ is related to the one of $\sigma$ as follows:
\begin{itemize}
\item since $\sigma^{+i}_{\leq i}=\sigma_{\leq i}$, $w_{j}=w'_{j}$ for $j<i$;
\item since $\sigma^{+i}(i+1)=n+1$, $i$ is necessarily an ascent of $\sigma^{+i}$ and $i+1$ is necessarily a descent of $\sigma^{+i}$, and thus $w'_{i}=+$ and $w'_{i+1}=-$;
\item since $\sigma^{+i}_{>i+1}=\sigma^{+i}_{>i}$, $w'_{j+1}=w_{j}$ for $i+1\leq j\leq n-1$.
\end{itemize} 
Thus, 
$$\begin{cases}w(\sigma^{+0})=-w_{1}\cdots w_{n-1}=w(\sigma)^{+0},\\w(\sigma^{+n})=w_{1}\cdots w_{n-1}+=w(\sigma)^{+n}, and \\w(\sigma^{+i})=w_{1}\cdots w_{i}+- w_{i+1}\cdots w_{n-1}= w(\sigma)^{+i} \text{ for }1\leq i\leq n-1,\\
\end{cases}$$
where we refer to \cref{construcGraphZ} for the definition of the word $w^{+i}$ given a word $w$ in $A_{2}^{*}$. Moreover, writing a permutation $\sigma$ as $\left(\cdots ((1)^{+i_{1}})^{+i_{2}}\cdots\right)^{+i_{n-1}}$ yields a unique path $\gamma(\sigma)=(\gamma_{0},\gamma_{1},\ldots,\gamma_{n})$ on $\mathcal{Z}$ such that $\gamma_{0}=*,\gamma_{1}=\emptyset$, and $\gamma_{j}=\gamma_{j-1}^{+i_{j-1}}$ for $j\geq 2$. By construction, $\gamma_{i}=w(\sigma_{\downarrow i})$ for $0\leq i\leq n$. An example of such correspondence is displayed in \cref{equivPathPerm}.

The map $\sigma\mapsto\gamma(\sigma)$ from $S_{n}$ to the set of paths of length $n$ on $\mathcal{Z}$ is clearly injective and thus bijective, since both sets have same cardinality. Since an infinite path on $\mathcal{Z}$ can be seen as a sequence of finite paths $(\tau_{i})_{i\geq 0}$ such that each $\tau_{i}$ has length $i$ and $(\tau_{i})_{\downarrow i-1}=\tau_{i-1}$, infinite paths are described by arrangements as follows.
\begin{thm*}{\cite[Proposition 2.4]{viennot1983maximal}}
The map 
$$\Omega:\left\lbrace \begin{matrix}
\mathbb{A}&\longrightarrow &\Gamma_{\infty}(\mathcal{Z})\\
(\sigma_{n})_{n\geq 0}&\mapsto&(w(\sigma_{n}))_{n\geq 0}
\end{matrix}\right.,$$
is bijective and for all $k\geq 1$, $\Omega((\sigma_{n})_{n\geq 0})_{\downarrow k}=\gamma(\sigma_{k})$.

\end{thm*}
In particular, any finite (resp.~infinite) random path on $\mathcal{Z}$ yields a random permutation (resp.~random arrangement) by the above maps.
\subsection{Harmonic measure in terms of permutations}\label{harmonicmeasurePermutations}
The equivalence between paths on $\mathcal{Z}$ and permutations allows to interpret harmonic measures in terms of random permutations. By \cref{harmMeasure}, a probability measure $\mathbb{P}$ on $\Gamma(\mathcal{Z})$ is harmonic when, for any finite path $\tau$, the probability $\mathbb{P}(\Gamma_{\tau})$ only depends on the last vertex of $\tau$. Therefore, by equivalence between finite paths on $\mathcal{Z}$ and permutations, finite (resp.~infinite) harmonic measures $\omega$ on $\mathcal{Z}$ are in bijection with random permutations (resp.~random arrangements) $\sigma_{\omega}$ such that for all $\sigma\in S_{k}$, $\mathbb{P}\big((\sigma_{\omega})_{\downarrow k}=\sigma\big)$ only depends on $w(\sigma)$.

For $w\in \mathcal{Z}$ the measure $\mathbb{P}_{w}$ is the uniform measure on all paths arriving at $w$; thus, the law of the corresponding random permutation $\sigma_{w}$ is the uniform law on all permutations with descent set $w$. For $k\geq 1$, let $w,w'$ be vertices of $\mathcal{Z}$ with $\rho(w')=k$ and $\rho(w)\geq \rho (w')$; then, the Martin kernel $K_{w'}(w)$, which is equal to $\mathbb{P}_{w}(\Gamma_{\tau})$ for any path $\tau$ ending at $w'$, is also equal to $\mathbb{P}\big((\sigma_{w})_{\downarrow k}=\sigma\big)$, where $\sigma$ is any permutation such that $w(\sigma)=w'$.

There is a convenient way to describe this uniform law with the ribbon Young diagram associated to $w$. A standard filling of a ribbon Young diagram of size $n$ is a filling of its cells with integers from $1$ to $n$, such that the filling is increasing from left to right along the rows and from bottom to top along the columns. A ribbon Young diagram $\lambda$ with a standard filling is called a standard ribbon Young tableau with shape $\lambda$, and the set of standard ribbon Young tableaux with shape $\lambda$ is denoted by $T(\lambda)$. Any standard ribbon Young tableau yields a permutation by reading the content of the consecutive cells, starting at the upper left cell. See \cref{fig2} for an example of standard ribbon Young tableau with its corresponding permutation.

\begin{figure}[!h]
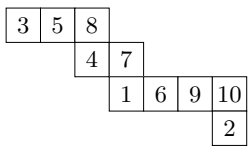

\gyoung(;3;5;8,::;4;7,:::;1;6;9;<10>,::::::;2)
\caption{\label{fig2}Standard filling of the ribbon Young diagram $(3,2,4,1)$ yielding the permutation $(3,5,8,4,7,1,6,9,10,2)$.}
\end{figure}
The rules of a standard filling imply that for any standard ribbon Young tableau of shape $\lambda$, the corresponding permutation has descent word $w(\lambda)$; moreover, it is readily seen that for any vertex $w $ of $\mathcal{Z}$, there is a bijection between $D_{w}$ and $T\big(\lambda(w)\big)$. Therefore, the random permutation $\sigma_{w}$ can be obtained from a uniform random standard ribbon Young tableau of shape $\lambda(w)$. This fact is used in the proof of \cref{mainresult1}.

\subsection{Oriented paintbox construction} \label{paintboxDef}
The minimal boundary of $\mathcal{Z}$ is described using the bijection between the set of infinite paths on $\mathcal{Z}$ and the set of arrangements. This description is called the oriented paintbox construction, and it gives for each element $U$ in $\mathcal{U}^{(2)}$ a random arrangement $\big(\sigma_{U}(n)\big)_{n\geq 1}$ as follows.
\begin{defn}\label{paintbox}{\cite[Definition 24]{gnedin2006coherent}}
Let $U=(U_{\uparrow},U_{\downarrow})$ be an oriented paintbox and let $(x_{1},x_{2},\ldots)$ be a sequence of distinct elements of  $[0,1]$. For each $k\geq 1$, the permutation $\sigma_{U}(x_{1},\ldots,x_{k})$ is defined by the following rule: $i$ is left to $j$ in the word $\sigma_{U}(x_{1},\ldots,x_{k})$ if and only if one of the three following situations arises:
\begin{itemize}
\item $x_{i}$ and $x_{j}$ are not in the same connected component of $U$ and $x_{i}<x_{j}$;
\item $x_{i}$ and $x_{j}$ are in the same connected component of $U_{\uparrow}$ and $i<j$;
\item $x_{i}$ and $x_{j}$ are in the same connected component of $U_{\downarrow}$ and $j<i$.
\end{itemize}  
For each $k\geq 1$, the random permutation $\sigma_{U}(k)$ of $k$ elements associated to the oriented paintbox $U$ is the random variable $\sigma_{U}(X_{1},\ldots,X_{k})$, where $(X_{1},X_{2},\ldots)$ is a family of independent uniform variables on $[0,1]$. The random arrangement $\sigma_{U}$ associated to the oriented paintbox $U$ is the sequence $(\sigma_{U}(1),\sigma_{U}(2),\ldots)$.
\end{defn}
The construction of $\sigma_{U}(X_{1},\ldots,X_{k})$ from $(X_{1},\ldots,X_{k})$ and $U\in\mathcal{U}^{(2)}$ is almost surely well-defined. If $U=(\emptyset,\emptyset)$, $\sigma_{(\emptyset,\emptyset)}(X_{1},\ldots,X_{k})$ is just the permutation associated to the ranking $(X_{i_{1}}<X_{i_{2}}<\cdots<X_{i_{k}})$. In particular, for each $k\geq 1$, the random variable $\sigma_{(\emptyset,\emptyset)}(k)$ has the uniform distribution on $S_{k}$. 

The next theorem is due to Gnedin and Olshanski in \cite{gnedin2006coherent} (based on an important work of Jacka and Warren \cite{jacka2007random}) and identifies $\mathcal{U}^{(2)}$ with the minimal boundary of the graded graph $\mathcal{Z}$.
\begin{thm}\label{minimalBoundary}{\cite[Theorem 44]{gnedin2006coherent}}
Let $U$ be an oriented paintbox. Through the bijection $\Omega$ between arrangements and infinite paths on $\mathcal{Z}$,
the random arrangement $\sigma_{U}$ comes from an extreme harmonic measure $\Phi_{\mathcal{Z}}(U)$ on the set of infinite paths of $\mathcal{Z}$.

Moreover, the map $\Phi_{\mathcal{Z}}$ from $\mathcal{U}^{(2)}$ to $\partial_{\min}\mathcal{Z}$ mapping $U$ to $\Phi_{\mathcal{Z}}(U)$ is a homeomorphism.
\end{thm} 
In particular, for each $k\geq 1$ and $\sigma\in S_{k}$, $\mathbb{P}(\sigma_{U}(k)=\sigma)$ only depends on the descent word $w$ of $\sigma$ and is denoted by $p_{U}(w)$.

\subsection{Sketch of the proof of the geometric realization}\label{ExplanationProof}
Let us explain the proof of \cref{mainresult1}. On the one hand, there is a homeomorphism $\Phi_{\mathcal{Z}}$ from the compact space $\mathcal{U}^{(2)}$ to the minimal boundary of $\mathcal{Z}$; on the other hand, an embedding $g_{\mathcal{Z}}$ of each level of $\mathcal{Z}$ into $\mathcal{U}^{(2)}$ has been introduced in \cref{sectionSummary} with the map $w\mapsto \big(U_{\uparrow}(w),U_{\downarrow}(w)\big)$. Therefore, from \cref{geomreal}, the only missing element to get the complete geometric realization of $\mathcal{Z}$ is the following convergence result: if $\big(w_{n}\big)_{n\geq 1}$ is a sequence of vertices of $\mathcal{Z}$ of increasing rank such that $\big(U_{\uparrow}(w),U_{\downarrow}(w)\big)$ converges in $\mathcal{U}^{(2)}$ to some oriented paintbox $(U_{\uparrow},U_{\downarrow})$, then $w_{n}$ converges to $\Phi_{\mathcal{Z}}(U_{\uparrow},U_{\downarrow})$ in $\hat{\mathcal{Z}}$.

Suppose that $(w_{n})_{n\geq 1}$ is a sequence of vertices with increasing rank. The convergence of $w_{n}$ to $\Phi_{\mathcal{Z}}(U_{\uparrow},U_{\downarrow})$ in $\hat{\mathcal{Z}}$ is equivalent to the convergence in law of $\mathbb{P}_{w_{n}}$ towards $\Phi_{\mathcal{Z}}(U_{\uparrow},U_{\downarrow})$. By the bijection between finite paths on $\mathcal{Z}$ and permutations, the finite harmonic measure $\mathbb{P}_{w_{n}}$ yields a random permutation $\sigma_{w_{n}}$ whose distribution is uniform on the set $D_{w_{n}}$ of permutations with descent word $w_{n}$; similarly, by the bijection between infinite paths on $\mathcal{Z}$ and arrangements, and by the oriented paintbox construction, the infinite harmonic measure $\Phi_{\mathcal{Z}}(U_{\uparrow},U_{\downarrow})$ yields the random arrangement $\big(\sigma_{U_{\uparrow},U_{\downarrow}}(k)\big)_{k\geq 1}$. Therefore, $w_{n}$ converges to $\Phi_{\mathcal{Z}}(U_{\uparrow},U_{\downarrow})$ in $\hat{\mathcal{Z}}$ if and only if for each $k\geq 1$ the random variable $(\sigma_{w_{n}})_{\downarrow k}$ converges in law to the random variable $\sigma_{U_{\uparrow},U_{\downarrow}}(k)$ when $n$ goes to infinity.

Let us briefly explain the proof of the convergence in law of $(\sigma_{w_{n}})_{\downarrow k}$ towards $\sigma_{U_{\uparrow},U_{\downarrow}}(k)$. Let $k\geq 1$ be a fixed integer. Then, $\sigma_{U_{\uparrow},U_{\downarrow}}(k)$ is the random variable $\sigma_{U_{\uparrow},U_{\downarrow}}(X_{1},\ldots,X_{k})$, where $X_{1},\ldots,X_{k}$ are independent random variables on $[0,1]$. It is shown in the next section that the variable $(\sigma_{w_{n}})_{\downarrow k}$ can also be sampled with the oriented paintbox construction as $\sigma_{\tilde{U}(w_{n})}(\xi_{1}^{w_{n}},\ldots,\xi_{k}^{w_{n}})$, where $\tilde{U}(w_{n})$ is a certain oriented paintbox depending on $w_{n}$ (different from the oriented paintbox $U(w_{n})$ associated to $w_{n}$ by the map $g_{\mathcal{Z}}$ in \cref{GeomRealizResult}) and $(\xi_{1}^{w_{n}},\ldots,\xi_{k}^{w_{n}})$ is some sequence of random variables in $[0,1]$. The main fact is that $\tilde{U}(w_{n})$ also converges to $(U_{\uparrow},U_{\downarrow})$ (see \cref{similarPaintbox} in \cref{SectionintroFamilyXi}) and the distribution of $(\xi_{1}^{w_{n}},\ldots,\xi_{k}^{w_{n}})$ becomes close to the one of $(X_{1},\ldots,X_{n})$ (see \cref{generalCase}) when $n$ becomes large. The convergence in law of $\sigma_{\tilde{U}(w_{n})}(\xi_{1}^{w_{n}},\ldots,\xi_{k}^{w_{n}})$ to $\sigma_{U_{\uparrow},U_{\downarrow}}(X_{1},\ldots,X_{k})$ is then deduced from a general theorem on convergence of random oriented paintbox constructions (see \cref{convergencePaintbox} in \cref{SectionintroFamilyXi}). 
The most difficult part is the proof of the approximation of the sequence $(\xi_{i}^{w})_{1\leq i\leq k}$ by a sequence of independent uniform variables on $[0,1]$ for large words $w$. The proof is done in two steps in \cref{SectionAsymptotics,SectionmartinBoundary}; it is based on combinatorial results related to ribbon Young diagrams, which have been obtained in a previous paper \cite{sawtooth} and which are recalled in \cref{SectionCombinatorics}.

\section{The familiy $(\xi^{w}_{i})_{i\geq 1}$}\label{SectionintroFamilyXi}

In this section, $w=w_{1}\cdots  w_{n-1}$ is a fixed word in $A_{2}^{*}$ of length $n-1$ and $\lambda:=\lambda(w)$ is the corresponding ribbon Young diagram, which has size $n$. Note that from \cref{sectionOrientedPaintbox}, the distribution of $\sigma_{w}$ is the same as the distribution of $\sigma_{\lambda}$, where $\sigma_{\lambda}$ is the permutation obtained by reading the cells of a uniformly random standard ribbon Young tableau of shape $\lambda$. The goal of this section is to introduce a family of random variables $(\xi^{w}_{i})_{i\geq 1}$ and an oriented paintbox $\tilde{U}(w)$ such that the random variable $\big(\sigma_{w}\big)_{\downarrow k}$ has the same law as the oriented paintbox construction $\sigma_{\tilde{U}(w)}(\xi_{1}^{w},\ldots,\xi_{k}^{w})$. In order to simplify the notations, we adopt the convention that $w_{0}=w_{n}=\emptyset$. Recall that the cells of $\lambda$ are identified with integers from $1$ to $n$ by ordering them along the ribbon Young diagram, starting at the upper left cell.
\subsection{Combinatorics of descents and ascents} \label{definitionRun+Extreme}
We will introduce several definitions related to a ribbon Young diagram; all these definitions are pictured in \cref{FigExtrem} and \cref{FigSlopes}.
A cell $i\in \llbracket 1,n\rrbracket$ is a peak of $\lambda$ if $w_{i-1}\in \lbrace \emptyset,+\rbrace$ and $w_{i}\in \lbrace \emptyset,-\rbrace$, and $i\in \llbracket 1,n\rrbracket$ is a valley if $w_{i-1}\in \lbrace \emptyset,-\rbrace$ and $w_{i}\in \lbrace \emptyset,+\rbrace$. This definition makes sense if we consider any standard filling $\sigma$ of $\lambda$: $\sigma(i)$ is a local maximum (resp.~minimum) of $\sigma=\sigma(1)\cdots\sigma(n)$ if and only if $i$ is a peak (resp.~valley) of $\lambda$. Let $V$ denote the set of valleys of $\lambda$, let $P$ denote its set of peaks, and let $\mathcal{E}=V\cup P$ denote its set of extreme cells, which are either peaks or valleys. We denote by $\lbrace e_{1}<e_{2}<\ldots<e_{t+1}\rbrace$ the elements of $\mathcal{E}$.

A run $s$ of $\lambda$ is an interval $\llbracket a,b\rrbracket$ of $\llbracket 1,n\rrbracket$ such that $a,b$ are consecutive elements of $\mathcal{E}$. A run $\llbracket a,b\rrbracket$ is called descending if $a\in P$ and ascending if $a\in V$. The runs are ordered by the smallest endpoint of the corresponding interval, which yields a total ordered set $S=\lbrace s_{i}\rbrace_{1\leq i\leq t}$. Thus, each element $s_{i}$ of $S$ corresponds to the interval $\llbracket e_{i};e_{i+1}\rrbracket $ of $\lambda$. In particular, two consecutive runs $s_{i}$ and $s_{i+1}$ overlap on $e_{i+1}$. The length of a run $s_{i}$ is defined as the value $l_{i}=e_{i+1}-e_{i}$. For example if $\lambda=(3,2,1,3,1)$, then $V=\lbrace 1,4,7,10\rbrace$, $P=\lbrace 3,5,9\rbrace$ and  $S=\lbrace \llbracket 1,3\rrbracket,\llbracket 3,4\rrbracket,\llbracket 4,5\rrbracket,\llbracket 5,7\rrbracket,\llbracket 7,9\rrbracket,\llbracket 9,10\rrbracket \rbrace$. The peaks, valleys and runs of $\lambda$ are displayed in \cref{FigExtrem}.

\begin{figure}[h!]
\begin{tikzpicture}
\node(1)[draw, fill=gray!20, minimum height=1cm, minimum width=1 cm] at (0,3){1};
\node(2)[draw, minimum height=1cm, minimum width=1 cm] at (1,3){2};
\node(3)[draw, fill=gray!60, minimum height=1cm, minimum width=1 cm] at (2,3){3};
\node(4)[draw, fill=gray!20, minimum height=1cm, minimum width=1 cm] at (2,2){4};
\node(5)[draw,  fill=gray!60,minimum height=1cm, minimum width=1 cm] at (3,2){5};
\node(6)[draw, minimum height=1cm, minimum width=1 cm] at (3,1){6};
\node(7)[draw, fill=gray!20, minimum height=1cm, minimum width=1 cm] at (3,0){7};
\node(8)[draw, minimum height=1cm, minimum width=1 cm] at (4,0){8};
\node(9)[draw, fill=gray!60, minimum height=1cm, minimum width=1 cm] at (5,0){9};
\node(10)[draw, fill=gray!20, minimum height=1cm, minimum width=1 cm] at (5,-1){10};

\draw[ultra thick] (1.5,3.3)--(1.5,1.7);
\draw[ultra thick] (2.5,2.3)--(2.5,-0.3);
\draw[ultra thick] (4.5,0.3)--(4.5,-1.3);

\draw[ultra thick] ( -0.3,3.5)--(2.3,3.5);
\draw[ultra thick] (1.7,2.5)--(3.3,2.5);
\draw[ultra thick] (2.7,0.5)--(5.3,0.5);

\node (A) at (1,3.5){};
\node (B) at (2.5,2.5){};
\node (C) at (4,0.5){};
\node(D) at (1.5,2.5){};
\node (E) at (2.5,1.5){};
\node (F) at (4.5,-0.5){};

\node(P) at (6,4){Ascending runs};
\node (V)at (0,-1){Descending runs};

\draw [->](P)->(A);
\draw [->](P)->(B);
\draw[->] (P)->(C);
\draw[->] (V)->(D);
\draw[->] (V)->(E);
\draw [->](V)->(F);

\node (V)[draw, fill=gray!20, minimum height=1cm, minimum width=1 cm] at (8,2){};
\node (V')at (9.5,2){Valley};
\node (P)[draw, fill=gray!60, minimum height=1cm, minimum width=1 cm] at (8,0.5){};
\node (P') at (9.5,0.5){Peak};
\end{tikzpicture}
\caption{\label{FigExtrem}Peaks, valleys and runs in $\lambda=(3,2,1,3,1)$.}
\end{figure}
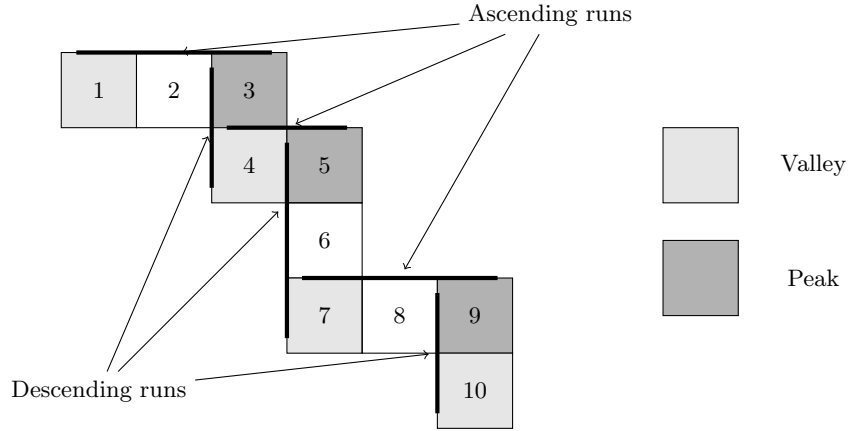
For any cell $i$ of $\lambda$, its slope $\mathfrak{s}(i)=\llbracket x(i)+1,y(i)\rrbracket$ is defined as the maximum subinterval of $\llbracket 1,n\rrbracket$ that contains $i$ and no other peak or valley. The use of the shifted variable $x(i)$ will simplify later formulas. In the previous example, $\mathfrak{s}(2)=\llbracket 2,2\rrbracket$ and $\mathfrak{s}(7)=\llbracket 6,8\rrbracket$. Note that slopes differ from runs: the slope of $i$ is a run with the extreme cells removed if $i$ is not an extreme cell, and the union of two runs with the first and last extreme cells removed when $i$ is an extreme cell.
\begin{figure}[h!]
\begin{tikzpicture}
\node(1)[draw,  minimum height=1cm, minimum width=1 cm] at (0,3){1};
\node(2)[draw,fill=gray!20, minimum height=1cm, minimum width=1 cm] at (1,3){2};
\node(3)[draw,  minimum height=1cm, minimum width=1 cm] at (2,3){3};
\node(4)[draw, minimum height=1cm, minimum width=1 cm] at (2,2){4};
\node(5)[draw, minimum height=1cm, minimum width=1 cm] at (3,2){5};
\node(6)[draw,fill=gray!60, minimum height=1cm, minimum width=1 cm] at (3,1){6};
\node(7)[draw, fill=gray!60, minimum height=1cm, minimum width=1 cm] at (3,0){7};
\node(8)[draw, fill=gray!60, minimum height=1cm, minimum width=1 cm] at (4,0){8};
\node(9)[draw, minimum height=1cm, minimum width=1 cm] at (5,0){9};
\node(10)[draw, minimum height=1cm, minimum width=1 cm] at (5,-1){10};

\node (V)[draw, fill=gray!20, minimum height=1cm, minimum width=1 cm] at (8,2){};
\node (V')at (9.5,2){Slope of $2$};
\node (P)[draw, fill=gray!60, minimum height=1cm, minimum width=1 cm] at (8,0.5){};
\node (P') at (9.5,0.5){Slope of $7$};
\end{tikzpicture}
\caption{\label{FigSlopes}The slopes $\mathfrak{s}(2)$ and $\mathfrak{s}(7)$ in the ribbon Young diagram $\lambda=(3,2,1,3,1)$.}
\end{figure}
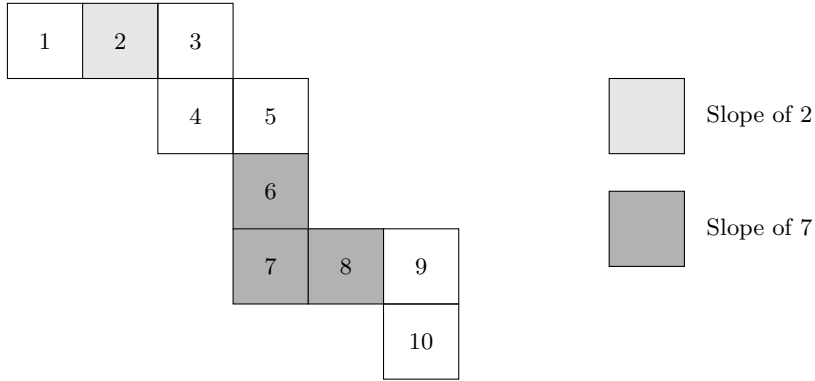
\subsection{Definition of $(\xi^{w}_{i})_{i\geq 1}$}\label{defiXi}
We are now constructing a family of random variables $(\xi^{w}_{i})_{i\geq 1}$ in $[0,1]$ in order to express the random variable $\sigma_{w}$ as an oriented paintbox construction. 
Let $\big(U_{i}\big)_{i\geq 1}$ be a family of independent random variables uniformly distributed on $[0,1]$. In the following definition, recall that $\llbracket x(j)+1,y(j)\rrbracket$ denotes the slope of the cell $j$ in $\lambda$, for $1\leq j\leq n$.
\begin{defn}\label{defAverageCoord}
Let $\sigma\in D_{w}$. The averaged coordinate of $i$ with respect to $\sigma$ is the random variable defined by $\xi_{i}(\sigma)=0$ if $i>n$, and
$$\xi_{i}(\sigma)=\frac{x\big(\sigma^{-1}(i)\big)}{n}+U_{i}\frac{y\big(\sigma^{-1}(i)\big)-x\big(\sigma^{-1}(i)\big)}{n},$$
for $1\leq i\leq n$.

For $\sigma_{w}$ chosen uniformly among $D_{w}$, $\xi_{i}(\sigma_{w})$ is denoted by $\xi^{w}_{i}$, and the vector $(\xi^{w}_{1},\ldots,\xi^{w}_{k})$ is denoted by $\xi^{w}(k)$.
\end{defn}
When $k$ is equal to $n$, we simply write $\xi^{w}$ instead of $\xi^{w}_{n}$. The construction of $\xi^{w}_{i}$ roughly means that we sample $\sigma_{w}$ as a uniformly random standard filling of $\lambda$, we look at the cell containing $i$ with respect to this filling, and then we sample a random variable uniformly distributed on the rescaled slope of this cell. 

As Proposition \ref{equivPaintbox} will show below, $\big(\xi_{i}(\sigma)\big)_{1\leq i\leq n}$ completely characterizes the permutation $\sigma\in D_{w}$: namely, knowing that $\sigma$ is in $D_{w}$, the value of the random vector $\big(\xi_{i}(\sigma)\big)_{1\leq i\leq n}$ allows almost surely to reconstruct $\sigma$. This reconstruction requires a slightly modified version of $U(w)=\big(U_{\uparrow}(w),U_{\downarrow}(w)\big)$. In the following definition, recall that $e_{i}$ denotes the first cell of the run $s_{i}$ (as defined in Section \ref{definitionRun+Extreme}), and that $e_{i}$ belongs either to $V$ or to $P$, depending whether $s_{i}$ is an ascending or a descending run. 
\begin{defn}\label{RunPaintbox}
The run paintbox $\tilde{U}(w)$ associated to a word $w$ of length $n-1$ is the element of $\mathcal{U}^{(2)}$ consisting of the open subsets
\begin{itemize}
\item $\tilde{U}_{\uparrow}(w)=\bigcup_{e_{i}\in V} \left]\frac{e_{i}-1}{n},\frac{e_{i+1}-1}{n}\right[$, and
\item $\tilde{U}_{\downarrow}(w)=\bigcup_{e_{i}\in P}\left]\frac{e_{i}-1}{n},\frac{e_{i+1}-1}{n}\right[$,
\end{itemize}
with $e_{i+1}=n+1$ if $e_{i}=n$.
\end{defn}
The definition of $\tilde{U}(w)$ is very close to the one of $U(w)$ in \eqref{definitionU(w)},\eqref{definitionU(w)bis}: the intervals of $\tilde{U}(w)$ are essentially a rescaled version of the intervals of $U(w)$ with parameter $\frac{n-1}{n}$. The following lemma shows that the run paintbox $\tilde{U}(w)$ becomes closer to $U(w)$ when $n$ goes to infinity.
\begin{lem}\label{similarPaintbox}
With respect to the distance on $\mathcal{U}^{(2)}$,
$$d\big(U(w),\tilde{U}(w)\big)\leq\frac{1}{n}.$$
\end{lem}
\begin{proof}
The definition of $U(w)$ yields the following open sets:
$$U_{\uparrow}(w)=\bigcup_{e_{i}\in V,e_{i}\not =n}\left]\frac{e_{i}-1}{n-1},\frac{e_{i+1}-1}{n-1}\right[,$$
and
$$U_{\downarrow}(w)=\bigcup_{e_{i}\in P,e_{i}\not =n}\left]\frac{e_{i}-1}{n-1},\frac{e_{i+1}-1}{n-1}\right[.$$
Let us show that $U_{\uparrow}(w)^{c}$ is included in the $\frac{1}{n}-$inflation of $\tilde{U}_{\uparrow}(w)^{c}$ and conversely (the proof for $U_{\downarrow}(w)^{c}$ and $\tilde{U}_{\downarrow}(w)^{c}$ is the same).
The $\frac{1}{n}-$inflation of $U_{\uparrow}(w)^{c}$ is 
\begin{align*}
U_{\uparrow}&(w)^{c,1/n}\\
=&\left(\bigcup_{e_{i}\in P,e_{i}\not =n}\left[\frac{e_{i}-1}{n-1}-\frac{1}{n},\frac{e_{i+1}-1}{n-1}+\frac{1}{n}\right]\cap[0,1]\right)\cup \left[0,\frac{1}{n}\right]\cup \left[1-\frac{1}{n},1\right].
\end{align*}
On the other hand, 
$$\tilde{U}_{\uparrow}(w)^{c}=\left(\bigcup_{e_{i}\in P}\left[\frac{e_{i}-1}{n},\frac{e_{i+1}-1}{n}\right]\right)\cup \lbrace 0\rbrace\cup \lbrace 1\rbrace.$$
Suppose that $e_{i}\not=n$. For all $1\leq k\leq n-1$, $\frac{k}{n-1}-\frac{1}{n}\leq \frac{k}{n}\leq \frac{k}{n-1}+\frac{1}{n}$, thus 
\begin{align*}
\left[\frac{e_{i}-1}{n},\frac{e_{i+1}-1}{n}\right]\subset& \left[\left(\frac{e_{i}-1}{n-1}-\frac{1}{n}\right)\vee 0,\left(\frac{e_{i+1}-1}{n-1}+\frac{1}{n}\right)\wedge 1\right]\\
\subset& U_{\uparrow}(w)^{c,1/n}.
\end{align*}
If $e_{i}=n$, $\left[\frac{e_{i}-1}{n},\frac{e_{i+1}-1}{n}\right]=[1-1/n,1]\subset U_{\uparrow}(w)^{c,1/n}$. Hence, in any case,  $\tilde{U}_{\uparrow}(w)^{c}\subset U_{\uparrow}(w)^{c,1/n}$.

For the converse inclusion, the $\frac{1}{n}-$inflation of $\tilde{U}_{\uparrow}(w)^{c}$ is
\begin{align*}
\tilde{U}_{\uparrow}&(w)^{c,1/n}\\
=&\left(\bigcup_{e_{i}\in P}\left[\left(\frac{e_{i}-1}{n}-\frac{1}{n}\right)\vee 0,\left(\frac{e_{i+1}-1}{n}+\frac{1}{n}\right)\wedge 1\right]\right)\cup [0,\frac{1}{n}]\cup [1-\frac{1}{n},1],
\end{align*}
and 
$$U_{\uparrow}(w)^{c}=\left(\bigcup_{e_{i}\in P,e_{i}\not=n}\left[\frac{e_{i}-1}{n-1},\frac{e_{i+1}-1}{n-1}\right]\right)\cup \lbrace 0\rbrace\cup \lbrace 1\rbrace.$$
Since for $1\leq k\leq n-1$, $\frac{k}{n}-\frac{1}{n}\leq \frac{k}{n-1}\leq \frac{k}{n}+\frac{1}{n}$, for each $e_{i}\not = n$, 
$$\left[\frac{e_{i}-1}{n-1},\frac{e_{i+1}-1}{n-1}\right]\subset \left[\left(\frac{e_{i}-1}{n}-\frac{1}{n}\right)\vee 0,\left(\frac{e_{i+1}-1}{n}+\frac{1}{n}\right)\wedge 1\right],$$
and therefore $U_{\uparrow}(w)^{c}\subset \tilde{U}_{\uparrow}(w)^{c,1/n}$.

Doing the same for $U_{\downarrow}(w)$ and $\tilde{U}_{\downarrow}(w)$ concludes the proof.
\end{proof}
The previous lemma implies that for any sequence $(w_{n})_{n\geq 1}$ with $\vert w_{n}\vert \rightarrow \infty$, the convergence of $U(w_{n})$ is equivalent to the convergence of $\tilde{U}(w_{n})$, and both have the same limit. We show in the following result how to reconstruct $\sigma\in D_{w}$ from $\big(\xi_{i}(\sigma)\big)_{1\leq i\leq n}$ and $\tilde{U}(w)$. Recall that $\sigma_{\tilde{U}(w)}\big((\xi_{i}(\sigma))_{1\leq i\leq k}\big)$ is the oriented paintbox construction (as defined in \cref{paintboxDef}) applied to the tuple $\big(\xi_{i}(\sigma)\big)_{1\leq i\leq k}$ and the oriented paintbox $\tilde{U}(w)$.
\begin{prop}\label{equivPaintbox}
Let $1\leq k\leq n$ and let $\sigma$ be a permutation in $D_{w}$. Then, almost surely,
$$\sigma_{\tilde{U}(w)}\Big(\big(\xi_{i}(\sigma)\big)_{1\leq i\leq k}\Big)=\sigma_{\downarrow k}.$$
In particular, the random variables $(\sigma_{w})_{\downarrow k}$ and $\sigma_{\tilde{U}(w)}\big(\xi^{w}(k)\big)$ have the same law.
\end{prop}
\begin{proof}
It is enough to prove it for $k=n$. Let $\sigma\in D_{w}$ and write $\xi_{i}(\sigma)=\xi_{i}$ and $\xi=(\xi_{i}(\sigma))_{1\leq i\leq n}$. It is equivalent to prove that for $1\leq i<j\leq n$, $$\left(\sigma_{\tilde{U}(w)}(\xi)\right)^{-1}(i)<\left(\sigma_{\tilde{U}(w)}(\xi)\right)^{-1}(j)\iff \sigma^{-1}(i)<\sigma^{-1}(j).$$
If $\sigma^{-1}(i)<\sigma^{-1}(j)$, then $i$ is left to $j$ in the associated filling of $\lambda$. This is possible in one of the two following situations.
\begin{enumerate}
\item $\mathfrak{s}(i)$ and $\mathfrak{s}(j)$ are disjoint and $\mathfrak{s}(i)$ is left to $\mathfrak{s}(j)$. In this case, $\xi_{i}$ and $\xi_{j}$ are not in the same interval component of $\tilde{U}(w)$ and $\xi_{i}$ is in an interval component left to the one of $\xi_{j}$. By the run paintbox construction, $$\left(\sigma_{\tilde{U}(w)}(\xi)\right)^{-1}(i)<\left(\sigma_{\tilde{U}(w)}(\xi)\right)^{-1}(j).$$
\item $\mathfrak{s}(i)$ and $\mathfrak{s}(j)$ overlap. This implies that $i$ and $j$ are in a same run $s=\llbracket e_{h},e_{h+1}\rrbracket$ of $\lambda$, with $1\leq h\leq t$. Let $I_{s}=\left]\frac{e_{h}-1}{n},\frac{e_{h+1}}{n}\right[$. Since $i<j$ and $\sigma^{-1}(i)<\sigma^{-1}(j)$, the run $s$ has to be an ascending one and thus $e_{h}\in V$, $e_{h+1}\in P$ and $I_{s}$ is an interval component of 
$\tilde{U}_{\uparrow}(w)$. In particular, $\sigma^{-1}(i)$ cannot be a peak, and $\sigma^{-1}(j)$ cannot be a valley. Thus $\xi_{i}$ is either in an interval component left to $I_{s}$, or in $I_{s}$. For similar reasons, $\xi_{j}$ is either in an interval component right to $I_{s}$, or in $I_{s}$. This implies that if $\xi_{i}$ or $\xi_{j}$ is not in $I_{s}$, $\left(\sigma_{\tilde{U}(w)}(\xi)\right)^{-1}(i)<\left(\sigma_{\tilde{U}(w)}(\xi)\right)^{-1}(j)$. But if $\xi_{i}$ and $\xi_{j}$ are both in $I_{s}$, since the latter is in $\tilde{U}_{\uparrow}(w)$, the same inequality holds.
\end{enumerate}
Finally, in any case, 
$$\sigma^{-1}(i)<\sigma^{-1}(j)\Longrightarrow \left(\sigma_{\tilde{U}(w)}(\xi)\right)^{-1}(i)<\left(\sigma_{\tilde{U}(w)}(\xi)\right)^{-1}(j).$$
The pattern is exactly the same to prove that
$$\sigma^{-1}(i)>\sigma^{-1}(j)\Longrightarrow \left(\sigma_{\tilde{U}(w)}(\xi)\right)^{-1}(i)>\left(\sigma_{\tilde{U}(w)}(\xi)\right)^{-1}(j),$$
yielding the first part of the proposition. 

Recall that if $(x_{i})_{1\leq i\leq n}$ is a sequence in $[0,1]^{n}$, the sequence of permutations $\big(\sigma_{\tilde{U}(w)}(x_{1},\ldots,x_{k})\big)_{1\leq k\leq n}$ is coherent in the sense of \cref{definitionCoherenceArragement}. From the first part of the proof, $\sigma_{w}=\sigma_{\tilde{U}(w_{n})}(\xi_{1},\ldots,\xi_{n})$, thus, for all $1\leq k\leq n$, 
$$(\sigma_{w})_{\downarrow k}=\big(\sigma_{\tilde{U}(w_{n})}(\xi_{1},\ldots,\xi_{n})\big)_{\downarrow k}=\sigma_{\tilde{U}(w_{n})}(\xi_{1},\ldots,\xi_{k}).$$
Therefore, $(\sigma_{w})_{\downarrow k}$ and $\sigma_{\tilde{U}(w)}\big(\xi^{w}(k)\big)$ are equal and have same distribution.
\end{proof}
The following lemma shows that it is possible to recover exactly the positions of $\lbrace 1,\ldots,k\rbrace$ in the filling $\sigma$ of $\lambda$ from $\big(\xi_{i}(\sigma)\big)_{1\leq i\leq k}$. 
\begin{lem}\label{fixedPosition}
Let $\sigma,\sigma'$ be two permutations in $D_{w}$. If the two sequences $\big(\sigma^{-1}(1),\ldots,\sigma^{-1}(k)\big)$ and $\big(\sigma'^{-1}(1),\ldots,\sigma'^{-1}(k)\big)$ are not equal, then the random vectors $\big(\xi_{1}(\sigma),\ldots,\xi_{k}(\sigma)\big)$ and $\big(\xi_{1}(\sigma'),\ldots,\xi_{k}(\sigma')\big)$ have disjoint supports (where the support of a random vector is defined as the support of its distribution).
\end{lem}
\begin{proof}
The proof is done by induction on $k\geq 1$. Let $k=1$. The integer $1$ has to be located in a valley of $\lambda$. If $\sigma^{-1}(1)\not =\sigma'^{-1}(1)$, then $1$ is located in a different valley in $\sigma$ and $\sigma'$. Thus, the slopes of $\sigma^{-1}(1)$ and $\sigma'^{-1}(1)$ are disjoint, and the random variables $\xi_{1}(\sigma)$ and $\xi_{1}(\sigma')$ have disjoint supports.

Let $k>1$. Suppose that $\big(\sigma^{-1}(1),\ldots,\sigma^{-1}(k)\big)\not = \big(\sigma'^{-1}(1),\ldots,\sigma'^{-1}(k)\big)$. By the induction hypothesis, if $\big(\sigma^{-1}(1),\ldots,\sigma^{-1}(k-1)\big)$ is not equal to $\big(\sigma'^{-1}(1),\ldots,\sigma'^{-1}(k-1)\big)$, the two random vectors $\big(\xi_{1}(\sigma),\ldots,\xi_{k-1}(\sigma)\big)$ and $\big(\xi_{1}(\sigma'),\ldots,\xi_{k-1}(\sigma')\big)$ have disjoint supports. This yields also that the two random vectors $\big(\xi_{1}(\sigma),\ldots,\xi_{k}(\sigma)\big)$ and $\big(\xi_{1}(\sigma'),\ldots,\xi_{k}(\sigma')\big)$ have disjoint supports.

Thus, let us assume that $\big(\sigma^{-1}(1),\ldots,\sigma^{-1}(k-1)\big)= \big(\sigma'^{-1}(1),\ldots,\sigma'^{-1}(k-1)\big)$. This implies that $\sigma^{-1}(k)\not=\sigma'^{-1}(k)$; since the positions of $\lbrace 1,\ldots, k-1\rbrace$ are the same in the fillings $\sigma$ and $\sigma'$ of $\lambda$, the cell containing $k$ in $\sigma$ is in a different run than the cell containing $k$ in $\sigma'$. Therefore their slopes are disjoint, and the random vectors $\big(\xi_{1}(\sigma),\ldots,\xi_{k}(\sigma)\big)$ and $\big(\xi_{1}(\sigma'),\ldots,\xi_{k}(\sigma')\big)$ have disjoint supports.
\end{proof}
Although the following result is crucial for the proof of \cref{mainresult1}, its proof is rather technical and has been postponed to the Appendix.
\begin{prop}\label{convergencePaintbox}
Let $(U_{n})_{n\geq 1}$ be a sequence of elements of $\mathcal{U}^{(2)}$ and let $\big((X^{n}(i))_{i\geq 1}\big)_{n\geq 1}$ be a sequence of random infinite vectors in $[0,1]^{\mathbb{N}}$. Let $(X^{0}(1),X^{0}(2),\ldots)$ be a random infinite vector in $[0,1]^{\mathbb{N}}$. Suppose that each finite dimensional marginal law of each of these random vectors admits a density with respect to the Lebesgue measure. If $(U_{n})_{n\geq 1}$ converges to $U$ in $\mathcal{U}^{(2)}$ and for each $k\geq 1$, $X_{k}^{n}=\big(X^{n}(1),\ldots,X^{n}(k)\big)$ converges in law to $X_{k}^{0}=\big(X^{0}(1),\ldots,X^{0}(k)\big)$, then for each $k\geq 1$,
$$\sigma_{U_{n}}(X_{k}^{n})\xrightarrow{\law} \sigma_{U}(X^{0}_{k}).$$
\end{prop}
In the previous proposition, the absolute continuity  with respect to the Lebesgue measure of the finite-dimensional marginal laws of the random vectors $(X^{n}(i))_{i\geq 1}$ for $n\geq 1$ may be dropped (contrary to the absolute continuity of the ones of $(X^{0}_{i})_{i\geq 1}$, which is necessary); however, the proof of the proposition is easier if we assume that all the distributions involved admit a density.
\section{Combinatorics of large compositions}\label{SectionCombinatorics}
The purpose of this section is to introduce the background material to prove that the distribution of $\xi^{w}(k)$ is close to the distribution of independent uniform random variables on $[0,1]$ when $w$ is long. Since $\xi^{w}_{j}$ depends uniquely on the cell of $\lambda(w)$ in which $j$ is located in the random filling $\sigma_{\lambda}$ of $\lambda(w)$, it is necessary to evaluate the probability for $j$ to be located in a particular cell $c$ of $\lambda(w)$.
For a composition $\lambda$ of $n$ and $i\in\lambda$ a fixed cell, denote by $\lambda_{\leq i}$ (resp.~$\lambda_{\geq i}$, resp.~$\lambda_{< i}$, resp.~$\lambda_{> i}$)  the composition $\lambda$ restricted to cells left (resp.~right, resp.~strictly left, resp.~strictly right) to $i$. Denote by $d(\lambda)$ the number of standard fillings of the ribbon Young diagram associated to $\lambda$; remark that $d(\lambda)=d\big(w(\lambda)\big)$, where $d(w(\lambda))$ is the number of paths from the root to $w(\lambda)$ on $\mathcal{Z}$.

Let us focus here on the location of $1$ in $\sigma_\lambda$. Since $1$ is necessary a local minimum in any filling of $\lambda$, it has to be located in a valley $v\in V$.  For a fixed valley $v$ of $\lambda$, the cardinality of standard fillings of $\lambda$ such that $1$ is located in $v$ is exactly the number of possibilities to fill in the part of $\lambda$ strictly left of $v$, with any subset $S$ of cardinality $\vert \lambda_{<v}\vert$ of $\llbracket 2,n\rrbracket$, and to fill in independently the part of $\lambda$ strictly right to $v$ with the complementary subset of $S$ in $\llbracket 2,n\rrbracket$. Thus,
\begin{equation}\label{position1}
\mathbb{P}_{\sigma_\lambda}(1\in v)=\frac{(n-1) !}{\vert \lambda_{< v}\vert !\  \vert \lambda_{>v}\vert !}\frac{d(\lambda_{>v})d(\lambda_{<v})}{d(\lambda)}.
\end{equation}
The problem is therefore essentially to relate $d(\lambda_{>v})d(\lambda_{<v})$ to $d(\lambda)$.
\subsection{Probabilistic approach to the combinatorics of descents}
Ehrenborg, Levin and Readdy (see \cite{ehrenborg2002probabilistic}) formalized in the context of descent words an old relation between permutations of $n$ and polytopes in $[0,1]^{n}$. Namely, since the volume of the set \[R_{\sigma}=\lbrace x_{\sigma(1)}<\cdots<x_{\sigma(n)}\rbrace\]
for $\sigma\in S_{n}$ is exactly $\frac{1}{n!}$, it is possible to determine probabilistic quantities on $ S_{n}$ by integrating certain functions that are constant on each region $R_{\sigma}$. The appropriate functions for descent words were found in \cite{ehrenborg2002probabilistic}, yielding some new estimates like in \cite{ehrenborg2002asymptotics} and \cite{bender2004asymptotics}. The model of Ehrenborg, Levin and Readdy is presented in this paragraph, but in a modified way to focus only on the set of extreme cells $\mathcal{E}$ (as defined in \cref{definitionRun+Extreme}). This yields the following framework: let $\lambda$ be a composition of $n\geq 2$ with the set of extreme cells $\mathcal{E}=\lbrace e_{1}=1,e_{2},\ldots,e_{t+1}=n\rbrace$. Suppose for example that the first cell is a valley (namely $e_{1}\in V$) and recall that $s_{j}$ is the run between $e_{j}$ and $e_{j+1}$, with $l_{j}$ its length. We associate to $\lambda$ the couple of random variables $(X_{\lambda},Y_{\lambda})$ on $[0,1]^{2}$ with density
\begin{align}
d_{X_{\lambda},Y_{\lambda}}&(x_{1},x_{t+1})\nonumber\\=&\frac{1}{\mathcal{V}_{\lambda}}\idotsint_{[0,1]^{t-1}}\mathbf{1}_{x_{1}<x_{2}>x_{3}<x_{4}>\dots} \prod_{1\leq i\leq t}\frac{\vert x_{i}-x_{i+1}\vert^{l_{i}-1}}{(l_{i}-1)!}\prod_{2\leq i\leq t}d x_{i},\label{densitySawtooth}
\end{align}
where $\mathcal{V}_{\lambda}$ is a normalization constant. If the first cell is a peak, the inequalities in the expression for the density are reversed. If $\lambda=(1)$ is the unique composition of $1$, the expression for the distribution of $(X_{(1)},Y_{(1)})$ is simply $\delta_{X_{(1)}=Y_{(1)}}$, since there is only one particle in the model.

The latter probabilistic model is related to the combinatorics of descent words through the equality
\begin{equation}\label{volumeDescentSet}
d(\lambda)=n !\ \mathcal{V}_{\lambda}, 
\end{equation}
whose proof can be found in \cite{ehrenborg2002probabilistic}. Moreover, this model behaves simply with respect to the concatenation of compositions.
\begin{defn}
Let $\lambda=(\lambda_{1},\ldots,\lambda_{r})$ and $\mu=(\mu_{1},\ldots,\mu_{s})$ be two compositions of $m$ and $n$. The concatenated composition $\lambda+\mu$ is the composition of $n+m$
$$\lambda+\mu=(\lambda_{1},\ldots,\lambda_{r}+\mu_{1},\mu_{2},\ldots,\mu_{s}),$$
and the concatenated composition $\lambda -\mu$ is the composition of $n+m$
$$\lambda-\mu=(\lambda_{1},\ldots,\lambda_{r},\mu_{1},\mu_{2},\ldots,\mu_{s}).$$
\end{defn}
This definition has a simple meaning in terms of associated ribbon Young diagrams: namely, the diagram of $\lambda+\mu$ (resp.~$\lambda-\mu$) is the juxtaposition of the one of $\lambda$ and the one of $\mu$ such that the last cell of $\lambda$ is left to (resp.~above) the first cell of $\mu$. An application of \cite[Section 2]{ehrenborg2002probabilistic} (see also \cite[Lemma 2]{bender2004asymptotics}) implies the following expression for the concatenation of compositions in the probabilistic framework.
\begin{prop}\label{concatenationVolume}
Let $\lambda,\mu$ be two compositions, $\vartheta\in\lbrace -,+\rbrace$. 
Then, 
$$\mathcal{V}_{\lambda\vartheta\mu}=\mathcal{V}_{\lambda}\mathcal{V}_{\mu}\mathbb{P}(Y_{\lambda}\leq_{\vartheta} X_{\mu})$$
and 
$$d_{X_{\lambda\vartheta\mu},Y_{\lambda\vartheta\mu}}(x,y)=\frac{1}{\mathbb{P}(Y_{\lambda}\leq_{\vartheta} X_{\mu})}\iint_{[0,1]^{2}}d_{X_{\lambda},Y_{\lambda}}(x,u)\mathbf{1}_{u\leq_{\vartheta}v}d_{X_{\mu},Y_{\mu}}(v,y)\,du\,dv,$$
where $\leq_{-}=\geq$ and $\leq_{+}=\leq$, and the couples $(X_{\lambda},Y_{\lambda})$ and $(X_{\mu},Y_{\mu})$ are considered as independent.
\end{prop}
The previous proposition yields a particular case that helps to compute the law of $\xi_{1}^{w}$. Denote by $F_{\lambda}$ (resp. $G_{\lambda}$) the cumulative distribution function of the random variable $X_{\lambda}$ (resp. $Y_{\lambda}$) associated to a composition $\lambda$ in \eqref{densitySawtooth}. Since all the random variables take value in $[0,1]$, we consider $F_{\lambda}$ and $G_{\lambda}$ as function on $[0,1]$.
\begin{cor}\label{firstStepPos1}
Let $\lambda$ be a composition of $n$ and $v$ a valley of $\lambda$. Then 
$$\mathbb{P}_{\lambda}(1\in v)=\frac{1}{n}\frac{1}{\int_{0}^{1}\big(1-G_{\lambda_{<v}}(t)\big)\big(1-F_{\lambda_{>v}}(t)\big)dt},$$
with the convention $G_{\lambda_{<1}}=F_{\lambda_{>n}}=0$.
\end{cor}
\begin{proof}
Since $v$ is a valley, $\lambda$ can be written as $\lambda_{<v}-\lambda_{\geq v}$. Thus, the previous proposition yields
$$\mathcal{V}_{\lambda}=\mathcal{V}_{\lambda_{<v}-\lambda_{\geq v}}=\mathcal{V}_{\lambda_{<v}}\mathcal{V}_{\lambda_{\geq v}}\mathbb{P}(Y_{\lambda_{<v}}\geq X_{\lambda_{\geq v}}).$$
Conditioning the expectation on the value of $X_{\lambda_{\geq v}}$ gives by independence,
$$\mathbb{P}(Y_{\lambda_{<v}}\geq X_{\lambda_{\geq v}})=\int_{0}^{1}\big(1-G_{\lambda_{<v}}(t)\big)d_{X_{\lambda_{\geq v}}}(t)dt.$$
On the other hand from the previous proposition, since $X_{\lambda_{\geq v}}=X_{(1)+\lambda_{>v}}$,
\begin{align*}
d_{X_{\lambda_{\geq v}}}(t)=&\frac{1}{\mathbb{P}(Y_{(1)}\leq X_{\lambda_{>v}})}\iiint_{[0,1]^{3}}\delta(t,u)\,\mathbf{1}_{u\leq v}\,d_{X_{\lambda_{>v}},Y_{\lambda_{>v}}}(v,y)\,du\,dv\,dy\\
=&\frac{\mathcal{V}_{\lambda_{>v}}\mathcal{V}_{(1)}}{\mathcal{V}_{\lambda_{\geq v}}}\big(1-F_{\lambda_{>v}}(t)\big),
\end{align*}
and finally, since $\mathcal{V}_{(1)}=1$,
$$\mathcal{V}_{\lambda}=\mathcal{V}_{\lambda_{<v}}\mathcal{V}_{\lambda_{>v}}\int_{0}^{1}\big(1-G_{\lambda_{<v}}(t)\big)\big(1-F_{\lambda_{>v}}(t)\big)dt.$$
Using the latter result in \cref{position1,volumeDescentSet} yields
\begin{align*}
\mathbb{P}_{\sigma_\lambda}(1\in v)=&\frac{(n-1) !}{\vert \lambda_{< v}\vert !  \:\vert \lambda_{>v}\vert !}\frac{d(\lambda_{>v})\,d(\lambda_{<v})}{d(\lambda)}\\
=&\frac{(n-1) !}{\vert \lambda_{< v}\vert !\:  \vert \lambda_{>v}\vert !}\,\frac{\vert \lambda_{< v}\vert !\: \vert \lambda_{>v}\vert !}{n !}\frac{\mathcal{V}_{\lambda_{>v}}\mathcal{V}_{\lambda_{<v}}}{\mathcal{V}_{\lambda}}\\
=&\frac{1}{n}\frac{\mathcal{V}_{\lambda_{>v}}\,\mathcal{V}_{\lambda_{<v}}}{\mathcal{V}_{\lambda_{<v}}\mathcal{V}_{\lambda_{>v}}\int_{0}^{1}\big(1-G_{\lambda_{<v}}(t)\big)\big(1-F_{\lambda_{>v}}(t)\big)dt}\\
=&\frac{1}{n}\frac{1}{\int_{0}^{1}\big(1-G_{\lambda_{<v}}(t)\big)\big(1-F_{\lambda_{>v}}(t)\big)dt}.
\end{align*}
\end{proof}
The previous corollary shows that the probability that the entry $1$ is located on a valley $v$ in $\lambda$ essentially depends on the quantity 
$$\Delta(\lambda_{<v},\lambda_{>v})=\int_{0}^{1}\big(1-G_{\lambda_{<v}}(t)\big)\big(1-F_{\lambda_{>v}}(t)\big)dt.$$
\begin{defn}
The correlation $\Delta(\lambda,\mu)$ between two compositions $\lambda$ and $\mu$ is the integral
$$\Delta(\lambda,\mu)=\int_{0}^{1}\big(1-G_{\lambda}(t)\big)\big(1-F_{\mu}(t)\big)dt.$$
\end{defn}
\subsection{Estimates on $\Delta(\lambda,\nu)$}
We obtain in this section several estimates on $\Delta(\lambda,\mu)$ by using some results on the behavior of $F_{\mu},G_{\mu}$ obtained in \cite{sawtooth}; the reader should refer to \cite{sawtooth} to find detailed proofs of the results used in this section. The first result is a bound on $F_{\lambda}$ which depends on the length of the first run of $\lambda$.
\begin{prop}[Cor.$\, 4.9$, \cite{sawtooth}]\label{evaluateRun}
Let $\lambda$ be a composition with the first run of length $R$. If the first run is increasing, the following inequality holds:
$$1-(1-t)^{R}\leq F_{\lambda}(t)\leq 1-(1-t)^{R+1}.$$
If the first run is decreasing, then
$$t^{R+1}\leq F_{\lambda}(t)\leq t^{R}.$$
\end{prop}
The same result holds for $G_{\lambda}$ after exchanging the increasing run and the decreasing run. Using the above proposition, we can prove the following estimates on $\Delta(\lambda,\mu)$.
\begin{lem}\label{boundDelta}
Denote by $L$ the size of the last run of $\lambda$ and by $R$ the size of the first run of $\mu$. Set $\vartheta_{1}=+$ if the last run of $\lambda$ is increasing and $\vartheta_{1}=-$ otherwise, and set similarly $\vartheta_{2}=+$ if the first run of $\mu$ is increasing and $\vartheta_{2}=-$ otherwise. Then, 
\begin{itemize}
\item if $\vartheta_{1}=+,\vartheta_{2}=+$,
$$\frac{1}{R+3}\leq \Delta(\lambda,\mu)\leq \frac{1}{R+1};$$
\item if $\vartheta_{1}=-,\vartheta_{2}=+$,
$$ \frac{1}{L+R+3}\leq \Delta(\lambda,\mu)\leq \frac{1}{L+R+1};$$
\item if $\vartheta_{1}=+,\vartheta_{2}=-$,
$$\frac{R+L+2}{R+L+1}-\frac{1}{R+1}-\frac{1}{L+1}\leq \Delta(\lambda,\mu)\leq \frac{R+L+4}{R+L+3}-\frac{1}{R+2}-\frac{1}{L+2};$$
\item if $\vartheta_{1}=-,\vartheta_{2}=-$,
$$\frac{1}{L+3}\leq \Delta(\lambda,\mu)\leq \frac{1}{L+1}.$$
\end{itemize}
\end{lem}
\begin{proof}
In each case, the proof is done by using the inequalities of Proposition \ref{evaluateRun} in the definition of $\Delta(\lambda,\mu)$. Let us detail the proof in the first case, since the three other cases are proven similarly. Since $\vartheta_{1}=+$, applying Proposition \ref{evaluateRun} to the last run of $\lambda$ yields the inequalities
\begin{equation}\label{ineq1}
t^{L+1}\leq G_{\lambda}(t)\leq t^{L}
\end{equation}
for $t\in[0,1]$. Since $\vartheta_{2}=+$, Proposition \ref{evaluateRun} gives similarly
\begin{equation}\label{ineq2}
1-(1-t)^{R}\leq F_{\mu}(t)\leq 1-(1-t)^{R+1}
\end{equation}
for $t\in[0,1]$. Therefore, using \eqref{ineq1} and \eqref{ineq2} in the definition of $\Delta(\lambda,\mu)$ yields
$$\int_{0}^{1}(1-t^{L})(1-t)^{R+1}dt\leq \Delta(\lambda,\mu)\leq \int_{0}^{1}(1-t^{L+1})(1-t)^{R}dt.$$
On the first hand,
\begin{align*}
\int_{0}^{1}(1-t^{L})&(1-t)^{R+1}dt\\
=&\frac{1}{R+2}-\frac{(R+1)!L!}{(R+L+2)!}=\frac{1}{R+2}\left(1-\frac{L!}{(R+3)\dots(L+R+2)}\right).
\end{align*}
Since $1-\frac{L!}{(R+3)\dots(L+R+2)}\geq 1-\frac{1}{R+3}\geq \frac{R+2}{R+3}$, 
$$\int_{0}^{1}(1-t^{L})\big[1-(1-(1-t)^{R+1})\big]dt\geq \frac{1}{R+3}.$$
And on the other hand,
\begin{align*}
\int_{0}^{1}(1-t^{L+1})(1-t)^{R}dt=&\frac{1}{R+1}-\frac{R!(L+1)!}{(R+L+2)!}\\
=&\frac{1}{R+1}\left(1-\frac{(L+1)!}{(R+2)\dots(L+R+2)}\right),
\end{align*}
which is smaller than $\frac{1}{R+1}$.
Finally,
$$\frac{1}{R+3}\leq \Delta(\lambda,\mu)\leq \frac{1}{R+1}.$$

\end{proof}
We denote respectively by $A(\lambda,\mu)$ and by $B(\lambda,\mu)$ the upper bound and the lower bound of $\Delta(\lambda,\mu)$ which have been established in the previous lemma. 
\begin{lem}\label{smallequivequiv}
Let $\epsilon>0$. There exists a constant $\eta>0$ independent of $\lambda$ and $\mu$ such that $B(\lambda,\mu)\leq \eta$ implies that 
$$\frac{A(\lambda,\mu)-B(\lambda,\mu)}{B(\lambda,\mu)}\leq \epsilon.$$
\end{lem}
\begin{proof}
Let $\epsilon>0$. If $(\vartheta_{1},\vartheta_{2})=(+,-)$, then $B(\lambda,\mu)=1-\frac{1}{R+1}-\frac{1}{L+1}+\frac{1}{L+R+1}$. Since $\frac{1}{L+R+1}-\frac{1}{L+1}$ is increasing in $L$,
$$B(\lambda,\mu)\geq 1-\frac{1}{R+1}-\frac{1}{2}+\frac{1}{R+2}\geq \frac{1}{2}-\frac{1}{(R+1)(R+2)}\geq \frac{1}{3}.$$
Thus, choosing $\eta<\frac{1}{3}$ yields trivially the implication of the lemma.\\
If $(\vartheta_{1},\vartheta_{2})\not=(+,-)$, we have $B(\lambda,\mu)=\frac{1}{X+3}$ and $A(\lambda,\mu)=\frac{1}{X+1}$ for $X\in\lbrace R,L,R+L\rbrace$ depending on the value of $(\vartheta_{1},\vartheta_{2})$. Thus, 
\begin{align*}
\frac{A(\lambda,\mu)-B(\lambda,\mu)}{B(\lambda,\mu)}=\frac{X+3}{X+1}-1=\frac{2}{X+1}\leq 4B(\lambda,\mu).
\end{align*}
Therefore, if $B(\lambda,\mu)\leq \frac{\epsilon}{4}$, 
$$\frac{A(\lambda,\mu)-B(\lambda,\mu)}{B(\lambda,\mu)}\leq \epsilon.$$
Finally, setting $\eta=\min\left(\frac{1}{3},\frac{\epsilon}{4}\right)$ yields the result.
\end{proof}
When the last run of $\lambda$ and the first run of $\mu$ remain bounded, the estimates of \cref{evaluateRun} are not so useful. By \cite[Proposition 6.9]{sawtooth}, it is however still possible to show that the distribution of $X_{\lambda}$ depends up to a small error only on the first cells of the composition, provided that the first run is not too big. We reformulate \cite[Proposition 6.9]{sawtooth} by using a convenient distance $D$ on the set of compositions. For two distinct compositions $\lambda,\mu$ of possibly different sizes with first run of the same size and same orientation, set
\begin{align*}
D(\lambda,&\mu)\\
=&\big(\sup \lbrace n\geq 1 \vert  \lambda_{\leq n}=\mu_{\leq n}\text{ and the first run of $\lambda$ is smaller than $n$}\rbrace\big)^{-1},
\end{align*}
and then set $D(\lambda,\lambda)=0$ and $D(\lambda,\mu):=1$ if the first run of $\lambda$ and the one of $\mu$ do not have the same length or same orientation. Remark that this distance is actually ultrametric, since it is readily seen that for three compositions $\lambda,\mu,\nu$ we have
$$D(\lambda,\nu)\leq \max \big(D(\lambda,\mu),D(\mu,\nu)\big).$$
From now on, we consider the set $\Lambda$ of compositions as a metric space with respect to this distance.
\begin{prop}[Prop.$\, 11$, \cite{sawtooth}]\label{evaluateIndependance}
The map 
$$F:(\Lambda,D)\longrightarrow (\mathcal{C}([0,1],\mathbb{R}),\Vert.\Vert_{\infty})$$ mapping a composition $\lambda$ to the cumulative distribution function $F_{\lambda}$ of $X_{\lambda}$ is uniformly continuous.
\end{prop}
We extend now Proposition \ref{evaluateIndependance} to a continuity result for the correlation function $\Delta$.
\begin{lem}\label{unifContDelta}
The map $\log\Delta:\Lambda\times \Lambda\longrightarrow \mathbb{R}$ is uniformly continuous with respect to the second argument. Moreover, for each $\epsilon>0$, there exists a constant of uniform continuity independent of the first argument.
\end{lem}
The above results mean that for all $\epsilon>0$, there exists $\delta>0$ such that for all $(\lambda,\mu_{1}),(\lambda,\mu_{2})\in\Lambda\times\Lambda$ with $D(\mu_1,\mu_2)\leq\delta$, we have 
$$\left\vert\log\Delta(\lambda,\mu_1)-\log\Delta(\lambda,\mu_2)\right\vert\leq\epsilon.$$
\begin{proof}
Let $\epsilon>0$ and let $\eta$ be the constant associated to $\epsilon$ in Lemma \ref{smallequivequiv}. Let $\delta<+\infty$ be a constant of uniform continuity associated to $\epsilon\eta$ for the function $F$ in \cref{evaluateIndependance} and let $\mu,\mu'\in \Lambda$ be two compositions such that $D(\mu,\mu')\leq \delta$. Then, Proposition \ref{evaluateIndependance} yields that $\Vert F_{\mu}-F_{\mu'}\Vert_{\infty}\leq \epsilon\eta$.\\
Let $\lambda$ be a composition. On the one hand, $B(\lambda,\mu)$ (resp. $B(\lambda,\mu')$) only depends on the orientation and the length of the last run of $\lambda$ and the first run of $\mu$ (resp. $\mu'$). On the other hand, since $D(\mu,\mu')\leq \delta$, the first run of $\mu$ and the one of $\mu'$ have same length and same orientation. Therefore, we can set $B:=B(\lambda,\mu)=B(\lambda,\mu')$. For the same reasons, we set $A:=A(\lambda,\mu)=A(\lambda,\mu')$.\\
If $B\leq \eta$, then, by Lemma \ref{smallequivequiv},
$$\frac{\Delta(\lambda,\mu)}{\Delta(\lambda,\mu')}\leq \frac{A}{B}\leq \frac{A-B+B}{B}\leq 1+\epsilon,$$
which yields 
$$\log\big(\Delta(\lambda,\mu)\big)-\log\big(\Delta(\lambda,\mu')\big)\leq \epsilon.$$
Exchanging $\mu$ and $\mu'$ gives finally
$$\left\vert\log\big(\Delta(\lambda,\mu)\big)-\log\big(\Delta(\lambda,\mu')\big)\right\vert\leq \epsilon.$$
Suppose now that $B\geq \eta$. By the choice of $\delta$ and by Proposition \ref{evaluateIndependance}, $\Vert F_{\mu}-F_{\mu'}\Vert_{\infty}\leq \eta\epsilon$. Since $\Vert 1-G_{\lambda}\Vert_{\infty}\leq 1$, this gives
$$\left\vert \Delta(\lambda,\mu')-\Delta(\lambda,\mu)\right\vert=\left\vert\int_{0}^{1}\big(1-G_{\lambda}(t)\big)\big(F_{\mu}(t)-F_{\mu'}(t)\big)dt\right\vert\leq \eta\epsilon.$$
Hence
$$\frac{\Delta(\lambda,\mu)}{\Delta(\lambda,\mu')}\leq \frac{\big\vert \Delta(\lambda,\mu)-\Delta(\lambda,\mu')\big\vert+\Delta(\lambda,\mu')}{\Delta(\lambda,\mu')}\leq 1+\frac{\eta\epsilon}{\eta}\leq 1+\epsilon,$$
which yields 
$$\log\big(\Delta(\lambda,\mu)\big)-\log\big(\Delta(\lambda,\mu')\big)\leq \epsilon.$$
By symmetry, we also have
$$\log\big(\Delta(\lambda,\mu')\big)-\log\big(\Delta(\lambda,\mu)\big)\leq \epsilon,$$
which concludes the proof.
\end{proof}
 
\section{Asymptotic law of $\xi^{w}_{1}$}\label{SectionAsymptotics}
This section is devoted to the asymptotic distribution of $\xi^{w}_{1}$ for $l(w)$ large.
\subsection{Preliminary results}
We first establish several estimates on the position of $1$ in a random standard filling. We write $\mathbb{P}_{\lambda}$ instead to $\mathbb{P}$ to emphasize that the probability is taken with respect to the random standard filling of a composition $\lambda$.

\begin{lem}\label{boundPosition1}
Let $\lambda$ be a composition of $n$, and $a<b$ be two peaks of $\lambda$. Then 
$$\mathbb{P}_{\lambda}(1\in \lambda_{>a,<b})\leq 3\frac{b-a}{n}.$$
\end{lem}
\begin{proof}
Since $1$ has to be located in a valley of $\lambda$, 
$$\mathbb{P}_{\lambda}(1\in \lambda_{>a,<b})=\sum_{\substack{v\in V\\a<v<b}}\mathbb{P}(1\in v).$$
From \cref{firstStepPos1}, for each $v\in V$, 
$$\mathbb{P}(1\in v)=\frac{1}{n}\frac{1}{\Delta(\lambda_{<v},\lambda_{>v})}.$$
Suppose that $v$ is the valley located at the intersection of the runs $s_{i}$ and $s_{i+1}$ (see  \cref{definitionRun+Extreme} for a precise definition of $s_{i}$ and $s_{i+1}$). By the bounds on $\Delta(\lambda_{<v},\lambda_{>v})$ from \cref{boundDelta},
$$\frac{1}{\Delta(\lambda_{<v},\lambda_{>v})}\leq l(s_{i})+l(s_{i+1})+3\leq 3(l(s_{i})+l(s_{i+1})),$$
where $l(s_{i})$ and $l(s_{i+1})$ are respectively the lengths of $s_{i}$ and $s_{i+1}$.
Summing the latter inequalities over all valleys between $a$ and $b$ yields
$$\mathbb{P}_{\lambda}(1\in \lambda_{>a,<b})\leq \frac{3}{n}\sum_{\substack{v\in V\\a<v<b}} l(s_{i})+l(s_{i+1})\leq \frac{3(b-a)}{n}.$$

\end{proof}
In the following lemma, we use the previous result to give a precise estimate on the probability that $1$ is located in a particular valley $v$ when the composition is large.
\begin{lem}\label{largeSlope}
Let $\epsilon>0$. There exists $n_{0}\geq 1$ such that if $\lambda$ is a composition of size $n$ larger than $n_{0}$, and $v\in \lambda$ is a valley with slope $\mathfrak{s}(v)=\llbracket a+1,b\rrbracket$, then
 
$$\frac{b-a}{n}-\epsilon\leq\mathbb{P}_{\lambda}(1\in v)\leq \frac{b-a}{n}+\epsilon.$$

\end{lem}
\begin{proof}
Let us first find $n_{\epsilon}\geq 1$ such that if $b-a\geq n_{\epsilon}$ and $n\geq n_{\epsilon}$, then 
$$\frac{b-a}{n}-\epsilon\leq\mathbb{P}_{\lambda}(1\in v)\leq \frac{b-a}{n}+\epsilon.$$
Suppose that $b-a\geq 2$. Then, since $v$ is a valley, $v$ belongs to two runs $s_{i}$ and $s_{i+1}$, one of them having a length greater or equal to $2$. Assume without loss of generality that $l(s_{i+1})\geq 2$. This implies that the first run of $\lambda_{>v}$ is increasing. Let $L$ denote the length of the last run of $\lambda_{<v}$, and $R$ the length of the first run of $\lambda_{>v}$.

 If $l(s_{i})=1$, the last run of $\lambda_{<v}$ is increasing. Moreover, in this case, $b-a=R+1$. Thus, \cref{boundDelta} yields
 $$\frac{1}{b-a+2}\leq\Delta(\lambda_{<v},\lambda_{>v})\leq \frac{1}{b-a}.$$
Hence, independently of $L$, there exists $n_{1}$ such that if $l(s_{i})=1$ and $b-a\geq n_{1}$, then $$(1-\epsilon)(b-a)\leq\frac{1}{\Delta(\lambda_{<v},\lambda_{>v})}\leq (1+\epsilon)(b-a).$$

If $l(s_{i})>1$, the last run of $\lambda_{<v}$ is decreasing. Then, $b-a=L+R$ and \cref{boundDelta} yields
$$\frac{1}{b-a+3}\leq \Delta(\lambda_{>v},\lambda_{<v})\leq \frac{1}{b-a+1}.$$
There exists $n_{2}$ such that if $l(s_{i})>1$ and $b-a\geq n_{2}$, then 
$$(1-\epsilon)(b-a)\leq\frac{1}{\Delta(\lambda_{<v},\lambda_{>v})}\leq (1+\epsilon)(b-a).$$
Set $n_{\epsilon}=\max(n_{1},n_{2})$, and suppose that $b-a\geq n_{\epsilon}$ and $n\geq n_{\epsilon}$. From \cref{firstStepPos1}, $\mathbb{P}_{\lambda}(1\in v)=\frac{1}{n\Delta(\lambda_{<v},\lambda_{>v})}$ and thus
$$(1-\epsilon)\frac{b-a}{n}\leq \mathbb{P}_{\lambda}(1\in v)\leq (1+\epsilon)\frac{b-a}{n},$$
which is the expected property for $n_{\epsilon}$.

Define now $n_{0}=\frac{3n_{\epsilon}}{\epsilon}$ and let $\lambda$ be a composition of size larger than $n_{0}$. Let $v$ be a valley of $\lambda$ with slope $\llbracket a+1,b\rrbracket$. If $b-a\geq n_{\epsilon}$, the first part of the proof yields
$$(1-\epsilon)\frac{b-a}{n}\leq \mathbb{P}_{\lambda}(1\in v)\leq (1+\epsilon)\frac{b-a}{n}.$$
If $b-a\leq n_{\epsilon}$, then $\frac{b-a}{n}\leq \frac{\epsilon}{3}$ and by \cref{boundPosition1} $\mathbb{P}(1\in v)\leq \epsilon$. Therefore, in any case,
$$\frac{b-a}{n}-\epsilon\leq \mathbb{P}_{\lambda}(1\in v)\leq \frac{b-a}{n}+\epsilon.$$
\end{proof}
The following lemma estimates the probability that the integer $1$ is situated in some initial part of a composition.
\begin{lem}\label{largeSubcompo}
Let $\epsilon>0$. There exists $n_{0}'\geq 1$ such that if $\lambda$ is a composition of size $n$ larger than $n_{0}'$ and $b$ is a peak of $\lambda$, then
$$\frac{b}{n}-\epsilon\leq\mathbb{P}_{\lambda}(1\in\lambda_{<b})\leq\frac{b}{n}+\epsilon.$$
\end{lem}
\begin{proof}
We suppose without loss of generality that $\epsilon\leq \frac{1}{2}$, so that $1-\epsilon\leq\exp(\epsilon)\leq 1+2\epsilon$.\\
Let $n_{\epsilon}$ be such that $\frac{1}{n_{\epsilon}}$ is a constant of uniform continuity associated to $\epsilon$ for the function $\log\Delta$ in \cref{unifContDelta}, and let $n_{0}'$ be such that $\frac{n_{\epsilon}}{n_{0}'}\leq \epsilon$. Suppose that $\lambda$ is a composition of size $n$ larger than $n_{0}'$.
We divide the proof in two cases. 

Suppose first that $b\geq n-n_{\epsilon}$, which implies that $\frac{n-b}{n}\leq \frac{n_{\epsilon}}{n}\leq\epsilon$. Then,
 $$\mathbb{P}_{\lambda}(1\in \lambda_{<b})=1-\mathbb{P}_{\lambda}(1\in\lambda_{>b}),$$
and by \cref{boundPosition1}, 
$$\mathbb{P}_{\lambda}(1\in\lambda_{>b})\leq \frac{3(n-b)}{n}\leq  3\epsilon.$$
Therefore,
$$\mathbb{P}_{\lambda}(1\in \lambda_{<b})\geq 1-3\epsilon\geq \frac{b}{n}-3\epsilon,$$
and 
$$\mathbb{P}_{\lambda}(1\in\lambda_{<b})\leq 1\leq \frac{b}{n}+\frac{n-b}{n}\leq \frac{b}{n}+\epsilon.$$
If $b\leq n-n_{\epsilon}$, set $t=b+n_{\epsilon}$. Then, for each valley $v\leq b$, $D(\lambda_{>v},\lambda_{>v,<t})\leq\frac{1}{n_{\epsilon}}$ and hence \cref{unifContDelta} yields after exponentiation that 
$$(1-\epsilon)\Delta\big(\lambda_{<v},\lambda_{>v,<t}\big)^{-1}\leq \Delta\big(\lambda_{<v},\lambda_{>v}\big)^{-1}\leq (1+2\epsilon)\Delta\big(\lambda_{<v},\lambda_{>v,<t}\big)^{-1},$$
where we have used that $1-\epsilon\leq \exp(\epsilon)\leq 1+2\epsilon$ when $\epsilon\leq \frac{1}{2}$. By the expression of $\mathbb{P}_{\lambda}(1\in v)$ from \cref{firstStepPos1}, we thus have
$$(1-\epsilon)\frac{t}{n}\mathbb{P}_{\lambda_{<t}}(1\in v)\leq\mathbb{P}_{\lambda}(1\in v)\leq (1+2\epsilon)\frac{t}{n}\mathbb{P}_{\lambda_{<t}}(1\in v).$$
Summing the latter inequalities on all valleys smaller than $b$, we get 
$$(1-\epsilon)\frac{t}{n}\mathbb{P}_{\lambda_{<t}}(1\in\lambda_{<b})\leq \mathbb{P}_{\lambda}(1\in \lambda_{<b})\leq (1+2\epsilon)\frac{t}{n}\mathbb{P}_{\lambda_{<t}}(1\in\lambda_{<b}).$$
Since $\frac{t}{n}=\frac{b+n_{\epsilon}}{n}\leq \frac{b}{n}+\epsilon$, the latter inequality already yields 
$$\mathbb{P}_{\lambda}(1\in \lambda_{<b})\leq \frac{b}{n}+4\epsilon.$$
On the other hand,, by \cref{boundPosition1}, 
$$\mathbb{P}_{\lambda_{<t}}(1\in\lambda_{<b})=1-\mathbb{P}_{\lambda_{<t}}(1\in \lambda_{>b})\geq 1-\frac{3n_{\epsilon}}{t}.$$
Therefore,
$$\mathbb{P}_{\lambda}(1\in \lambda_{<b})\geq(1-\epsilon)\frac{t}{n}(1-\frac{3n_{\epsilon}}{t})\geq  (1-\epsilon)(\frac{b}{n}+\frac{n_{\epsilon}}{n}-\frac{3n_{\epsilon}}{n})\geq \frac{b}{n}-2\epsilon.$$
Finally, we have proven that 
$$\frac{b}{n}-4\epsilon\leq\mathbb{P}_{\lambda}(1\in\lambda_{<b})\leq \frac{b}{n}+4\epsilon,$$
for $\lambda$ larger than $n'_{0}$.
\end{proof}

\subsection{Convergence to a uniform distribution}
We prove now that $\xi^{w}_{1}$, the averaged coordinate of the cell containing $1$ in $\sigma_{w}$ (see \cref{defiXi}), converges in law towards a uniform distribution on $[0,1]$.
\begin{prop}\label{casek=1}
Let $\epsilon>0$. There exists $N\geq 1$ such that for any word $w$ longer that $N$,
$$\Vert F_{\xi^{w}_{1}}-\Id_{[0,1]}\Vert_{\infty}\leq \epsilon,$$
where $F_{\xi^{w}_{1}}$ denotes the cumulative distribution function of $\xi_{1}^{w}$.
\end{prop}
\begin{proof}
Let $0<\epsilon\leq 1$. Since $F_{\xi^{w}_{1}}$ and $\Id_{[0,1]}$ are increasing functions, it is enough to prove that for $s\in [0,1]$ and for $w$ long enough,  
$$\big\vert \mathbb{P}(\xi^{w}_{1}\in [0,s])- s\big\vert\leq \epsilon.$$ 
Let $0<s<1$ and let $n_{0}$, $n_{0}'$ be the constants respectively given by \cref{largeSlope} and by \cref{largeSubcompo} for $\epsilon$. Set $N=\max(1/s,1/\epsilon,n_0,n_{0}')$, and let $w$ be a word of size $n\geq N$, with $\lambda$ its corresponding composition. Let $v_{0}$ denote the last valley of $\lambda$ such that the associated slope $\llbracket a+1,b\rrbracket$ intersects $[0,ns]$, namely $[a+1,b+1]\cap [0,ns]\not=\emptyset$: since $\frac{1}{n}<s$, $v_{0}$ always exists. 

Note that $a$ is either zero or a peak of $\lambda$ and if $1\in \lambda_{< a}$, then $\xi_{1}^{w}\in \left[0,\frac{a}{n}\right[\subset [0,s]$. Thus,
\begin{equation}\label{equation1}\mathbb{P}(\xi^{w}_{1}\in [0,s])=\mathbb{P}_{\lambda}(1\in \lambda_{<a})+\mathbb{P}\big(1\in v_{0}\cap \xi_{1}^{w}\in [0,s]\big).
\end{equation}
On the one hand, because $n\geq n_{0}'$, \cref{largeSubcompo} gives
\begin{equation}\label{equation2}\frac{a}{n}-\epsilon\leq\mathbb{P}_{\lambda}(1\in \lambda_{<a})\leq \frac{a}{n}+\epsilon.
\end{equation}
On the other hand by independence between $\sigma_{w}$ and the random variable $U_{1}$ in the definition of $\xi_{1}^{w}$ (see \cref{defAverageCoord}), 
\begin{align*}
\mathbb{P}\big(1\in v_{0}\cap \xi_{1}^{w}\in [0,s]\big)=&\mathbb{P}_{\lambda}(1\in v_{0})\frac{\Leb\big([0,s]\cap [a/n,b/n]\big)}{(b-a)/n}\\
=&\mathbb{P}_{\lambda}(1\in v_{0})\left(\frac{sn-a}{b-a}\wedge 1\right).
\end{align*}
Since $n\geq n_0$, \cref{largeSlope} yields
\begin{equation}\label{equation3}
\frac{b-a}{n}-\epsilon\leq \mathbb{P}_{\lambda}(1\in v_0)\leq \frac{b-a}{n}+\epsilon.
\end{equation}
Applying \eqref{equation2} and \eqref{equation3} to \eqref{equation1}, we get
\begin{equation}\label{equation4}\frac{a}{n}+\frac{b-a}{n}\left(\frac{sn-a}{b-a}\wedge 1\right)-2\epsilon\leq \mathbb{P}(\xi^{w}_{1}\in [0,s])\leq \frac{a}{n}+\frac{b-a}{n}\left(\frac{sn-a}{b-a}\wedge 1\right)+2\epsilon.
\end{equation}
First, 
\begin{equation}\label{equation5}\frac{b-a}{n}\left(\frac{sn-a}{b-a}\wedge 1\right)\leq \frac{b-a}{n}\;\frac{sn-a}{b-a}\leq s-\frac{a}{n}.
\end{equation}
Then, since $sn<b+1$ (otherwise the slope of the valley following $v_{0}$ would intersect $[0,ns]$), we have
$$\left(\frac{sn-a}{b-a}\wedge 1\right)\geq \frac{sn-a-1}{b-a},$$ 
and the inequality $\frac{1}{n}\leq \epsilon$ yields
\begin{equation}\label{equation6}\frac{b-a}{n}\left(\frac{sn-a}{b-a}\wedge 1\right)\geq \frac{b-a}{n}\,\frac{sn-a-1}{b-a}\geq s-\frac{a}{n} -\epsilon.
\end{equation}
Using \eqref{equation5} and \eqref{equation6} in \eqref{equation4}, we get
$$s-3\epsilon\leq \mathbb{P}(\xi^{w}_{1}\in [0,s])\leq s+3\epsilon,$$
which concludes the proof.

\end{proof}

\section{Martin boundary of $\mathcal{Z}$}\label{SectionmartinBoundary}
This section is devoted to the proof of \cref{mainresult1}, yielding the identification of the Martin boundary of $\mathcal{Z}$ with its minimal boundary.  
\subsection{Generalization of \cref{casek=1}}
Recall that $\xi^{w}(k)$ denotes the random vector $(\xi^{w}_{i})_{1\leq i\leq k}$ (see \cref{defiXi}). We generalize here the result of \cref{casek=1} by proving that, for each $k\geq 1$, $\xi^{w}(k)$ converges in law towards a tuple of independent random variables uniformly distributed on $[0,1]$ as the length of $w$ goes to infinity. Let $F_{k}$ be the cumulative distribution function of a $k$-tuple of independent uniform random variables on $[0,1]$, which satisfies $F_{k}(s_{1},\ldots,s_{k})=\prod_{i=1}^{k} s_{i}$ for $s_{1},\ldots,s_{k}\in[0,1]$. 
\begin{prop}\label{generalCase}
Let $\epsilon>0$. There exists $N_{k}\in\mathbb{N}$ such that for any word $w$ longer than $N_{k}$,
$$\Vert F_{\xi^{w}(k)}-F_{k}\Vert_{\infty}\leq \epsilon,$$
where $F_{\xi^{w}(k)}$ denotes the cumulative distribution function of $\xi^{w}(k)$.
\end{prop}
The proof of this generalization is done by induction, by conditioning $\xi^{w}_{k}$ on $\xi^{w}(k-1)$. We first prove several preliminary results.  Throughout this section and unless stated otherwise, $k\geq 2$ is a fixed integer. Let $\lambda_{1},\ldots,\lambda_{r+1}$, $\mu_{1},\ldots,\mu_{r}$ be compositions such that $\sum_{i=1}^{r}\vert \mu_{i}\vert=k-1$, where $\vert \mu_{i}\vert$ denotes the size of the composition $\mu_{i}$. Set 
\begin{equation}\label{constructionLambda}
\lambda=\lambda_{1}-\mu_{1}+\lambda_{2}-\cdots-\mu_{r}+\lambda_{r+1}.
\end{equation}
We denote by $w_{i}$ the word in $\mathcal{Z}$ corresponding to $\lambda_i$, and we denote by $S:=\{j_{1},\ldots,j_{k-1}\}$ the set of cells of $\lambda$ which correspond to any of the compositions $\mu_{i}$, $1\leq i\leq r$, in \eqref{constructionLambda}. Let $\vec{i}=(i_{1},\ldots,i_{k-1})$ be a permutation of the set $S$: formally, there exists $\sigma\in S_{k-1}$ with $i_{t}=j_{\sigma(t)}$. Denote by $\mathcal{X}_{\vec{i}}$ the event $\big\lbrace \sigma_{w}^{-1}(1)=i_{1},\ldots,\sigma_{w}^{-1}(k-1)=i_{k-1}\big\rbrace$ and suppose that $\mathbb{P}_{\lambda}\big(\mathcal{X}_{\vec{i}}\big)\not=0$.\\
\indent Conditioned on $\mathcal{X}_{\vec{i}}$, the random filling of $\lambda$ consists in sampling a uniformly random multiset $\vec{R}=(R_{1},\ldots,R_{r+1})$ of cardinality $(\vert \lambda_{1}\vert,\ldots,\vert \lambda_{r+1}\vert)$ among $\llbracket k,n\rrbracket$, and then independently filling each subcomposition $\lambda_{1},\ldots,\lambda_{r+1}$ respectively with $R_{1},\ldots, R_{r+1}$. If $v$ is a cell of $\lambda_i$, denote by $\tilde{v}$ the corresponding cell in $\lambda$: namely, $\tilde{v}=v+a_{i}$. Since $k$ is the lowest element of $\llbracket k,n\rrbracket$, for $v\in\lambda_{i}$, $\mathbb{P}_{\lambda}\big(k\in \tilde{v}\vert \mathcal{X}_{\vec{i}}\big)\not=0$ if and only if $v$ is a valley of $\lambda_{i}$, and if this is the case,
$$\mathbb{P}_{\lambda}\big(k\in \tilde{v}\vert \mathcal{X}_{\vec{i}}\big)=\mathbb{P}_{(R_{1},\ldots,R_{r+1})}(k\in R_{i})\mathbb{P}_{\lambda_{i}}(1\in v).$$

The following result shows that the law of $\xi^{w}_{k}$ conditioned on $\mathcal{X}_{\vec{i}}$ can be deduced from the knowledge of the distribution of $\xi^{w_{i}}_{1}$ for each subcomposition $\lambda_{i}$. Recall that the Lévy distance $L(A,B)$ between two real random variables $A$ and $B$ is defined as 
\begin{align*}
L(A,B)=\inf\lbrace &\epsilon>0\vert \mathbb{P}(A\leq s-\epsilon)-\epsilon\leq \mathbb{P}(B\leq s)\leq\mathbb{P}(A\leq s+\epsilon)+\epsilon\\
& \text{ for all } s\in\mathbb{R}\rbrace.
\end{align*}

\begin{lem}\label{reducedCase}
Let $\mathfrak{r}$ be a random variable on $\lbrace 1,\ldots,r+1\rbrace$ such that $\mathbb{P}(\mathfrak{r}=i)=\frac{\vert\lambda_{i}\vert}{n-k+1}$, and suppose that the random variables $\sigma_{w_{i}},1\leq i\leq r+1,$ and $\mathfrak{r}$ are independent. Then,
$$L\left(\big(\xi^{w}_{k}\big\vert \mathcal{X}_{\vec{i}}\big),\left(\frac{a_{\mathfrak{r}}}{n}+\frac{\vert\lambda_{\mathfrak{r}}\vert}{n} \xi^{w_{\mathfrak{r}}}_{1}\right)\right)\leq\frac{2k}{n}.$$
\end{lem}
\begin{proof}
We show the result by coupling $X:=\left(\xi^{w}_{k}\big\vert\mathcal{X}_{\vec{i}}\right)$ and $Y:=\frac{a_{\mathfrak{r}}}{n}+\frac{\vert \lambda_{\mathfrak{r}}\vert}{n}\xi^{w_{\mathfrak{r}}}_{1}$ on a same probability space. For $1\leq i\leq r+1$, denote by $\llbracket x^{\lambda_{i}}(c)+1,y^{\lambda_{i}}(c)\rrbracket$ the slope of a cell $c$ in $\lambda_{i}$ and by $\llbracket x^{\lambda}(\tilde{c})+1, y^{\lambda}(\tilde{c})\rrbracket$ the slope of the corresponding cell $\tilde{c}=a_{i}+c$ in $\lambda$. Remark that $ x^{\lambda}(\tilde{c})=x^{\lambda_{i}}(c)+a_{i}$ if $c$ is not in the first run of $\lambda_{i}$, and else, due to the presence of the composition $\mu_{i-1}$ left to $\lambda_{i}$,
$$\vert x^{\lambda}(\tilde{c})-(x^{\lambda_{i}}(c)+a_{i})\vert \leq \vert \mu_{i-1}\vert\leq k-1.$$
Likewise, $ y^{\lambda}(\tilde{c})=y^{\lambda_{i}}(c)+a_{i}$ if $c$ is not in the last run of $\lambda_{i}$, and else
$$\vert  y^{\lambda}(\tilde{c})-(y^{\lambda_{i}}(c)+a_{i})\vert \leq \vert \mu_{i}\vert\leq k-1.$$
Denote by $Z_{i}$ the random variable $\sigma_{w_{i}}^{-1}(1)$. From the discussion preceding the lemma,
\begin{equation}\label{equalLaw}\left(\sigma_{w}^{-1}(k)\big\vert\mathcal{X}_{\vec{i}}\right)=_{\law}a_{\mathfrak{r}}+\sigma_{w_{\mathfrak{r}}}^{-1}(1)=\tilde{Z}_{\mathfrak{r}}.
\end{equation} 
Modifying if necessary the probability spaces, we can consider that the two previous random variables are equal. Recall that $X:=\left(\xi^{w}_{k}\big\vert\mathcal{X}_{\vec{i}}\right)$ is defined as
$$X=\frac{x^{\lambda}_{k}}{n}+\frac{y^{\lambda}_{k}-x^{\lambda}_{k}}{n}U_{k},$$
where $\llbracket x^{\lambda}_{k}+1,y^{\lambda}_{k}\rrbracket$ denotes the slope of $\left(\sigma_{w}^{-1}(k)\big\vert\mathcal{X}_{\vec{i}}\right)$ in $\lambda$ and $U_{k}$ is a uniform random variable on $[0,1]$ independent from $\sigma_{w}$(see \cref{defiXi}). By \eqref{equalLaw},
$$X=\frac{ x^{\lambda}(\tilde{Z}_{\mathfrak{r}})}{n}+\frac{ y^{\lambda}(\tilde{Z}_{\mathfrak{r}})- x^{\lambda}(\tilde{Z}_{\mathfrak{r}})}{n}U_{k}.$$
Likewise, we can define $\xi^{w_{\mathfrak{r}}}_{1}$ using the same random variable $U_{k}$ with
$$\xi^{w_{\mathfrak{r}}}_{1}=\frac{x^{\lambda_{\mathfrak{r}}}(Z)}{\vert \lambda_{\mathfrak{r}}\vert}+\frac{y^{\lambda_{\mathfrak{r}}}(Z_{\mathfrak{r}})-x^{\lambda_{\mathfrak{r}}}(Z)}{\vert \lambda_{\mathfrak{r}}\vert}U_{k},$$
which yields
$$Y=\frac{a_{\mathfrak{r}}}{n}+\frac{\vert\lambda_{\mathfrak{r}}\vert}{n}\xi^{\lambda_{\mathfrak{r}}}_{1}=\frac{a_{\mathfrak{r}}+x^{\lambda_{\mathfrak{r}}}(Z)}{n}+\frac{y^{\lambda_{\mathfrak{r}}}(Z_{\mathfrak{r}})-x^{\lambda_{\mathfrak{r}}}(Z_{\mathfrak{r}})}{n}U_{k}.$$
Therefore,
\begin{align*}
\vert Y-X\vert =&\left\vert (1-U_{k})\frac{a_{\mathfrak{r}}+x^{\lambda_{\mathfrak{r}}}(Z_{\mathfrak{r}})- x^{\lambda}(\tilde{Z}_{\mathfrak{r}})}{n}+U_{k}\frac{a_{\mathfrak{r}}+y^{\lambda_{\mathfrak{r}}}(Z_{\mathfrak{r}})- y^{\lambda}(\tilde{Z}_{\mathfrak{r}})}{n}\right\vert\\
\leq& \left\vert\frac{a_{\mathfrak{r}}+x^{\lambda_{\mathfrak{r}}}(Z_{\mathfrak{r}})- x^{\lambda}(\tilde{Z}_{\mathfrak{r}})}{n}\right\vert +\left\vert \frac{a_{\mathfrak{r}}+y^{\lambda_{\mathfrak{r}}}(Z_{\mathfrak{r}})- y^{\lambda}(\tilde{Z}_{\mathfrak{r}})}{n}\right\vert \leq \frac{2k}{n}.
\end{align*}
Hence, for $s\in[0,1]$, 
$$\mathbb{P}(X\leq s)\leq \mathbb{P}\left(Y\leq s+\frac{2k}{n}\right)\quad\text{and}\quad \mathbb{P}(X\leq s)\geq \mathbb{P}\left(Y\leq s-\frac{2k}{n}\right),$$
which implies that $L(X,Y)\leq \frac{2k}{n}$.
\end{proof}
In the following lemma, $F_{(\xi_{k}^{w}\vert \mathcal{X}_{\vec{i}})}$ denotes the cumulative distribution function of $(\xi_{k}^{w}\vert \mathcal{X}_{\vec{i}})$.
\begin{lem}\label{mainPart}
Let $\epsilon>0$. There exists $n_{k}\geq 1$ such that for $w\in \mathcal{Z}$ of length $n\geq n_{k}$ and for $\vec{i}\in\llbracket 1,n\rrbracket ^{k-1}$ such that $\mathbb{P}_{\lambda}(\mathcal{X}_{\vec{i}})\not=0$,
$$\Vert F_{(\xi_{k}^{w}\vert \mathcal{X}_{\vec{i}})}-\Id_{[0,1]}\Vert_{\infty}\leq \epsilon.$$
\end{lem}
\begin{proof}
Set $n_{k}=\max(\frac{N}{\epsilon},\frac{k}{\epsilon})+(k-1)$, where $N$ is given by \cref{casek=1}, and suppose that $w\in\mathcal{Z}$ is a word of length $n\geq n_{k}$. Let $s\in[0,1]$ with $s\geq \frac{k}{n}$ and let $i$ be the last integer such that $[0,ns]\cap[a_{i},b_{i}]\not=0$ (recall that $a_{i}$ and $b_{i}$ are such that $\lambda_{i}=\lambda_{>a_{i},\leq b_{i}}$). Such an integer exists because $a_{1}\leq k$ and $ns\geq k$. Note that if $j<i$, then $b_{j}<a_{i}<ns$ and thus $\frac{a_{j}}{n}+\frac{b_{j}-a_{j}}{n}\xi_{1}^{w_{j}}\leq s$; similarly, if $j>i$, then $\frac{a_{j}}{n}+\frac{b_{j}-a_{j}}{n}\xi_{1}^{w_{j}}>s$. Set 
$$Y:=\frac{a_{\mathfrak{r}}}{n}+\frac{\vert\lambda_{\mathfrak{r}}\vert}{n} \xi^{w_{\mathfrak{r}}}_{1}.$$ 
Then,
\begin{align*}
\mathbb{P}(Y\leq s)=&\mathbb{P}(\mathfrak{r}<i)+\mathbb{P}\left[\mathfrak{r}=i\cap\left(\frac{a_{i}}{n}+\frac{b_{i}-a_{i}}{n}\xi_{1}^{w_{i}}\leq s\right)\right]\\
=&\mathbb{P}(\mathfrak{r}<i)+\mathbb{P}(\mathfrak{r}=i)\mathbb{P}\left(\xi_{1}^{w_{i}}\leq \frac{ns-a_{i}}{b_{i}-a_{i}}\right):=C_{1}+C_{2},
\end{align*}
where $C_{1}=\mathbb{P}(\mathfrak{r}<i)$ and $C_{2}=\mathbb{P}\left(\mathfrak{r}=i)\mathbb{P}(\xi_{1}^{w_{i}}\leq \frac{ns-a_{i}}{b_{i}-a_{i}}\right)$. Note first that a counting argument shows that $\mathbb{P}(\mathfrak{r}=j)=\frac{b_{j}-a_{j}}{n-k+1}$, so that
\begin{align*}
\left\vert\mathbb{P}(\mathfrak{r}<i)-\frac{a_{i}}{n-k+1}\right\vert=&\left\vert \frac{1}{n-k+1}\sum_{j<i}(b_{j}-a_{j})-\frac{a_{i}}{n-k+1}\right\vert\\
=&\frac{\#\lbrace 1\leq t\leq k-1, i_{t}\leq a_{i}\rbrace}{n-k+1}\leq\frac{k}{n-k+1}.
\end{align*}
Since $n\geq \frac{k}{\epsilon}+(k-1)$,
\begin{equation}\label{step2}
\left\vert C_{1}-\frac{a_{i}}{n-k+1}\right\vert\leq \epsilon.
\end{equation}
If $ns>b_{i}$, then $\mathbb{P}\left(\xi_{1}^{w_{i}}\leq \frac{ns-a_{i}}{b_{i}-a_{i}}\right)=1$. By the choice of $i$, $ns\leq b_{i}+\vert \mu_{i}\vert$ which yields $\vert ns-b_{i}\vert \leq k$. Thus, since $n\geq \frac{k}{\epsilon}+(k-1)$,
\begin{align*}
\left\vert \mathbb{P}(\mathfrak{r}=i)\mathbb{P}\left(\xi_{1}^{w_{i}}\leq \frac{ns-a_{i}}{b_{i}-a_{i}}\right)-\frac{ns-a_{i}}{n-k+1}\right\vert=&\left\vert \frac{b_{i}-a_{i}}{n-k+1}-\frac{ns-a_{i}}{n-k+1}\right\vert\\
=&\frac{ns-b_{i}}{n-k+1}\leq \frac{k}{n-k+1}\leq \epsilon.
\end{align*}
If $ns\leq b_{i}$ and $b_{i}-a_{i}\leq N$, since $n\geq \frac{N}{\epsilon}+(k-1)$, 
$$\left\vert \frac{b_{i}-a_{i}}{n-k+1}\mathbb{P}\left(\xi_{1}^{w_{i}}\leq \frac{ns-a_{i}}{b_{i}-a_{i}}\right)-\frac{ns-a_{i}}{n-k+1}\right\vert\leq 2\left \vert \frac{b_{i}-a_{i}}{n-k+1}\right\vert\leq 2\epsilon.$$
If $ns\leq b_{i}$ and $b_{i}-a_{i}> N$, then $w_{i}$ is longer than $N$ and \cref{casek=1} yields that
\begin{align*}
\bigg\vert \frac{b_{i}-a_{i}}{n-k+1}\mathbb{P}\left(\xi_{1}^{w_{i}}\leq \frac{ns-a_{i}}{b_{i}-a_{i}}\right)&-\frac{ns-a_{i}}{n-k+1}\bigg\vert\\
= &\frac{b_{i}-a_{i}}{n-k+1} \left\vert F_{\xi_{1}^{w_{i}}}\left(\frac{ns-a_{i}}{b_{i}-a_{i}}\right)-\frac{ns-a_{i}}{b_{i}-a_{i}}\right\vert \leq \epsilon.
\end{align*}
Thus, in any case,
\begin{equation}\label{step3}
\left\vert C_{2}-\frac{ns-a_{i}}{n-k+1}\right\vert\leq 2\epsilon.
\end{equation}
Therefore, 
\begin{align*}
\vert \mathbb{P}(Y\leq s)-s\vert\leq&\left\vert C_{1}-\frac{a_{i}}{n-k+1}\right\vert +\left\vert C_{2}- \frac{ns-a_{i}}{n-k+1}\right\vert +\left\vert \frac{k-1}{n-k+1}s\right\vert \\
\leq& \epsilon+2\epsilon+\epsilon\leq 4\epsilon.
\end{align*}
By \cref{reducedCase}, $L\left(\big(\xi^{w}_{k}\vert\mathcal{X}_{\vec{I}}\big),Y\right)\leq \frac{2k}{n}$. Thus, for $n\geq n_{k}$ and $s\geq \frac{3k}{n}$, since $\frac{k}{n}\leq \epsilon$,
$$\mathbb{P}\big((\xi^{w}_{k}\vert\mathcal{X}_{\vec{i}}\big)\leq s)\leq \mathbb{P}(Y\leq s+\frac{2k}{n})+\frac{2k}{n}\leq s+8\epsilon,$$
and similarly $\mathbb{P}\big((\xi^{w}_{k}\vert\mathcal{X}_{\vec{i}}\big)\leq s)\geq s-8\epsilon$, so that finally for $s\geq \frac{3s}{n}$,
$$\left\vert F_{\xi^{w}_{k}\vert \mathcal{X}_{\vec{i}}}(s)-s\right\vert \leq 8\epsilon.$$
If $s\leq \frac{3k}{n}$, then
\begin{align*}
\left\vert F_{\xi^{w}_{k}\vert \mathcal{X}_{\vec{i}}}(s)-s\right\vert \leq &  F_{\xi^{w}_{k}\vert \mathcal{X}_{\vec{i}}}(\frac{3k}{n})+\frac{3k}{n}\\
\leq & \frac{3k}{n}+8\epsilon+\frac{3k}{n}\leq 14\epsilon.
\end{align*}
\end{proof}

The proof of \cref{generalCase} is now done by induction.
\begin{proof}[Proof of \cref{generalCase}]
Let us prove by induction on $k\geq 1$ that for $\epsilon>0$, there exists $N_{k}\in\mathbb{N}$ such that for $n\geq N_{k}$ and $w\in \mathcal{Z}$ with $l(w)= n-1$,
$$\Vert F_{\xi^{w}(k)}-F_{k}\Vert_{\infty}\leq \epsilon.$$
For $k=1$, the result is given by \cref{casek=1}. Let $k>1$ and suppose that the result is proven for $k-1$. There exists $N_{k-1}\geq 1$ such that, for $n\geq N_{k-1}$,
$$\Vert F_{\xi^{w}(k-1)}-F_{k-1}\Vert_{\infty}\leq \epsilon.$$
 Suppose that $n\geq \max(N_{k-1},n_{k})$, where $n_{k}$ is given in \cref{mainPart}. For $s_{1},\ldots,s_{k}\in[0,1]$,
\begin{equation}\label{expresCumuFunction}
F_{\xi^{w}(k)}(s_{1},\ldots,s_{k})=\mathbb{E}\left(\mathbb{P}\big(\xi^{w}_{k}\leq s_{k}\vert \xi^{w}(k-1))\mathbf{1}_{\xi^{w}_{1}\leq s_{1},\ldots,\xi^{w}_{k-1}\leq s_{k-1}}\right).
\end{equation}
By \cref{fixedPosition},  different values of $\big(\sigma_{w}^{-1}(i)\big)_{1\leq i\leq k-1}$ yield disjoint supports for $\xi^{w}(k-1)$. Thus, $\big(\sigma_{w}^{-1}(i)\big)_{1\leq i\leq k-1}$ is $\xi^{w}(k-1)$-measurable. Hence,
\begin{equation}\label{measureImplyInv}
\big(\xi^{w}_{k}\big\vert \xi^{w}(k-1)\big)=\left(\xi^{w}_{k}\Big\vert \xi^{w}(k-1),\big(\sigma_{w}^{-1}(i)\big)_{1\leq i\leq k-1}\right).
\end{equation}
Recall that 
$$\xi^{w}(k)=\left(\frac{x\big(\sigma_{w}^{-1}(i)\big)}{n}+\frac{y\big(\sigma_{w}^{-1}(i)\big)-x\big(\sigma^{-1}_{w}(i)\big)}{n}U_{i}\right)_{1\leq i\leq k},$$
where $(U_{i})_{i\geq 1}$ is a sequence of independent random variables uniformly distributed on $[0,1]$ and independent of $\sigma_{w}$. Thus, conditioned on the value of $\big(\sigma_{w}^{-1}(i)\big)_{1\leq i\leq k-1}$, the random variables $\xi^{w}_{k}$ and $\xi^{w}(k-1)$ are independent. Therefore, by Doob's conditional independence property \cite[Proposition 5.6]{kallenberg2006foundations},
\begin{equation}\label{DoobIndepProper}
\left(\xi^{w}_{k}\Big\vert \xi^{w}(k-1),\big(\sigma_{w}^{-1}(i)\big)_{1\leq i\leq k-1}\right)=\left(\xi^{w}_{k}\Big\vert \big(\sigma_{w}^{-1}(i)\big)_{1\leq i\leq k-1}\right).
\end{equation}
Recall that $\mathcal{X}_{\vec{j}}$ denotes the event $\left\lbrace\big(\sigma_{w}^{-1}(i)\big)_{1\leq i\leq k-1}=\vec{j}\right\rbrace$ for $\vec{j}\in\llbracket 1,n\rrbracket^{k-1}$. By \cref{mainPart}, since $n\geq n_{k}$,
$$\big\vert\mathbb{P}\big(\xi^{w}_{k}\leq s_{k}\big\vert \mathcal{X}_{\vec{j}}\big)-s_{k}\big\vert \leq \epsilon$$
for all $\vec{j}$ such that $\mathbb{P}(\mathcal{X}_{\vec{j}}>0)$. This implies that
$$\left \vert \mathbb{P}\left(\xi^{w}_{k}\leq s_{k}\Big\vert \big(\sigma_{w}^{-1}(i)\big)_{1\leq i\leq k-1}\right)-s_{k}\right\vert \leq \epsilon.$$
Hence, by \eqref{measureImplyInv} and \eqref{DoobIndepProper}, for $n\geq n_{k}$,
$$\left\vert \mathbb{P}\big(\xi^{w}_{k}\leq s_{k}\big\vert \xi^{w}(k-1)\big)-s_{k}\right\vert\leq \epsilon,$$
which yields
$$\Bigg\vert\mathbb{E}\left(\mathbb{P}\big(\xi^{w}_{k}\leq s_{k}\big\vert \xi^{w}(k-1)\big)\prod_{i=1}^{k-1}\mathbf{1}_{\xi^{w}_{i}\leq s_{i}}\right)-s_{k}F_{\xi^{w}(k-1)}(s_{1},\ldots,s_{k-1})\Bigg\vert\leq \epsilon.$$
Therefore, for $n\geq \max(N_{k-1},n_{k})$,
\begin{align*}
\big\vert F_{\xi^{w}(k)}(s_{1},\ldots,s_{k})-&F_{k}(s_{1},\ldots,s_{k})\big\vert\\
\leq &\epsilon+\big\vert s_{k}F_{\xi^{w}(k-1)}(s_{1},\ldots,s_{k-1})- F_{k}(s_{1},\ldots,s_{k})\big\vert\\
\leq&\epsilon+s_{k} \big\vert F_{\xi^{w}(k-1)}(s_{1},\ldots,s_{k-1})-s_{1}\cdots s_{k-1} \big\vert
\leq 2\epsilon.
\end{align*}

\end{proof}
\subsection{Proof of \cref{mainresult1}}
\begin{proof}
By the discussion of \cref{ExplanationProof}, it suffices to prove that for any sequence of words $w_{n}$ in $\mathcal{Z}$ with increasing length and any element $(U_{\uparrow},U_{\downarrow})$ in $\mathcal{U}^{(2)}$, the convergence of $\big(U_{\uparrow}(w_{n}),U_{\downarrow}(w_{n})\big)$ towards $(U_{\uparrow},U_{\downarrow})$ in $\mathcal{U}^{(2)}$ implies the convergence of $w_{n}$ towards $\Phi(U_{\uparrow},U_{\downarrow})$ in 
$\hat{\mathcal{Z}}$. Recall that the latter convergence is equivalent to the convergence of $\mathbb{P}_{w_{n}}(\Gamma_{\tau})$ towards $\mathbb{P}_{(U_{\uparrow},U_{\downarrow})}(\Gamma_{\tau})$ for all finite paths $\tau\in\Gamma_{f}(\mathcal{Z})$.

Let $(w_{n})_{n\geq 1}$ be a sequence of words in $\mathcal{Z}$ and let $U=(U_{\uparrow},U_{\downarrow})\in\mathcal{U}^{(2)}$ be such that $l(w_{n})=n-1$ for all $n\geq 1$ and such that $\big(U_{\uparrow}(w_{n}),U_{\downarrow}(w_{n})\big)$ converges to $U$ in $\mathcal{U}^{(2)}$ as $n$ goes to infinity. By \cref{similarPaintbox}, $\tilde{U}_{w_{n}}$ converges also to $U$, where $\tilde{U}_{w_{n}}$ is the run paintbox defined for the word $w_{n}$ in \cref{RunPaintbox}.

Let $\tau\in\Gamma_{f}(\mathcal{Z})$ be a finite path of length $k$. By the equivalence between paths and permutations given in \cref{EquivalencePathsArrangements}, $\tau$ corresponds to a permutation $\sigma_{\tau}\in S_{k}$ in such a way that for $n\geq k$,
$$\mathbb{P}_{w_{n}}\big(\Gamma_{\tau}\big)=\mathbb{P}\big((\sigma_{w_{n}})_{\downarrow k}=\sigma_{\tau}\big).$$
By \cref{equivPaintbox},
$$\mathbb{P}\big((\sigma_{w_{n}})_{\downarrow k}=\sigma_{\tau}\big)=\mathbb{P}\Big(\sigma_{\tilde{U}_{w_{n}}}\big(\xi^{w_{n}}(k)\big)=\sigma_{\tau}\Big).$$
By \cref{generalCase}, as $n$ goes to infinity, $\xi^{w_{n}}(k)$ converges in law to a sequence $(X_{1},\ldots, X_{k})$ of uniform independent random variables on $[0,1]$.

Thus, since $\tilde{U}_{w_{n}}\rightarrow U$, by \cref{convergencePaintbox}
$$\sigma_{\tilde{U}_{w_{n}}}\big(\xi^{w}(k)\big)\xrightarrow[\law]{} \sigma_{U}\big(X_{1},\ldots, X_{k}\big)=\sigma_{U}(k).$$
Therefore, since $\mathbb{P}\big(\sigma_{U}(k)=\sigma_{\tau}\big)=\mathbb{P}_{\Phi(U)}\big(\Gamma_{\tau}\big)$,
$$\mathbb{P}_{w_{n}}(\Gamma_{\tau})\xrightarrow[n\rightarrow\infty]{}\mathbb{P}_{\Phi(U)}(\Gamma_{\tau}).$$
Finally, $w_{n}$ converges to $\Phi(U)$ in $\hat{\mathcal{Z}}$.
\end{proof}
\section{Relation with the Young graph}\label{SectionRelationYong}
We explain here the relation between $\mathcal{Z}$ and the Young graph $\mathcal{Y}$ at the level of paths. The relation between the minimal boundaries of both graphs has already been explained in \cite{gnedin2006coherent} using the rings $Sym$ and $QSym$. 
\subsection{The graph $\mathcal{Y}$}\label{presentationYoungGraph}
A partition $\pi$ of size $n$ and length $r$ is a decreasing sequence of positive integers $(\pi_{1}\geq\cdots\geq\pi_{r})$ such that $\sum \pi_{i}=n$.  We denote by $\vert \pi\vert$ the size of $\pi$ and we also write $\pi\vdash n$ when $\vert \pi\vert=n$. We denote by $l(\pi)$ the length of the partition $\pi$. The set of partitions of size $n$ is denoted by $\mathcal{Y}_{n}$, and the set of all partitions is denoted by $\mathcal{Y}$. The set $\mathcal{Y}$ is partially ordered with the order relation $\preceq$ given by $\pi\preceq \tau$ if and only if $l(\pi)\leq l(\tau)$ and for all $1\leq i\leq l(\pi)$, $\pi_{i}\leq \tau_{i}$.

As for compositions, a Young diagram is associated to each partition by drawing $\pi_{1}$ cells on the first row, $\pi_{2}$ cells on the second row and so on, such that the first cell of the row $i+1$ is just above the first cell of the row $i$. A standard filling of $\pi$ is a filling of $\pi$ with elements of $\lbrace 1,\ldots,n\rbrace$, such that the filling is increasing to the right along the rows and to the top along the columns. A partition $\pi$ together with a standard filling is called a standard tableau of shape $\pi$. We denote by $T_{\pi}$ the set of standard tableaux of shape $\pi$. The partition $\pi=(6,3,2)$ and an example of standard filling of the associated diagram is given in \cref{fig3}.
\begin{figure}[!h]
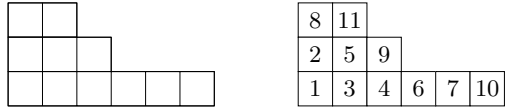

\yng(2,3,6)\hspace{1cm} \young(8<11>,259,13467<10>)  
\caption{\label{fig3}Young diagram of $(6,3,2)$ and an example of standard filling.}
\end{figure}

We say that $\pi\nearrow \tau$  if and only if $\vert\tau\vert=\vert\pi\vert+1$ and $\pi\preceq \tau$. When $T\in T_{\tau}$ is a standard tableau of shape $\tau\vdash n$, $T_{\downarrow}$ is defined as the standard tableau obtained by deleting the cell containing $n$. In particular $T_{\downarrow}$ has a shape $\pi$ such that $\pi\nearrow \tau$.
Adding an edge from $\pi$ to $\tau$ if and only if $\pi\nearrow\tau$ transforms $\mathcal{Y}$ into a graded graph. The latter graph is a major construction for the representation theory of the symmetric groups $(S_{n})_{n\geq 1}$, since the irreducible representations $V_{\tau}$ of $ S_{n}$ are indexed by elements $\tau$ of $\mathcal{Y}_{n}$, and there is a decomposition 
$$\Res(V_{\tau})^{ S_{n}}_{ S_{n-1}}=\bigoplus_{\pi\nearrow \tau} V_{\pi}.$$
The set of paths on $\mathcal{Y}$ between the root $\emptyset$ and a partition $\pi$ is in bijection with the set of standard tableaux of shape $\pi$; thus, an infinite path on $\mathcal{Y}$ can be seen as an infinite standard tableau, which is a sequence $(T_{n})_{n\geq 1}$ of standard tableaux such that $T_{n}=\big(T_{n+1}\big)_{\downarrow}$ for $n\geq 1$. The latter correspondence between infinite paths on $\mathcal{Y}$ and infinite standard tableaux is similar to the one between infinite paths on $\mathcal{Z}$ and arrangements. Each element of $\partial_{\min}\mathcal{Y}$ yields a random infinite path on the graph $\mathcal{Y}$, or equivalently a random infinite standard tableau. The minimal and Martin boundaries of $\mathcal{Y}$ have been intensively studied (see \cite{kerov1998boundary},\cite{kerov2003asymptotic},\cite{thoma1964unzerlegbaren}) and fully described. In particular, Vershik and Kerov \cite{kerov1986characters} have proven that the minimal boundary $\partial_{\min}\mathcal{Y}$ is equal to the Martin boundary $\partial_{M}\mathcal{Y}$ and homeomorphic to the space $$\Delta^{(2)}=\big\lbrace (a_{1}\geq a_{2}\geq \ldots\geq 0),(b_{1}\geq b_{2}\geq \ldots\geq 0),\sum a_{j}+b_{j}\leq1\big\rbrace,$$
considered with the topology of the pointwise convergence. Moreover, $\mathcal{Y}$ admits a geometric realization with the map $g_{\mathcal{Y}}:\mathcal{Y}\longrightarrow \Delta^{(2)}$ given by 
$$g_{\mathcal{Y}}\big((\pi_{1},\ldots,\pi_{r})\big)=\left(\frac{1}{n}\pi_{1}\geq \ldots\geq\frac{1}{n}\pi_{r}\geq 0,\frac{1}{n}\pi'_{1}\geq  \ldots\geq\frac{1}{n}\pi'_{r'}\geq 0\right),$$ 
where $n$ is the size of $\pi$ and
$\pi'$ is the conjugate partition of $\pi$.

 \subsection{Robinson\--Schensted\--Knuth algorithm and the projection $\Gamma(\mathcal{Z})\longrightarrow\Gamma(\mathcal{Y})$}
In this section, we construct the map $\Pi$ of \cref{mainResultYoung} by using the Robinson\--Schensted\--Knuth (or simply RSK) algorithm. Let us first recall this algorithm in the special case of permutations. Initiated by Robinson in \cite{zbMATH03030977} and created by Schensted in \cite{schensted1961longest}, it establishes a bijection between permutations of $n$ and pairs of standard tableaux of size $n$ with the same shape. The algorithm has been later extended to a more general framework by Knuth \cite{knuth1970permutations}. Let $\sigma=\big(\sigma(1),\ldots,\sigma(n)\big)\in S_{n}$. The algorithm constructs a pair of standard tableaux from $\sigma$ as follows. 
\begin{enumerate}
\item Start with an infinite array $A^{0}=(a_{k,l}^{0})_{k,l\geq 1}$ such that each cell is filled with the entry $n+1$ (namely, $a_{k,l}^{0}=n+1$ for all $k,l\geq 1$), and an infinite array $B=(b_{k,l})_{k,l\geq 1}$ such that each cell is empty ($B$ is called the recording tableau).
\item At each step $i$, $1\leq i\leq n$, the following insertion is done on the array $A^{i-1}$:
\begin{itemize}
\item let $(1,l_{1})$ be the first cell (starting from the left) on the first row of $A^{i-1}$ such that $\sigma(i)\leq a^{i-1}_{1,l_{1}}$. Set $a^{i}_{1,l_{1}}=\sigma(i)$;
\item let $(2,l_{2})$ be the first cell on the second row of $A^{i-1}$ such that \\$a^{i-1}_{1,l_{1}}\leq a^{i-1}_{2,l_{2}}$. Set $a^{i}_{2,l_{2}}=a^{i-1}_{1,l_{1}}$;
\item continue the process until the step $k_{0}$ where $a^{i-1}_{k_{0},l_{k_{0}}}>n$. For $k> k_{0}$ or $k\leq k_{0}, l\not= l(k)$, define $a^{i}_{k,l}=a^{i-1}_{k,l}$. Return $A^{i}=(A^{i}_{k,l})_{k,l\geq 1}$. Set $b_{k_{0},l_{k_{0}}}=i$.
\end{itemize}
\item Let $P(\sigma)$ be the part of the array $A^{n}$ containing entries less than or equal to $n$, and $Q(\sigma)$ the part of the array $B$ consisting of non-empty cells. Then, $P(\sigma)$ and $Q(\sigma)$ are two standard tableaux of the same shape which are respectively called the insertion tableau and the recording tableau of $\sigma$.
 \end{enumerate}
 
 \begin{thm}[\cite{zbMATH03030977,schensted1961longest}]\label{RSK}
 The map $S:\sigma\mapsto \big(P(\sigma),Q(\sigma)\big)$ is a bijection between $ S_{n}$ and pairs of standard tableaux of $n$ of the same shape. Moreover 
 $$\big(P(\sigma^{-1}),Q(\sigma^{-1})\big)=\big(Q(\sigma),P(\sigma)\big).$$
 \end{thm}
From now on, $\Pi(\sigma)$ denotes the shape of $P(\sigma)$ (or $Q(\sigma)$). Recall from \cref{presentationYoungGraph} that $T_{\tau}$ denotes the set of standard fillings of the Young diagram $\tau$, which is isomorphic to the set of paths from $\emptyset$ to $\tau$ on $\mathcal{Y}$. In the following lemma, we use the bijection between paths on $\mathcal{Z}$ and coherent sequences of permutations (see \cref{EquivalencePathsArrangements}) and the RSK algorithm to construct a map from $\Gamma(\mathcal{Z})$ to $\Gamma(\mathcal{Y})$. 
\begin{lem}\label{correspondence}
Let $l\in \mathbb{N}\cup\lbrace \infty\rbrace$ and let $(\sigma_{k})_{1\leq k\leq l}$ be a path on $\mathcal{Z}$. Then, $\big(\Pi(\sigma_{k})\big)_{1\leq k\leq l}$ is a path on $\mathcal{Y}$. The induced map $\Pi:\Gamma(\mathcal{Z})\longrightarrow \Gamma(\mathcal{Y})$ is surjective and satisfies $\Pi(\gamma_{\downarrow k})=\Pi(\gamma)_{\downarrow k}$ for all path $\gamma\in \Gamma(\mathcal{Z})$ and all $k\leq l(\gamma)$.
\end{lem}
\begin{proof}
Let $\sigma\in S_{n}$. It is clear from the RSK algorithm that
$$P(\sigma_{\downarrow})=P(\sigma)_{\downarrow},$$ 
which yields that $\Pi(\sigma_{\downarrow})\nearrow \Pi(\sigma)$. Hence, for any coherent sequence of permutations $(\sigma_{k})_{1\leq k\leq l}$, the sequence $\big(\Pi(\sigma_{k})\big)_{1\leq k\leq l}$ is a well-defined path on $\mathcal{Y}$.

Let $\gamma=(\sigma_{k})_{1\leq k\leq l}\in\Gamma(\mathcal{Z})$ and let $k_{0}\leq l(\gamma)$. Then, 
$\gamma_{\downarrow k_{0}}=(\sigma_{k})_{1\leq k\leq k_{0}}$, so that 
$$\Pi(\gamma_{\downarrow k_{0}})=\big(\Pi(\sigma_{k})\big)_{1\leq k\leq k_{0}}=\left[\big(\Pi(\sigma_{k})\big)_{1\leq k\leq l}\right]_{\downarrow k_{0}}=\big(\Pi(\gamma)\big)_{\downarrow k_{0}}.$$

We prove now the surjectivity of $\Pi$.  The reading word $\sigma(T)$ of a standard tableau $T$ is the permutation obtained by concatenating the rows of $T$ from top to bottom. For example, the reading word of the standard tableau of Figure \ref{fig3} is $(8,11,2,5,9,1,3,4,6,7,10)$. By \cite[Lemma A.1.1.10]{stanley2011enumerative}, $P\big(\sigma(T)\big)=T$ for all standard tableaux $T$. Moreover, it is readily seen that if $T'=T_{\downarrow}$, then $\sigma(T')=\sigma(T)_{\downarrow}$. Let $(\pi_{i})_{1\leq i\leq l}$ be a path on $\mathcal{Y}$ of length $l\in\mathbb{N}\cup \lbrace \infty\rbrace$. Then, for each finite number $k$ such that $k\leq l$, the finite path $(\pi_{i})_{1\leq i\leq k}$ corresponds to a standard tableau $T_{k}$ of shape $\pi_{k}$. Set $\sigma_{k}=\sigma(T_{k})$. Then, $(\sigma_{k})_{1\leq k\leq l}$ is a path on $\mathcal{Z}$ and  for all $k\geq 1$, $\Pi\big(\sigma_{k}\big)=\pi_{k}$. Hence, $\Pi\big((\sigma_{k})_{1\leq k\leq l}\big)=(\pi_{i})_{1\leq i\leq l}$, and the map $\Pi$ is surjective.
 \end{proof}
Remark that \cref{correspondence} corresponds to the first assertion of \cref{mainResultYoung}. In the following lemma, which proves the second assertion of \cref{mainResultYoung}, we use the notion of descent word of a standard tableau. The descent word of a standard tableau $Q$ of size $n$ is the word $w(Q)$ in $\mathcal{Z}_{n}$ such that $w(Q)_{i}=-$ if and only if $i+1$ is in a strictly lower row than $i$ in $Q$. By a standard combinatorial result (see \cite[Lemma 7.23.1]{stanley2011enumerative}), $i$ is a descent of a permutation $\sigma$ (and thus $w(\sigma)_{i}=-$) if and only if $i+1$ is in a strictly lower row than $i$ in $Q(\sigma)$, which implies that $w(\sigma)=w\big(Q(\sigma)\big)$. 
\begin{lem}\label{yieldHarmon}
Let $\omega$ be a harmonic measure on $\Gamma_{\infty}(\mathcal{Z})$. Then, the pushforward of $\omega$ by $\Pi$ is a harmonic measure $\tilde{\Pi}(\omega)$ on $\Gamma_{\infty}(\mathcal{Y})$. Moreover, if $P\in \Gamma_{f}(\mathcal{Y})$ is a finite path of length $k$ ending at $\tau\in \mathcal{Y}$, then
$$\mathbb{P}_{\tilde{\Pi}(\omega)}(\Gamma_{P})=\sum_{\substack{Q\in T_{\tau}}}p_{\omega}\big(w(Q)\big),$$
where $p_{\omega}(w)=\mathbb{P}_{\omega}(\Gamma_{\gamma})$ for any path $\gamma\in\Gamma_{f}(\mathcal{Z})$ ending at $w$.
\end{lem}
\begin{proof}
Let $\omega\in \mathcal{H}_{\infty}(\mathcal{Z})$ and let $P\in\Gamma_{f}(\mathcal{Y})$ be a path of length $k$ ending at $\tau\in\mathcal{Y}_{k}$. Let us first prove the second assertion of the lemma. By \cref{correspondence},
\begin{align*}
\Pi^{-1}(\Gamma_{P})\cap\Gamma_{\infty}(\mathcal{Z})=&\left\lbrace (\sigma_{n})_{n\geq 1}\in\Gamma_{\infty}(\mathcal{Z})\vert \big(\Pi(\sigma_{1}),\ldots,\Pi(\sigma_{k})\big)=P\right\rbrace\\
=&\left\lbrace (\sigma_{n})_{n\geq 1}\in\Gamma_{\infty}(\mathcal{Z})\vert P(\sigma_{k})=P\right\rbrace
=\bigsqcup_{\substack{\sigma\in S_{k}\\P(\sigma)=P}}\Gamma_{\sigma}.
\end{align*}
Hence,
$$\mathbb{P}_{\tilde{\Pi}(\omega)}(\Gamma_{P})=\mathbb{P}_{\omega}\left(\Pi^{-1}\big(\Gamma_{P}\big)\right)=\sum_{\substack{\sigma\in S_{k}\\ P(\sigma)=P}}\mathbb{P}_{\omega}(\Gamma_{\sigma})=\sum_{\substack{\sigma\in S_{k}\\ P(\sigma)=P}}\mathbb{P}_{\omega}\big((\sigma_{\omega})_{\downarrow k}=\sigma\big).$$
Since $\omega$ is harmonic on $\Gamma_{\infty}(\mathcal{Z})$, 
$$\mathbb{P}_{\omega}\big((\sigma_{\omega})_{\downarrow k}=\sigma\big)=p_{\omega}\big(w(\sigma)\big)=p_{\omega}\left[Q\big(w(\sigma)\big)\right].$$
Hence, if $P\in T_{\tau}$, then \cref{RSK} yields
\begin{align*}
\mathbb{P}_{\tilde{\Pi}(\omega)}\big(\Gamma_{P}\big)=&\sum_{\substack{\sigma\in S_{k}\\ P(\sigma)=P}}p_{\omega}\left[w(\big(Q(\sigma)\big)\right]=\sum_{\substack{(P',Q)\in T_{\tau}\\ P'=P}}p_{\omega}\big(w(Q)\big)\\
=&\sum_{Q\in T_{\tau}}p_{\omega}\big(w(Q)\big).
\end{align*}
From the above equation, $\mathbb{P}_{\tilde{\Pi}(\omega)}\big(\Gamma_{P}\big)$ only depends on the endpoint $\tau$ of $P$ (which is the shape of $P$). Thus, $\tilde{\Pi}(\omega)$ is a harmonic measure.
 \end{proof}
Note that we could prove as well that $\Pi$ induces a map $\hat{\Pi}$ from $\mathcal{H}_{f}(\mathcal{Z})$ to $\mathcal{H}_{f}(\mathcal{Y})$. However, contrary to the infinite case that we are studying below, the map $\hat{\Pi}$ is not surjective in the finite case: for example, we can prove by direct computation that the uniform measure on the set of paths on $\mathcal{Y}$ arriving at the partition $(2,2)$ can not be obtained as the projection of a measure in $\mathcal{H}_{f}(\mathcal{Z})$.

Since we already know that the map $\Psi$ is surjective (see the definition of $\Psi$ before \cref{mainResultYoung}), the only thing to prove in the third assertion of \cref{mainResultYoung} is that $\tilde{\Pi}$ coincides with the map $\Psi$ on $\partial_{\min}\mathcal{Z}$. Let us recall first how this map $\Psi$ has been defined in \cite[Section 3]{gnedin2006coherent}. A harmonic measure $\omega$ in $\partial_{\min}\mathcal{Z}$ yields a linear map $\phi_{\omega}:QSym\longrightarrow\mathbb{R}$ such that $\phi_{\omega}(F_{\lambda})=\mathbb{P}_{\omega}(\Gamma_{\tau})$ for any finite path $\tau\in\Gamma_{f}(\mathcal{Z})$ ending on the word $w(\lambda)$. Since $Sym$ is a subring of $QSym$ and since each Schur function $s_{\pi}$ in $Sym$ admits a non-negative decomposition on the basis $\lbrace F_{\lambda}\rbrace_{\lambda\in\Lambda}$, the restriction of the map $\phi_{\omega}$ to $Sym$ satisfies $\phi_{\omega}(s_{\pi})\geq 0$ and $\phi_{\omega}(s_{\pi}s_{(1)})=\phi_{\omega}(s_{\pi})$ for all $\pi\in\mathcal{Y}$. By \cite[Proposition 12]{gnedin2006coherent}, the map $\big(\phi_{\omega}\big)_{\vert Sym}$ yields a harmonic measure $\Psi(\omega)$ in $\partial\mathcal{H}_{\infty}(\mathcal{Y})$ by setting $\mathbb{P}_{\Psi(\omega)}(\Gamma_{P})=\phi_{\omega}(s_{\pi})$ for any path $P\in\Gamma_{f}(\mathcal{Y})$ ending at $\pi$.

\begin{lem}\label{equivPsi}
The map $\Psi$ is equal to the map $\tilde{\Pi}$ restricted to $\partial_{\min}\mathcal{Z}$.
\end{lem}
\begin{proof}
 We use here the notations given before the statement of the lemma. Let $\omega\in\partial_{\min}\mathcal{Z}$ and let $\phi_{\omega}:QSym\longrightarrow \mathbb{R}$ be the corresponding linear form given above. Comparing the definition of $\phi_{\omega}$ and the one of $p_{\omega}$ from \cref{correspondence} yields that $\phi_{\omega}(F_{\lambda})=p_{\omega}\big(w(\lambda)\big)$ for $\lambda\in\Lambda$. Let $P\in\Gamma_{f}(\mathcal{Y})$ be a finite path ending on $\tau\in\mathcal{Y}_{k}$. 
 Then, by \cref{correspondence} and the previous remark,
$$\mathbb{P}_{\tilde{\Pi}(\omega)}(\Gamma_{P})=\sum_{\substack{Q\in T_{\tau}}}\phi_{\omega}\big(F_{\lambda(Q)}\big),$$
where we have written $\lambda(Q)$ instead of $\lambda\big(w(Q)\big)$.
The decomposition of the Schur function $s_{\tau}$ into the basis $\lbrace F_{\lambda}\rbrace_{\lambda\in\Lambda}$ is given by the formula (see for example \cite[Theorem 7]{gessel1984multipartite})
$$s_{\tau}=\sum_{Q\in T_{\tau}}F_{\lambda(Q)}.$$
Hence,
$$\mathbb{P}_{\Psi(\omega)}(\Gamma_{P})=\phi_{\omega}(s_{\tau})=\sum_{Q\in T_{\tau}}\phi_{\omega}\big(F_{\lambda(Q)}\big)
=\mathbb{P}_{\tilde{\Pi}(\omega)}(\Gamma_{P}),$$
and $\tilde{\Pi}(\omega)=\Psi(\omega)$.
\end{proof}
The proof of \cref{mainResultYoung} is just a compilation of \cref{correspondence}, \cref{yieldHarmon} and \cref{equivPsi}.

\section{Law of large numbers and central limit theorem}\label{SectionLawLargeNumber}
The purpose of this section is to prove \cref{lawLageNumbers}. For $\sigma\in S_{n}$, let $f_{\sigma}$ be the piecewise linear function on $[0,1]$ such that $f_{\sigma}(0)=0$, and for $1\leq i\leq n-1$,
$$f_{\sigma}\left(\frac{i}{n-1}\right)-f_{\sigma}\left(\frac{i-1}{n-1}\right)=\left\lbrace \begin{matrix}
-\frac{1}{n-1}&\text{ if $i$ is a descent of $\sigma$}\\
\;\frac{1}{n-1}&\text{ if $i$ is an ascent of $\sigma$.}\\
\end{matrix}\right.$$
Note that equivalently, $f_{\sigma}=f_{g_{\mathcal{Z}}(w(\sigma))}$ with the definition of $f_{g_{\mathcal{Z}}(w)}$ given in \cref{lawLargeNumberCentralLimit}. In order to get the law of large numbers, it is more convenient to prove first the central limit theorem for the Gibbs measure $\Phi_{\mathcal{Z}}(\emptyset,\emptyset)$. The result consists mainly in a mathematical formalization of the results obtained by Oshanin and Voituriez from a physical point of view in \cite{oshanin2004random}. The reader should refer to the latter paper for interesting additional details on the process $\left(f_{\sigma_{(\emptyset,\emptyset)}(n)}\right)_{n\geq 1}$.
\begin{prop}\label{centralLimit}
For $n$ going to infinity, the following convergence holds in distribution in $\mathcal{C}\left([0,1],\Vert.\Vert_{\infty}\right)$:
$$\sqrt{n}f_{\sigma_{(\emptyset,\emptyset)}(n)}\xrightarrow[n\rightarrow \infty]{}\frac{1}{\sqrt{3}}\mathcal{B},$$
where $\mathcal{B}$ denotes a standard Brownian motion on $[0,1]$.
\end{prop}
\begin{proof}
Recall from \cref{paintboxDef} that $\big(\sigma_{\emptyset,\emptyset}(n)\big)$ can be sampled with a family of independent uniform random variables $(X_{i})_{i\geq 1}$ on $[0,1]$ by applying the oriented paintbox construction with the oriented paintbox 
$(\emptyset,\emptyset)$; for $n\geq 1$, the distribution of $\sigma_{\emptyset,\emptyset}(n)$ is just the uniform distribution on $S_{n}$. Since the map $\sigma\mapsto\sigma^{-1}$ preserves the uniform distribution on $S_{n}$, $f_{\sigma_{\emptyset,\emptyset}(n)}\overset{\law}{=} f_{\sigma_{\emptyset,\emptyset}(n)^{-1}}$. As remarked by Oshanin and Voituriez, $\big((n-1)f_{\sigma_{\emptyset,\emptyset}(n)^{-1}},X_{n}\big)_{n\geq 1}$ is a Markov chain: indeed $i$ is a descent of $\sigma_{\emptyset,\emptyset}(n)^{-1}$ if and only if $X_{i}>X_{i+1}$. Therefore, 
$$w\left(\big(\sigma_{\emptyset,\emptyset}(X_{1},\ldots, X_{n+1})\big)^{-1}\right)=w\left(\big(\sigma_{\emptyset,\emptyset}(X_{1},\ldots, X_{n+1})\big)^{-1}\right)\epsilon,$$
 where $\epsilon=+$ if $X_{n}<X_{n+1}$ and $\epsilon=-$ if $X_{n}>X_{n+1}$. In the sequel, $(n-1)f_{\sigma_{\emptyset,\emptyset}(n)}\left(\frac{i}{n-1}\right)$ is denoted by $Y_{i}$ (the subscript $n$ is dropped, since the the latter quantity only depends on $\sigma_{\emptyset,\emptyset}(n)_{\downarrow i+1}$).

For $\sigma\in S_{n}$, denotes by $D(\sigma)$ the set of descents of $\sigma$. By the Markov property, for $R=\llbracket r_{1},r_{1}+r_{2}\rrbracket$ and $S=\llbracket s_{1},s_{1}+s_{2}\rrbracket$ with $s_{1}\geq r_{1}+r_{2}+2$, $n\geq s_{1}+s_{2}$, we have  
\begin{equation}\label{indepDescents}
\Big(\#D\big(\sigma_{\emptyset,\emptyset}(n)\big)\cap R,\# D\big(\sigma_{\emptyset,\emptyset}(n)\big)\cap S\Big)\overset{\law}{=} (T_{1}, T_{2}),
\end{equation}
with $T_{1}\overset{\law}{=} \#D\big(\sigma_{\emptyset,\emptyset}(r_{2}+1)\big)$, $T_{2}\overset{\law}{=}\# D\big(\sigma_{\emptyset,\emptyset}(s_{2}+1)\big)$, and $T_{1}$ and $T_{2}$ are independent. Moreover, the number of permutations of $n$ with $k$ descents is given by the Eulerian number $A_{k}^{n}$, whose asymptotic value yields the following lemma.
\begin{lem}\cite[Proposition IX.9]{flajolet2009analytic}\label{convergeLawDescent}
Let $\sigma_{n}$ be a uniformly random permutation in $ S_{n}$. For $n$ going to infinity,
$$\frac{1}{\sqrt{n}}\left(\# D(\sigma_{n})-\frac{n}{2}\right)\xrightarrow[\law]{}\frac{1}{2\sqrt{3}}N,$$
where $N$ denotes a Gaussian variable with mean $0$ and variance $1/12$.
\end{lem}
Remark that for $1\leq r_{1}<r_{2}\leq n-1$, 
\begin{align*}
\sqrt{n-1}\left(f_{\sigma_{\emptyset,\emptyset}(n)}\left(\frac{r_{2}}{n-1}\right)-f_{\sigma_{\emptyset,\emptyset}(n)}\left(\frac{r_{1}}{n-1}\right)\right)&\\=\frac{1}{\sqrt{n-1}}(Y_{r_{2}}-Y_{r_{1}})
\overset{\law}{=}\frac{1}{\sqrt{n-1}}\big((r_{2}-r_{1})&-2\#D(\sigma_{r_{2}-r_{1}+1})\big),
\end{align*}
where $\sigma_{r_{2}-r_{1}+1}$ denotes a uniformly random permutation in $S_{r_{2}-r_{1}+1}$. Thus, \cref{convergeLawDescent} together with \eqref{indepDescents} yield the convergence of the marginal distributions  of $\sqrt{n}f_{\sigma_{\emptyset,\emptyset}(n)}$ towards the ones of $\frac{1}{\sqrt{3}}\mathcal{B}$. We end now the proof of \cref{centralLimit} with a tightness argument. We follow Theorem 8.4 of the first version of the book \cite{billingsley2009convergence} of Billingsley.
\begin{thm}\cite[Theorem 8.4]{billingsley2009convergence}\label{theoremBillingsley}
Let $(X_{i})_{i\in\mathbb{N}}$ be a real random process. Let $g_{n}:[0,1]\rightarrow \mathbb{R}$ be the piecewise linear function such that
 $$g_{n}\left(\frac{i}{n}\right)=\frac{1}{\sqrt{n}}X_{i}, \quad 0\leq i\leq n.$$
 Suppose that for all $\epsilon>0$, there exists $\lambda>1$ and $n_{0}\geq 0$ such that for all $k\in\mathbb{N}$ and $n\geq n_{0}$,
 $$\mathbb{P}\big(\max_{i\leq n} \vert X_{k+i}-X_{k}\vert\geq \lambda\sqrt{n}\big)\leq \epsilon/\lambda^{2} .$$
 Then, the sequence $g_{n}$ is tight.
 \end{thm}
 The hypothesis of \cref{theoremBillingsley} is verified in our case through the following lemma, which mimics the situation of a simple random walk. In the following statement, $F_{U}$ denotes the cumulative distribution function of a random variable $U$.
\begin{lem}\label{propertyBillingsley}
Set $K_{n}=\sup_{\llbracket 0,n\rrbracket}Y_{n}$. Then, 
$$F_{K_{n}}(t)\geq  F_{\vert Y_{n}\vert}(t-1)$$ 
for all $t\in \mathbb{R}$.
\end{lem}
\begin{proof}
Let $a$ and $b$ be two integers such that $b\leq a-2$. We first prove that 
$$\mathbb{P}(K_{n}\geq a, Y_{n}\leq b)\leq \mathbb{P}(Y_{n}\geq 2a-b-2).$$
Note first that $T=\inf(u\in\mathbb{N},Y_{u}= a)$ is a stopping time for the Markov chain $(Y_{n},X_{n})$. Since $\lbrace K_{n}\geq a\rbrace=\lbrace T\leq n\rbrace$, $\lbrace K_{n}\geq a\rbrace\in\mathcal{F}_{T}$ and by the strong Markov property,
\begin{align*}
\mathbb{P}(K_{n}\geq a, Y_{n}\leq b)=&\mathbb{P}\big((T\leq n)\cap (Y_{n}-Y_{T}\leq b-a)\big)\\
=&\mathbb{E}\big(\mathbf{1}_{T\leq n}\mathbb{P}_{(Y_{T},X_{T})}(\tilde{Y}_{n-T}-\tilde{Y}_{0}\leq b-a)\big)\\
\leq &\mathbb{E}\big(\mathbf{1}_{T\leq n}\mathbb{P}_{(Y_{T},X_{T})}(\tilde{Y}_{n-T}-\tilde{Y}_{1}\leq b-a+1)\big),
\end{align*}
with $(\tilde{Y}_{i},\tilde{X}_{i})_{i\geq 0}$ being an independent random walk starting at $(\tilde{Y}_{0},\tilde{X}_{0})=(Y_{T},X_{T})$.
Since $\tilde{Y}_{n-T}-\tilde{Y}_{1}$ is independent of the value $(\tilde{Y}_{0},\tilde{X}_{0})$ and distributed as $\tilde{Y}_{n-T-1}$, a symmetric random variable,
$$\tilde{Y}_{n-T}-\tilde{Y}_{1}\overset{\law}{=}-(\tilde{Y}_{n-T}-\tilde{Y}_{1}).$$ 
Hence,
\begin{align*}
\mathbb{E}\big(\mathbf{1}_{T\leq n}\mathbb{P}_{(Y_{T},X_{T})}(\tilde{Y}_{n-T}-\tilde{Y}_{1}\leq &b-a+1)\big)\\
=&\mathbb{E}(\mathbf{1}_{T\leq n}\mathbb{P}_{(Y_{T},X_{T})}(-(\tilde{Y}_{n-T}-\tilde{Y}_{1})\leq b-a+1))\\
=&\mathbb{E}\big(\mathbf{1}_{T\leq n}\mathbb{P}_{(Y_{T},X_{T})}(\tilde{Y}_{n-T}\geq a-(b+1)+\tilde{Y}_{1})\big)\\
\leq &\mathbb{E}\big(\mathbf{1}_{T\leq n}\mathbb{P}_{(Y_{T},X_{T})}(\tilde{Y}_{n-T}\geq a-(b+1)-1)\big)\\
\leq &\mathbb{P}\big((T\leq n)\cap(Y_{n}\geq 2a-b-2)\big).
\end{align*}
Since $b\leq a-2$, $(Y_{n}\geq 2a-b-2)\subset (T\leq n)$ and the above inequalities yield
\begin{align}
 \mathbb{P}(K_{n}\geq a, Y_{n}\leq b)\leq& \mathbb{E}\big(\mathbf{1}_{T\leq n}\mathbb{P}_{(Y_{T},X_{T})}(\tilde{Y}_{n-T}-\tilde{Y}_{1}\leq b-a+1)\big)\label{firstPartKn}\\
 \leq &\mathbb{P}(Y_{n}\geq 2a-b-2).\nonumber
 \end{align}
Thus, for $a\in\mathbb{N}$,
\begin{align*}
\mathbb{P}(K_{n}\geq a)=&\mathbb{P}\big((K_{n}\geq a)\cap (Y_{n}\leq a-2)\big)+\mathbb{P}\big((K_{n}\geq a)\cap (Y_{n}\geq a-1)\big)\\
\leq& \mathbb{P}(Y_{n}\geq a)+\mathbb{P}(Y_{n}\geq a-1 )\leq \mathbb{P}(\vert Y_{n}\vert \geq a-1),
\end{align*}
the last inequality being due to the fact that $Y_{n}$ is a symmetric random variable. This yields
$$F_{K_{n}}(u)\geq F_{\vert Y_{n}\vert}(u-1)$$
for $u$ integer. Since both cumulative distribution functions are constant between two consecutive integers, the assertion of the lemma is proved for all $t\in\mathbb{R}$. 
\end{proof}
Let us conclude now the proof of \cref{centralLimit}. Let $\epsilon>0$ and let $\lambda>1$ be such that $\mathbb{P}\big(N\geq \lambda-1\big)\leq \frac{\epsilon}{4\lambda^{2}}$, where $N$ is a gaussian variable of mean $0$ and variance $1/3$. Such $\lambda$ exists because $\mathbb{P}(N\geq \lambda)$ decays faster than exponentially. Since  $Y_{k+n}-Y_{k}$ is distributed as $Y_{n}$ and $\frac{1}{\sqrt{n}}Y_{n}$ converges in distribution to $N$ as $n$ goes to infinity, there exists $n_{0}$ such that for $n\geq n_{0}$ and $k\geq 0$,
$$\mathbb{P}\big(\vert Y_{k+n}-Y_{k}\vert\geq \lambda\sqrt{n}-1\big)\leq \mathbb{P}(\vert N\vert \geq \lambda-1)+\frac{\epsilon}{2\lambda^{2}}.$$
Hence, for $k\geq 0$ and $n\geq n_{0}$, \cref{propertyBillingsley} and the choice of $\lambda$ yield that
$$\mathbb{P}\big(\max_{i\leq n} \vert Y_{k+i}-Y_{k}\vert\geq \lambda\sqrt{n}\big)\leq 2\mathbb{P}\big(N\geq \lambda-1\big)+\frac{\epsilon}{2\lambda^{2}}
\leq\frac{\epsilon}{\lambda^{2}}.$$
Therefore, we can apply \cref{theoremBillingsley} to get the tightness of the sequence $\big(\sqrt{n}f_{\sigma_{\emptyset,\emptyset}(n)}\big)_{n\geq 1}$. Since we already now that the finite dimensional marginal converges to the ones of $\frac{1}{\sqrt{3}}\mathcal{B}$, the convergence in distribution of the sequence of random variables $\big(\sqrt{n}f_{\sigma_{\emptyset,\emptyset}(n)}\big)_{n\geq 1}$ follows.
\end{proof}

A direct consequence of \cref{centralLimit} is that almost surely, $f_{\sigma_{(\emptyset,\emptyset)}(n)}$ converges uniformly to $0$ as $n$ goes to infinity. 
\begin{proof}{Proof of \cref{lawLageNumbers}.}
Let $U=(U_{\uparrow},U_{\downarrow})\in \mathcal{U}^{(2)}$. Since $f_{(U_{\uparrow},U_{\downarrow})}$ and $f_{\sigma_{U}(n)}$ are $1-$Lipschitz, it suffices to show the almost sure pointwise convergence on $\mathbb{Q}\cap [0,1]$.

By the second assertion of \cref{martinEntranceBoundary}, $w\big[\sigma_{U}(n)\big]$ converges almost surely to $\Psi(U)$ in $\hat{Z}$, thus \cref{mainresult1} yields that $U\big(w[\sigma_{U}(n)]\big)$ converges almost surely to $U$ in $\mathcal{U}^{(2)}$. In particular, for any connected component $I$ of $U_{\downarrow}$ or $U_{\uparrow}$, and $x,y\in I$, 
$$\big\vert \big(f_{\sigma_{U}(n)}(y)-f_{\sigma_{U}(n)}(x)\big)-\big(f_{(U_{\uparrow},U_{\downarrow})}(y)-f_{(U_{\uparrow},U_{\downarrow})}(x)\big)\big\vert \xrightarrow[n\rightarrow \infty]{} 0.$$ 
Let $I$ be a connected component of $[0,1]\setminus U$ and let $x,y\in I$. Denote by $d(n)$ the random variable $\#\lbrace 1\leq i\leq n\vert X_{i}\in]x,y[\rbrace$. Then, by the oriented paintbox construction, 
$$f_{\sigma_{U}(n)}(y)-f_{\sigma_{U}(n)}(x)\overset{\law}{=}\frac{d(n)}{n}\big(f_{\sigma_{\emptyset,\emptyset}(d(n))}(1)-f_{\sigma_{\emptyset,\emptyset}(d(n))}(0)\big),$$
where the random variable $\sigma_{\emptyset,\emptyset}(d(n))$ is independent of $d(n)$. By the law of large numbers, 
$\frac{d(n)}{n}$ converges almost surely to $y-x$ and thus, by \cref{centralLimit}, $\big(f_{\sigma_{\emptyset,\emptyset}(d(n))}(1)-f_{\sigma_{\emptyset,\emptyset}(d(n))}(0)\big)$ converges to $0$ as $n$ goes to infinity. Therefore, for $x,y\in I$, 
$$f_{\sigma_{U}(n)}(y)-f_{\sigma_{U}(n)}(x) \xrightarrow[n\rightarrow \infty]{} 0=f_{(U_{\uparrow},U_{\downarrow})}(y)-f_{(U_{\uparrow},U_{\downarrow})}(x).$$
Let $x\in \mathbb{Q}\cap]0,1[$ and $\epsilon>0$. There exist $r\geq 1$ and $I_{1},\ldots ,I_{r}$ interval components of either $U$ or $[0,1]\setminus U$ such that $\Leb\left(\bigcup_{i=1}^{r} I_{i}\cap [0,x]\right)\geq x-\epsilon$. Since $f_{(U_{\uparrow},U_{\downarrow})}$ and $f_{\sigma_{U}(n)}$ are $1-$Lipschitz, the previous convergence results yield the almost sure existence of $n_{x}\geq 1$ such that for $n\geq n_{x}$,
$$\left\vert f_{\sigma_{U}(n)}(x)-f_{(U_{\uparrow},U_{\downarrow})}(x)\right\vert\leq (r+2)\epsilon.$$
Thus, almost surely, for all $x\in \mathbb{Q}\cap [0,1]$, 
$$f_{\sigma_{U}(n)}(x)\xrightarrow[n\rightarrow \infty]{}f_{(U_{\uparrow},U_{\downarrow})}(x).$$
This concludes the proof of \cref{lawLageNumbers}.
\end{proof}

\section*{Acknowledgments}
I wish to thank Philippe Biane and Jean-Yves Thibon who suggested these problems, and gave me the material to understand the subject. This work greatly benefited from the discussions with the people of the combinatorics workshop of Jean-Yves Thibon at the University of Marne-la-Vallée, and the team of Roland Speicher at the University of Saarland.

I am also grateful to the referees for several comments which significantly improved the overall quality of the manuscript.

\section*{Appendix: Proof of \cref{convergencePaintbox}}
Let $k\geq 1$. For $n\geq 1$, we denote by $\mathbb{P}_{n}$ the distribution of $X^{n}_{k}$ and by $\mathbb{P}$ the distribution of $X_{k}$.
\begin{lem}\label{mainStepBis}
For each $\epsilon>0$, there exists $\mathcal{X}_{\epsilon}\in\mathcal{B}([0,1])$ and $n_{0}\geq 1$ such that 
\begin{itemize}
\item $\mathbb{P}(\mathcal{X}_{\epsilon})\geq 1-\epsilon$, and
\item for all $n\geq n_{0}$, $\sigma_{U_{n}}$ and $\sigma_{U}$ coincide on $\mathcal{X}_{\epsilon}$.
\end{itemize}
\end{lem}
\begin{proof}
Let $\epsilon>0$. For $\delta>0$, set 
$$\Delta_{\delta}:=\bigcup_{1\leq i<j\leq k}\left\lbrace (x_{1},\ldots,x_{k})\in [0,1]^{k}\big\vert \vert x_{i}-x_{j}\vert \leq \delta\right\rbrace.$$ Then, $\partial_{[0,1]^{k}} \Delta_{\delta}=\bigcup_{1\leq i,j\leq k}\left\lbrace (x_{1},\ldots,x_{k})\in [0,1]^{k} \big\vert\vert x_{i}-x_{j}\vert = \delta\right\rbrace$. Since the latter  has Lebesgue measure zero, $\mathbb{P}\left(X_{k}\in \partial_{[0,1]^{k}} \Delta_{\delta}\right)=0$. Since $\Delta_{\delta}$ is decreasing in $\delta$ and $\Leb\left(\bigcap_{\delta>0} \Delta_{\delta}\right)=0$, there exists $\delta>0$ such that $\mathbb{P}\left(X\in \Delta_{\delta}\right)\leq \epsilon/2$.

Denote by $\mathcal{U}=\left\lbrace U_{i}=]r_{i},s_{i}[\right\rbrace_{1\leq i\leq r}$ the finite ordered collection of interval components of $U_{\downarrow}\cup U_{\uparrow}$ of size larger than $\delta$, where we suppose that $s_{i}<r_{i+1}$ for $1\leq i\leq r-1$. For $\eta>0$, let $$B_{\eta}:=\bigcup_{\substack{1\leq j\leq k\\1\leq i\leq r}}\left\lbrace (x_{1},\ldots,x_{k})\in [0,1]^{k}\big\vert x_{j}\in ]r_{i}-\eta,r_{i}+\eta[\cup]s_{i}-\eta,s_{i}+\eta[\right\rbrace.$$
Once again, $\Leb\left(\partial_{[0,1]^{k}}B_{\eta}\right)=0$,  $B_{\eta}$ is decreasing in $\eta$ and  $\left(\bigcap_{\eta>0} B_{\eta}\right)$ is a null set, thus there exists $\eta>0$ such that $\mathbb{P}\left(X_{k}\in B_{\eta}\right)\leq \epsilon/2$. Let $K_{\epsilon}=B_{\eta}\cup \Delta_{\delta}$ and let $\mathcal{X}_{\epsilon}= \Delta\setminus K_{\epsilon}$. Then, $\mathbb{P}(\mathcal{X}_{\epsilon})\geq 1-\epsilon$. Moreover, since $\Leb(\partial\mathcal{X}_{\epsilon})=\Leb(\partial K_{\epsilon})=0$, $\mathbb{P}(\partial\mathcal{X}_{\epsilon})=0$.

Set $\kappa=\min(\eta,\delta)$ and let $n_{0}\geq 1$ be such that for $n\geq n_{0}$, $d_{\mathcal{U}^{(2)}}(U_{n},U)\leq \kappa$. Suppose from now on that $n\geq n_{0}$. Since $d_{\mathcal{U}^{(2)}}(U_{n},U)\leq \kappa\leq \delta$, the interval components of $U_{\downarrow}^{n}$ (resp.~$U_{\uparrow}^{n})$ of size larger than $\delta$ are in order respecting bijection with those of $U_{\downarrow}$ (resp.~$U_{\uparrow}$). Denote by $\mathcal{U}_{n}=\left\lbrace U_{i}^{n}:=]r_{i}^{n},s_{i}^{n}[\right\rbrace_{1\leq i\leq r}$ the interval components of $U_{n}$ of size larger than $\delta$: then, $U_{i}^{n}\subset U^{n}_{\uparrow}$ if and only if $U_{i}\subset U_{\uparrow}$. Moreover, since $d_{\mathcal{U}^{(2)}}(U_{n},U)\leq \kappa\leq \eta$, $\vert r_{i}^{n}-r_{i}\vert <\eta$ and $\vert s_{i}^{n}-s_{i}\vert <\eta$ for $1\leq i\leq r$.

We prove now that $\sigma_{U}$ and $\sigma_{U_{n}}$ coincide on $\mathcal{X}_{\epsilon}$ for $n\geq n_{0}$. Let $\vec{x}\in \mathcal{X}_{\epsilon}$ and $n\geq n_{0}$. We set $\sigma:=\sigma_{U}(\vec{x})$ and $\sigma_{n}:=\sigma_{U_{n}}(\vec{x})$. Let $1\leq i<j\leq k$ and suppose that $\sigma^{-1}(i)<\sigma^{-1}(j)$. By the oriented paintbox construction, this happens if and only if 
\begin{enumerate}
\item either $x_{i}$ and $x_{j}$ are not in the same connected component of $U_{\uparrow}\cup U_{\downarrow}$ and $x_{i}<x_{j}$, or
\item $x_{i}$ and $x_{j}$ are in the same connected component of $U_{\uparrow}$. 
\end{enumerate}
Let us suppose that we are in the case $(1)$. Since $\vec{x}\in \mathcal{X}_{\epsilon}$, $\vert x_{i}-x_{j}\vert >\delta$. Thus, $x_{i}$ and $x_{j}$ can not be in a same connected component of $U_{n}$ of size smaller than $\delta$. Suppose that $x_{i}$ is in a connected component $U_{i}^{n}=]r_{i}^{n},s_{i}^{n}[$ of size larger than $\delta$, with $1\leq i\leq r$. Since $d_{\mathcal{U}^{(2)}}(U,U_{n})\leq \eta$, $\vert r_{i}-r_{i}^{n}\vert \leq \eta$ and $\vert s_{i}-s_{i}^{n}\vert \leq \eta$. Hence, $x_{i}\in ]r_{i}-\eta,s_{i}+\eta[$. Since $\vec{x}\in \mathcal{X}_{\epsilon}$, $\vert x_{i}-r_{i}\vert >\eta$ and $\vert x_{i}-s_{i}\vert >\eta$, and thus $x_{i}\in ]r_{i}+\eta,s_{i}-\eta[\subset U_{i}$. Since we are in the case $(1)$, this means that $x_{j}\not \in U_{i}$. Similarly, the fact that $\vec{x}\in\mathcal{X}_{\epsilon}$ yields that $x_{j}\not \in ]r_{i}-\eta,s_{i}+\eta[$, and then the fact that $d_{\mathcal{U}^{(2)}}(U,U_{n})\leq \eta$ implies that $x_{j}\not \in ]r_{i}^{n},s_{i}^{n}[=U_{i}^{n}\subset U_{\uparrow}^{n}$. We have proven that $x_{i}$ and $x_{j}$ are not in the same connected component of $U_{n}$, and since $x_{i}<x_{j}$ we have $\sigma_{n}^{-1}(i)<\sigma_{n}^{-1}(j)$.

Let us suppose that we are in case $(2)$. Then, $x_{i}$ and $x_{j}$ are in the same connected component $U_{i}\subset U_{\uparrow}$. Since $\vec{x}\in\mathcal{X}_{\epsilon}$, this implies that $x_{i}$ and $x_{j}$ are in $]r_{i}+\eta,s_{i}-\eta[$. Since $d_{\mathcal{U}^{(2)}}(U,U_{n})\leq \eta$, we thus have that $x_{i}$ and $x_{j}$ are in $]r_{i}^{n},s_{i}^{n}[=U_{i}^{n}$. Thus, we have also $\sigma_{n}^{-1}(i)<\sigma_{n}^{-1}(j)$. The case $\sigma^{-1}(i)>\sigma^{-1}(j)$ is done similarly.

Finally, $\sigma=\sigma_{n}$ and we have proven that for $n\geq n_{0}$, $\sigma_{U_{n}}$ and $\sigma_{U}$ are equal on $\mathcal{X}_{\epsilon}$. 
\end{proof}

\begin{proof}[Proof of \cref{convergencePaintbox}]
Let $\epsilon>0$. Let $\mathcal{X}_{\epsilon}\in\mathcal{B}([0,1])$ and $n_{0}\geq 1$ be given by \cref{mainStepBis} for $\epsilon$, and let $\sigma\in S_{k}$. Then,
\begin{equation}\label{conditXepsilon}
\mathbb{P}_{n}(\sigma_{U_{n}}(X_{k}^{n})=\sigma)\geq \mathbb{P}_{n}\big(\lbrace\sigma_{U_{n}}(X_{k}^{n})=\sigma\rbrace \cap \mathcal{X}_{\epsilon}).
\end{equation}
By \cref{mainStepBis}, for $n\geq n_{0}$, $\sigma_{U_{n}}$ and $\sigma_{U}$ coincide on $\mathcal{X}_{\epsilon}$, which yields
\begin{equation}\label{indepN}
\mathbb{P}_{n}\big(\lbrace\sigma_{U_{n}}(X_{k}^{n})=\sigma\rbrace \cap \mathcal{X}_{\epsilon}\big)=\mathbb{P}_{n}\big(\lbrace\sigma_{U}(X_{k}^{n})=\sigma\rbrace \cap \mathcal{X}_{\epsilon})=\mathbb{P}_{n}(A_{\sigma,\epsilon}),
\end{equation}
where $A_{\sigma,\epsilon}=\sigma_{U}^{-1}(\lbrace \sigma\rbrace)\cap \mathcal{X}_{\epsilon}$. Since $A_{\sigma,\epsilon}$ is independent of $n$ and $\mathbb{P}(\partial A_{\sigma,\epsilon})\leq \mathbb{P}(\partial \mathcal{X}_{\epsilon})=0$, there exists $n_{\sigma}\geq n_{0}$ such that for $n\geq n_{\sigma}$,
\begin{equation}\label{converProba}
\mathbb{P}_{n}(A_{\sigma,\epsilon})\geq \mathbb{P}(A_{\sigma,\epsilon})-\epsilon.
\end{equation}
On the one hand, using \eqref{converProba} in \eqref{indepN} yields that 
\begin{equation}\label{minorPnP}
\mathbb{P}_{n}\big(\lbrace\sigma_{U_{n}}(X_{k}^{n})=\sigma\rbrace \cap \mathcal{X}_{\epsilon}\big)\geq \mathbb{P}(A_{\sigma,\epsilon})-\epsilon
\end{equation}
for $n\geq n_{\sigma}$. On the other hand, since $\mathbb{P}(\mathcal{X}_{\epsilon})\geq 1-\epsilon$ and $A_{\sigma,\epsilon}=\sigma_{U}^{-1}(\lbrace \sigma\rbrace)\cap \mathcal{X}_{\epsilon}$, 
$$\mathbb{P}(A_{\sigma,\epsilon})\geq \mathbb{P}(\sigma_{U}(X_{k})=\sigma)-\epsilon.$$
Using the latter inequality in \eqref{minorPnP} yields that
$$\mathbb{P}_{n}\big(\lbrace\sigma_{U_{n}}(X_{k}^{n})=\sigma\rbrace \cap \mathcal{X}_{\epsilon}\big)\geq \mathbb{P}(\sigma_{U}(X_{k})=\sigma)-2\epsilon,$$
which combined with \eqref{conditXepsilon} gives
$$\mathbb{P}_{n}(\sigma_{U_{n}}(X_{k}^{n})=\sigma)\geq \mathbb{P}(\sigma_{U}(X_{k})=\sigma)-2\epsilon.$$
Let $n_{1}=\max_{\sigma\in S_{k}}(n_{\sigma})$. Then, for $n\geq n_{1}$ and $\sigma\in S_{k}$,
\begin{align*}
\mathbb{P}_{n}(\sigma_{U_{n}}(X_{k}^{n})=\sigma)=&1-\sum_{\substack{\sigma'\in S_{k}\\\sigma'\not=\sigma}}\mathbb{P}_{n}(\sigma_{U_{n}}(X_{k}^{n})=\sigma')\\
\leq &1-\sum_{\substack{\sigma'\in S_{k}\\\sigma'\not=\sigma}}\left(\mathbb{P}\big(\sigma_{U}(X_{k})=\sigma'\big)-2\epsilon \right)\\
\leq &\mathbb{P}(\sigma_{U}(X_{k})=\sigma)+2\epsilon k!.
\end{align*}
To summarize, we have found $n_{1}\geq 1$ such that for all $\sigma\in S_{k}$ and $n\geq n_{1}$,
$$\mathbb{P}(\sigma_{U}(X_{k})=\sigma)-2\epsilon\leq \mathbb{P}_{n}(\sigma_{U_{n}}(X_{k}^{n})=\sigma)\leq \mathbb{P}(\sigma_{U}(X_{k})=\sigma)+2\epsilon k!.$$
This implies the convergence in law of $\sigma_{U_{n}}(X_{k}^{n})$ towards $\sigma_{U}(X_{k})$.
\end{proof}

\bibliographystyle{imsart-number}
\bibliography{descent2}
\end{document}